\documentclass[11pt, reqno]{amsart}
\usepackage{amsmath, amsthm, amscd, amsfonts, amssymb, graphicx, color, mathtools}
\usepackage[bookmarksnumbered, colorlinks, plainpages]{hyperref}
\usepackage{enumerate}
\usepackage{setspace}
\usepackage{multicol}
\usepackage[margin=2cm]{geometry}
\usepackage{comment}
\usepackage{booktabs}
\usepackage{caption}
\usepackage{subcaption}
\usepackage{pgfplots}

\theoremstyle{definition}
\newtheorem{theorem}{Theorem}[section]
\newtheorem{lemma}[theorem]{Lemma}
\newtheorem{proposition}[theorem]{Proposition}

\newtheorem{definition}[theorem]{Definition}

\newtheorem{remark}[theorem]{Remark}
\numberwithin{equation}{section}

\newcommand{\cal}{\mathcal}
\newcommand{\bff}{\boldsymbol}
\newcommand{\bb}{\mathbb}

\newcommand{\dt}{\mathrm{d}t}
\newcommand{\ddt}{\frac{\mathrm{d}}{\mathrm{d}t}}
\newcommand{\dx}{\mathrm{d}x}
\newcommand{\ds}{\mathrm{d}s}

\newcommand{\dtau}{\mathrm{d}\tau}
\newcommand{\norm}[2]{\left\|{#1}\right\|_{#2}}
\newcommand{\inpro}[2]{\left\langle#1,#2\right\rangle}
\newcommand{\abs}[1]{\left|{#1}\right|}
\newcommand{\umin}{\underline{\min} \ }

\allowdisplaybreaks

\begin{document}
\setcounter{page}{1}

\title[Stable $C^1$-conforming FEM for a class of nonlinear fourth-order evolution equations]
{Stable $C^1$-conforming finite element methods for a class of nonlinear fourth-order evolution equations}

\author[Agus L. Soenjaya]{Agus L. Soenjaya}
\address{School of Mathematics and Statistics, The University of New South Wales, Sydney 2052, Australia}
\email{\textcolor[rgb]{0.00,0.00,0.84}{a.soenjaya@unsw.edu.au}}

\author[Thanh Tran]{Thanh Tran}
\address{School of Mathematics and Statistics, The University of New South Wales, Sydney 2052, Australia}
\email{\textcolor[rgb]{0.00,0.00,0.84}{thanh.tran@unsw.edu.au}}

\date{\today}


\begin{abstract}
We propose some finite element schemes to solve a class of fourth-order nonlinear PDEs, which include the vector-valued Landau--Lifshitz--Baryakhtar equation, the Swift--Hohenberg equation, and various Cahn--Hilliard-type equations with source and convection terms, among others. The proposed numerical methods include a spatially semi-discrete scheme and two linearised fully-discrete $C^1$-conforming schemes utilising a semi-implicit Euler method and a semi-implicit BDF method. We show that these numerical schemes are stable in $\bb{H}^2$. Error analysis is performed which shows optimal convergence rates in
each scheme. Numerical experiments corroborate our theoretical results.
\end{abstract}
\maketitle

\section{Introduction}
This paper aims to develop and analyse some fully-discrete $C^1$-conforming finite element methods to solve a class of nonlinear fourth-order evolution PDEs with diffusion, convection, as well as nonlinear precession and source terms. The proposed numerical methods would be applicable to some concrete nonlinear models arising in micromagnetics, materials science, and population dynamics. These include the vector-valued Landau--Lifshitz--Baryakhtar equation, the scalar-valued Swift--Hohenberg equation, and various convective Cahn--Hilliard-type equations with source term, among others.

We now describe the problem which is discussed in this paper. Let $\Omega\subset \bb{R}^d$, where $d\in\{1,2,3\}$, be a bounded open domain with boundary $\partial \Omega$, and let $t\in [0,T]$ be the time variable. Let $\bff{u}(t):\Omega \subset \bb{R}^d \to \bb{R}^m$, where $m=1$ or $3$, be the unknown (scalar-valued or vector-valued) function which satisfies
\begin{equation}\label{equ:llbar}
	\begin{cases}
		\displaystyle 
		\frac{\partial\bff{u}}{\partial t} 
		- \beta_1 \Delta \bff{u} 
		+ \beta_2 \Delta^2 \bff{u} 
		= 
		\beta_3 (1-|\bff{u}|^2) \bff{u} 
		- \beta_4 \bff{u}\times \Delta \bff{u} 
		\quad &\, 
		\\[1ex]
		\hspace{117pt}
		+
		\beta_5 \Delta(|\bff{u}|^2 \bff{u})
		+
		\beta_6 (\bff{j}\cdot\nabla)\bff{u}
		\quad &\text{for $(t,\bff{x})\in(0,T)\times\Omega$,}
		\\[1ex]
		\bff{u}(0,\bff{x})= \bff{u_0}(\bff{x}) 
		\quad & \text{for } \bff{x}\in \Omega,
		\\[1ex]
		\displaystyle{
			\frac{\partial \bff{u}}{\partial \bff{n}}= \bff{0}}, 
		\;\displaystyle{\frac{\partial (\Delta\bff{u})}{\partial \bff{n}}= \bff{0}} 
		\quad & \text{for } (t,\bff{x})\in (0,T) \times \partial \Omega.
	\end{cases}
\end{equation}
Here, $\beta_2$ and $\beta_3$ are positive constants of physical significance, while the constants $\beta_4,\beta_5, \beta_6$ are non-negative. The constant $\beta_1$ can be positive or negative, but without loss of generality it is taken as positive here. The symbol $\times$ denotes the cross product. We set $\beta_4=0$ when $\bff{u}$ is scalar-valued. The vector field $\bff{j}:\Omega\to \bb{R}^d$ is a given current density. For physical reason, we assume that the current density $\bff{j}$ is divergence-free and is tangential to the boundary, i.e. $\nabla\cdot \bff{j}=0$ on $\Omega$ and $\bff{j}\cdot \bff{n}=0$ on $\partial\Omega$. The Neumann boundary condition is set for the above problem, but similar argument will also work for the Dirichlet boundary condition.

In~\eqref{equ:llbar}, $-\Delta \bff{u}$ and
$\Delta^2 \bff{u}$ represent (homogeneous) linear diffusion, $\Delta(|\bff{u}|^2
\bff{u})$ corresponds to nonlinear diffusion, $\bff{u} \times \Delta \bff{u}$
represents a form of precession, while the term~$(1-|\bff{u}|^2) \bff{u}$ describes interaction between the source and cubic absorption terms. The term $(\bff{j}\cdot \nabla)\bff{u}$ is a convective term arising from a given current density $\bff{j}$. Some combinations of these features are widespread across various systems of interest in physics, chemistry, and biology, which we elaborate in the following paragraphs.

When $m=3$, i.e. $\bff{u}$ is vector-valued, the problem~\eqref{equ:llbar} includes the following:
\begin{enumerate}
	\renewcommand{\labelenumi}{\theenumi}
	\renewcommand{\theenumi}{{\rm (\roman{enumi})}}
	\item The case $\beta_6=0$ is the Landau--Lifshitz--Baryakhtar equation in the theory of micromagnetics (see~\cite{Bar84, BarDan13, SoeTra23, SoeTra23b, WanDvo15} for physical motivations). When $\beta_6\neq 0$, this is the model with spin-torque considered in~\cite{YasFasIvaMak22}.
	\item The case $\beta_5=\beta_6=0$ is a regularisation of the Landau--Lifshitz--Bloch equation, whose stochastic version is analysed in~\cite{GolJiaLe24}.
\end{enumerate}
When $m=1$ and $\beta_4=0$, i.e. $\bff{u}$ is scalar-valued, the present model includes the following:
\begin{enumerate}
	\renewcommand{\labelenumi}{\theenumi}
	\renewcommand{\theenumi}{{\rm (\roman{enumi})}}
	\item The case $\beta_1<0$ and $\beta_5=0$ gives the Swift--Hohenberg equation~\cite{SwiHoh77}, which describes the effects of thermal fluctuations on the convective instability~\cite{CroHoh93}. A system of such equations is also considered in~\cite{BecFreNie18}.
	\item The case $\beta_1\geq 0$ and $\beta_5=0$ gives a fourth-order reaction-diffusion-convection equation, which could be used to model anomalous bi-flux diffusion \cite{BevGalSimRio13, Chu18, JiaBevZhu20}.
	\item The case $\beta_3=0$ gives the Cahn--Hilliard equation with convection~\cite{CahHil58, EdeKal07, KayStySul09}, which is used to study phase separation in a mixture. A system of such equations is also considered in~\cite{ChoDlo94}. The case $\beta_3\neq 0$ gives the Cahn--Hilliard equation with mass source~\cite{Fak17, KhaSan08} or the Allen--Cahn/Cahn--Hilliard equation~\cite{KarNag14}.
	\item The case $\beta_6=0$ gives a generalised diffusion model for growth and dispersal in a multi-species population \cite{CohMur81, Och84}.
\end{enumerate}

The wide applicability of \eqref{equ:llbar} to model various physical phenomena motivates the numerical analysis of the system as done in the present paper.
To the best of our knowledge, numerical scheme to approximate
the solution of the problem~\eqref{equ:llbar} in its generality do not exist yet in the literature. While numerical methods to separately solve the scalar-valued Swift--Hohenberg equation~\cite{Qi24, QiHou21}, the fourth order bi-flux reaction diffusion equation~\cite{JiaBevZhu20}, the extended Fisher--Kolmogorov equation~\cite{IsmAnkOmr24}, and the Cahn--Hilliard equation~\cite{CheFenWanWis19, EllFre87, KayStySul09} are present in the literature, none of these is sufficiently general to include~\eqref{equ:llbar}. Related to this, a conforming method for a stochastic counterpart of~\eqref{equ:llbar} with $\beta_5=\beta_6=0$ is analysed in~\cite{GolJiaLe24}, where convergence in probability of the scheme is shown. For the general case considered in the present paper (where all the coefficients of the equation~\eqref{equ:llbar} are nonzero), more careful analysis needs to be done to obtain an optimal order of convergence.

Since~\eqref{equ:llbar} is a fourth-order equation, a conforming finite element method
requires $C^1$-continuous elements. In this paper, we propose two fully discrete
$C^1$-conforming finite element schemes to solve \eqref{equ:llbar} for $d=1,2,3$, namely, a semi-implicit Euler scheme and a semi-implicit BDF scheme. 
While the case $d=3$ is also analysed in this paper for completeness, we admit that it is very restrictive in practice since a $C^1$-conforming finite element on tetrahedron requires piecewise polynomials of degree at least 9, thus computationally infeasible. Therefore for $d=3$, non-conforming methods should be used instead. A mixed finite element method specifically developed to solve only the Landau--Lifshitz--Baryakhtar equation in mixed formulation is studied in a separate paper~\cite{Soe24}, where a certain nonlinear scheme is proposed.

It is noted that the schemes proposed in this paper are linear, in that one only needs to solve a system of linear equations at each time step. This is an advantage, for instance in micromagnetic simulations as argued by \cite{SunCheDuWan23}. These methods are unconditionally stable in $\bb{H}^2$ for $d=1$ and $2$ (in the sense that no restriction of the time-step size in terms of the spatial mesh size is needed, and without assuming quasi-uniformity of the triangulation), and conditionally stable for $d=3$. We then prove rigorously that the numerical schemes converge to the exact solution in $\bb{H}^\beta$ norm (for $\beta=0$, $1$, $2$) at the optimal order. This is done by developing an appropriate elliptic projection which is adapted to the problem and performing careful analysis to estimate the nonlinear terms. In particular, these linearised finite element schemes offer simple methods to solve various nonlinear vector-valued or scalar-valued problems mentioned previously.

While such conforming methods could be rather expensive computationally
due to the $C^1$-continuity requirement across element boundaries, this feature
leads to an enhanced accuracy and a higher order of convergence in capturing information about solutions and their gradients compared to the non-conforming methods~\cite{ZhaTon13}. In general, one also achieves a better stability property compared to the non-conforming approximations. With a tremendous growth in computational power during the last decade, it is now relatively easier to implement a $C^1$-conforming finite element to approximate the solution of a fourth-order nonlinear PDE in 1D or 2D, with an added advantage that these proposed schemes are linear. Without trying to be exhaustive, we mention several recent papers which discuss the analysis and implementation of $C^1$-conforming finite element schemes to approximate the solution to fourth-order nonlinear equations, including the KdV--Rosenau equation~\cite{AnkJiwKum23}, the Rosenau-RLW equation~\cite{AtoOmr13}, the extended Fisher--Kolmogorov equation~\cite{IsmAnkOmr24}, a fourth-order optimal control problem \cite{BreSun17}, quasi-geostrophic equation \cite{FosIliWan13, FosIliWel16}, and the von {K}\'{a}rm\'{a}n equation \cite{MalNat16}. However, the problems discussed in those papers are scalar-valued and have different nonlinearities from ours.

This paper is organised as follows. In Section \ref{sec:formula}, some notations
and preliminary results used in this paper are discussed. In Section
\ref{sec:semidiscrete}, error estimates for the semi-discrete approximation are
established. Section \ref{sec:semi euler} and \ref{sec:bdf2} discuss fully
discrete schemes based on semi-implicit Euler and BDF2 methods, respectively. Numerical results 
are presented in Section \ref{sec:numeric exp}. Finally, a lemma concerning a regularity result for an auxiliary equation is discussed in Section \ref{sec:lem aux H4}, while some commonly used results are presented in the appendix. For ease of presentation, we will assume that $\bff{u}$ is vector-valued in the analysis ($m=3$), noting that the same argument will also work when $\bff{u}$ is scalar-valued.

\section{Preliminaries}\label{sec:formula}

\subsection{Notations}

We begin by stating some function spaces used in this paper. The function space~$\bb{L}^p := \bb{L}^p(\Omega; \bb{R}^3)$ denotes the space of $p$-th integrable functions taking values in $\bb{R}^3$ and $\bb{W}^{m,r} := \bb{W}^{m,r}(\Omega; \bb{R}^3)$ denotes the Sobolev space of 
functions on $\Omega \subset \bb{R}^d$, for $d=1,2,3$, taking values in $\bb{R}^3$. As usual, we write $\bb{H}^m := \bb{W}^{m,2}$.
If $X$ is a Banach space, the spaces $L^p(0,T;X)$ and $W^{m,r}(0,T;X)$ denote respectively the Lebesgue and Sobolev spaces of functions on $(0,T)$ taking values in $X$. The space $C([0,T];X)$ denotes the space of continuous function on $[0,T]$ taking values in $X$. For brevity, we will write $L^p(\bb{W}^{m,r}) := L^p(0,T; \bb{W}^{m,r})$ and $L^p(\bb{L}^q) := L^p(0,T; \bb{L}^q)$.

Throughout this paper, we denote the scalar product in a Hilbert space $H$ by $\inpro{ \cdot}{ \cdot}_H$ and its corresponding norm by $\norm{\cdot}{H}$. The scalar product of $\bb{L}^2$ vector-valued functions taking values in $\bb{R}^3$ and the scalar product of $\bb{L}^2$ matrix-valued functions taking values in $\bb{R}^{3\times 3}$ will both be denoted by $\inpro{\cdot}{\cdot}_{\bb{L}^2}$.

We also define the function spaces
\begin{equation}\label{equ:X12}
	\bb{H}_{\bff{n}}^2 := \left\{ \bff{v} \in \bb{H}^2 \, : \, \displaystyle
	\frac{\partial\bff{v}}{\partial\bff{n}} = 0 \right\}
	\quad\text{and}\quad
	\bb{H}_{\bff{n}}^4 := \left\{ \bff{v} \in \bb{H}^4 \, : \, \displaystyle
	\frac{\partial\bff{v}}{\partial\bff{n}} = 
	\frac{\partial(\Delta\bff{v})}{\partial\bff{n}} = 0 \right\}.
\end{equation}
We assume sufficient conditions on the smoothness of the boundary of $\Omega$
such that the following hold:
\begin{alignat}{2}
	\label{equ:equivnorm-h2}
	\norm{\bff{v}}{\bb{H}^2}
	&\lesssim 
	\norm{\bff{v}}{\bb{L}^2} + \norm{\Delta \bff{v}}{\bb{L}^2},  
	&\qquad & \bff{v}\in\bb{H}_{\bff{n}}^2,
	\\
	\label{equ:equivnorm-h3}
	\norm{\bff{v}}{\bb{H}^3}
	&\lesssim 
	\norm{\bff{v}}{\bb{L}^2} 
	+ \norm{\Delta \bff{v}}{\bb{L}^2}
	+ \norm{\nabla \Delta \bff{v}}{\bb{L}^2},  
	&\qquad & \bff{v}\in\bb{H}_{\bff{n}}^2 \cap \bb{H}^3,
	\\
	\label{equ:equivnorm-h4}
	\norm{\bff{v}}{\bb{H}^4} 
	&\lesssim 
	\norm{\bff{v}}{\bb{L}^2} + \norm{\Delta \bff{v}}{\bb{L}^2} + \norm{\Delta^2 \bff{v}}{\bb{L}^2},
	&\qquad & \bff{v}\in\bb{H}_{\bff{n}}^4.
\end{alignat}
The above are true for $C^4$-domains; see e.g. \cite[Lemma~3.3]{SoeTra23}.
For domains with less smooth boundaries, it
is known, for instance, that the $\bb{H}^2$-regularity
result~\eqref{equ:equivnorm-h2} is guaranteed to
hold for any convex Lipschitz domains \cite[Section~3.2]{Gri11}. The
$\bb{H}^3$-regularity result~\eqref{equ:equivnorm-h3} holds for convex polygonal
domains, see \cite{BluRan80}, \cite[Section~5.1]{Gri11}. The
$\bb{H}^4$-regularity result~\eqref{equ:equivnorm-h4} can be guaranteed for convex polygonal domains satisfying certain interior angle conditions; see \cite{BluRan80, Dau92},
\cite[Section~7.3]{Gri11} and references therein for more details.

Finally, the constant $C$ denotes a
generic constant which can take different values at different occurrences. If
the dependence of $C$ on some variable, e.g.~$T$, is highlighted, we will write
$C(T)$. The notation $A \lesssim B$ means $A \le C B$, where the specific form of
the constant $C$ is not important to clarify.

\subsection{Elementary results}
In this subsection, we state a few elementary results which will be frequently
used.

\begin{lemma}\label{lem:der u2u}
	For any vector-valued function $\bff{v}:\Omega\to\bb{R}^3$, we have
	\begin{align}
		\label{equ:j nab u}
		(\bff{j}\cdot\nabla)\bff{v} 
		&= 
		\nabla\cdot (\bff{v}\otimes \bff{j}) - (\nabla \cdot\bff{j}) \bff{v},
		\\
		\nabla (|\bff{v}|^2 \bff{v}) 
		&= 
		2 \bff{v} \ (\bff{v}\cdot \nabla\bff{v}) 
		+ |\bff{v}|^2 \nabla \bff{v},
		\label{equ:nab un2}
		\\
		\frac{\partial\big(|\bff{v}|^2\bff{v}\big)}{\partial\bff{n}}
		&=
		2 \bff{v} \Big(\bff{v}\cdot \frac{\partial\bff{v}}{\partial\bff{n}}\Big)
		+
		|\bff{v}|^2 \frac{\partial\bff{v}}{\partial\bff{n}},
		\label{equ:nor der v2v}
	\end{align}
	provided that the partial derivatives are well defined.
\end{lemma}
\begin{proof}
See \cite[Lemma~3.1]{SoeTra23}.
\end{proof}

We introduce, for any~$a_1,a_2,\ldots,a_m\in\bb{R}^+$, the notation
\[
	\umin(a_1,a_2,\ldots,a_m) := \min\{a_1,a_2,\ldots,a_m\}+ \sum_{i=1}^m a_i.
\]
We denote by $\odot$ either the dot product or cross product, and by
$\circledast$ either the scalar-vector product, dot product, or cross
product, as long as the expression is meaningful. For example,
\begin{equation}\label{equ:odo cir}
	(\bff{u} \odot \bff{v}) \circledast \bff{w}
	=
	(\bff{u} \cdot \bff{v}) \bff{w}
	\quad\text{or}\quad
	(\bff{u} \times \bff{v}) \cdot \bff{w}
	\quad\text{or}\quad
	(\bff{u} \times \bff{v}) \times \bff{w}.
\end{equation}
For simplification of notation, we write $ \bff{u}\odot\bff{v}\in\bb{W}^{1,p} $
to indicate both $ \bff{u}\cdot\bff{v}\in W^{1,p} $
and $ \bff{u}\times\bff{v}\in\bb{W}^{1,p} $.

We have the following results on the estimates for various products of vector-valued functions which will be useful in subsequent analysis.

\begin{lemma}\label{lem:equivnorm}
	Let $\bff{u},\bff{v}, \bff{w}: \Omega \subset \bb{R}^d \to \bb{R}^3$ and
	$\varphi: \Omega \subset \bb{R}^d \to \bb{R}$, where $d=1,2,3$. We denote by $\odot$
	either the dot product or cross product.
	\begin{enumerate}
		\item For any $\bff{v} \in \bb{H}_{\bff{n}}^2$ and any $\epsilon>0$,
		\begin{align}
			\label{equ:equivnorm-nabL2}
			\norm{\nabla \bff{v}}{\bb{L}^2}^2 
			&\leq  
			\norm{\bff{v}}{\bb{L}^2} 
			\norm{\Delta \bff{v}}{\bb{L}^2}^, 
			\\
			\label{equ:equivnorm-nabL2 young}
			\norm{\nabla \bff{v}}{\bb{L}^2}^2 
			&\leq  
			\frac{1}{4\epsilon} \norm{\bff{v}}{\bb{L}^2}^2
			+ 
			\epsilon \norm{\Delta \bff{v}}{\bb{L}^2}^2,
		\end{align}
		\item
		For $ p, q_j, r_j \in[1,\infty] $ satisfying $1/q_j + 1/r_j =
		1/p$, $j=1,2$, if $\varphi \in L^{q_1} \cap W^{1,q_2}$,
		$\bff{v} \in \bb{W}^{1,r_1} \cap \bb{L}^{r_2}$,
		and $\bff{w} \in \bb{W}^{2,r_1} \cap \bb{W}^{1,r_2}$, then
		$ \varphi\bff{v}\in\bb{W}^{1,p} $, 
		$ \varphi\nabla\bff{w}\in\bb{W}^{1,p} $, and
		\begin{align}\label{equ:prod sobolev scal vec}
			\norm{\varphi \bff{v}}{\bb{W}^{1,p}} 
			&\leq
			\umin
			\big(
			\norm{\varphi}{L^{q_1}} \norm{\bff{v}}{\bb{W}^{1,r_1}},
			\norm{\varphi}{W^{1,q_2}} \norm{\bff{v}}{\bb{L}^{r_2}}
			\big),
			\\
			\label{equ:prod sobolev scal mat}
			\norm{\varphi \nabla \bff{w}}{\bb{W}^{1,p}} 
			&\leq
			\umin
			\big(
			\norm{\varphi}{L^{q_1}} \norm{\bff{w}}{\bb{W}^{2,r_1}},
			\norm{\varphi}{W^{1,q_2}} \norm{\bff{w}}{\bb{W}^{1,r_2}}
			\big).
		\end{align}
		\item 
		For $ p, q_j, r_j \in[1,\infty] $ satisfying $1/q_j + 1/r_j =
		1/p$, $j=1,2$, if $\bff{u} \in \bb{L}^{q_1} \cap \bb{W}^{1,q_2}$,
		$\bff{v} \in \bb{W}^{1,r_1} \cap \bb{L}^{r_2}$,
		and $\bff{w} \in \bb{W}^{2,r_1} \cap \bb{W}^{1,r_2}$, then
		$ \bff{u}\odot\bff{v}\in\bb{W}^{1,p} $,
		$ \bff{u}\odot\nabla\bff{w}\in\bb{W}^{1,p} $, and
		\begin{align}
			\label{equ:prod sobolev vec dot}
			\norm{\bff{u} \odot \bff{v}}{\bb{W}^{1,p}}
			&\leq
			\umin
			\big(
			\norm{\bff{u}}{\bb{L}^{q_1}}
			\norm{\bff{v}}{\bb{W}^{1,r_1}},
			\norm{\bff{u}}{\bb{W}^{1,q_2}} \norm{\bff{v}}{\bb{L}^{r_2}}
			\big),
			\\
			\label{equ:prod sobolev mat dot}
			\norm{\bff{u}\odot \nabla\bff{w}}{\bb{W}^{1,p}}
			&\leq
			\umin
			\big(
			\norm{\bff{u}}{\bb{L}^{q_1}}
			\norm{\bff{w}}{\bb{W}^{2,r_1}},
			\norm{\bff{u}}{\bb{W}^{1,q_2}} 
			\norm{\bff{w}}{\bb{W}^{1,r_2}}
			\big).
		\end{align}
		\item 
		For $p, q_j,r_j,s_j \in [1,\infty]$
		satisfying
		$1/q_j+1/r_j+1/s_j=1/p$, $j=1,2,3$,
		if $\bff{u} \in \bb{W}^{1,q_1} \cap \bb{L}^{q_2} \cap \bb{L}^{q_3}$, 
		$\bff{v} \in \bb{L}^{r_1} \cap \bb{W}^{1,r_2} \cap \bb{L}^{r_3}$, 
		and 
		$\bff{w} \in \bb{L}^{s_1} \cap \bb{L}^{s_2} \cap \bb{W}^{1,s_3}  $,
		then $ (\bff{u}\odot \bff{v}) \circledast \bff{w} \in \bb{W}^{1,p} $
		and
		\begin{align}\label{equ:prod 3}
		\nonumber
		\norm{(\bff{u}\odot \bff{v}) \circledast \bff{w}}{\bb{W}^{1,p}} 
		&\leq
		\umin
		\Big(
		\norm{\bff{u}}{\bb{W}^{1,q_1}}
		\norm{\bff{v}}{\bb{L}^{r_1}}
		\norm{\bff{w}}{\bb{L}^{s_1}},
		\\
		&\qquad \qquad
		\norm{\bff{u}}{\bb{L}^{q_2}}
		\norm{\bff{v}}{\bb{W}^{1,r_2}}
		\norm{\bff{w}}{\bb{L}^{s_2}},
		\norm{\bff{u}}{\bb{L}^{q_3}}
		\norm{\bff{v}}{\bb{L}^{r_3}}
		\norm{\bff{w}}{\bb{W}^{1,s_3}} \Big).
		\end{align}
		\item 
		For $p, q_j,r_j,s_j \in [1,\infty]$
		satisfying
		$1/q_j+1/r_j+1/s_j=1/p$, $j=1,2,3$,
		if 
		$\bff{u} \in \bb{W}^{1,q_1} \cap \bb{L}^{q_2} \cap \bb{L}^{q_3}$, 
		$\bff{v} \in \bb{L}^{r_1} \cap \bb{W}^{1,r_2} \cap \bb{L}^{r_3}$, 
		and 
		$\bff{w} \in \bb{W}^{1,s_1} \cap \bb{W}^{1,s_2} \cap \bb{W}^{2,s_3}$,
		then $ (\bff{u}\odot \bff{v}) \circledast \nabla \bff{w} \in \bb{W}^{1,p} $,
		$\bff{u} (\bff{v}\cdot \nabla \bff{w}) \in \bb{W}^{1,p}$,
		and
		\begin{align}
		\label{equ:prod 3 vec dot vec mat}
		\nonumber
		\norm{(\bff{u}\odot \bff{v}) \circledast \nabla \bff{w}}{\bb{W}^{1,p}} 
		&\leq
		\umin
		\Big(
		\norm{\bff{u}}{\bb{W}^{1,q_1}}
		\norm{\bff{v}}{\bb{L}^{r_1}}
		\norm{\bff{w}}{\bb{W}^{1,s_1}},
		\\
		&\qquad \qquad
		\norm{\bff{u}}{\bb{L}^{q_2}}
		\norm{\bff{v}}{\bb{W}^{1,r_2}}
		\norm{\bff{w}}{\bb{W}^{1,s_2}},
		\norm{\bff{u}}{\bb{L}^{q_3}}
		\norm{\bff{v}}{\bb{L}^{r_3}}
		\norm{\bff{w}}{\bb{W}^{2,s_3}} \Big).
		\\
		\label{equ:prod 3 vec dot mat}
		\nonumber
		\norm{\bff{u} (\bff{v}\cdot \nabla \bff{w})}{\bb{W}^{1,p}}
		&\leq
		\umin
		\Big(
		\norm{\bff{u}}{\bb{W}^{1,q_1}}
		\norm{\bff{v}}{\bb{L}^{r_1}}
		\norm{\bff{w}}{\bb{W}^{1,s_1}},
		\\
		&\qquad \qquad
		\norm{\bff{u}}{\bb{L}^{q_2}}
		\norm{\bff{v}}{\bb{W}^{1,r_2}}
		\norm{\bff{w}}{\bb{W}^{1,s_2}},
		\norm{\bff{u}}{\bb{L}^{q_3}}
		\norm{\bff{v}}{\bb{L}^{r_3}}
		\norm{\bff{w}}{\bb{W}^{2,s_3}} \Big).
		\end{align}
		\item For any $\varphi\in H^k$, $\bff{v},\bff{w} \in \bb{H}^k$, where $k>d/2$,
		\begin{align}\label{equ:prod Hs scal vec}
			\norm{\varphi \bff{w}}{\bb{H}^k}
			&\leq
			C \norm{\varphi}{H^k}
			\norm{\bff{w}}{\bb{H}^k},
			\\
			\label{equ:prod Hs mat dot}
			\norm{\bff{v} \odot \bff{w}}{\bb{H}^k}
			&\leq
			C \norm{\bff{v}}{\bb{H}^k}
			\norm{\bff{w}}{\bb{H}^k},
			\\
			\label{equ:prod Hs triple}
			\norm{(\bff{u} \times \bff{v}) \odot \bff{w}}{\bb{H}^k}
			&\leq
			C \norm{\bff{u}}{\bb{H}^k} \norm{\bff{v}}{\bb{H}^k} \norm{\bff{w}}{\bb{H}^k},
		\end{align}
		where $C=C(\Omega, k, d)$, but is independent of $\bff{u}, \bff{v}, \bff{w}$.
		\item For any $\varphi\in H^k$, $\bff{v}\in \bb{H}^k$, $\bff{w} \in \bb{H}^{k+1}$, where $k>d/2$,
		\begin{align}\label{equ:prod Hs scal mat}
			\norm{\varphi \nabla\bff{w}}{\bb{H}^k}
			&\leq
			C \norm{\varphi}{H^k}
			\norm{\bff{w}}{\bb{H}^{k+1}},
			\\
			\label{equ:prod Hs vec dot mat}
			\norm{\bff{v} \odot \nabla \bff{w}}{\bb{H}^k}
			&\leq
			C \norm{\bff{v}}{\bb{H}^k}
			\norm{\bff{w}}{\bb{H}^{k+1}},
			\\
			\label{equ:prod Hs triple mat}
			\norm{(\bff{u} \times \bff{v}) \odot \nabla \bff{w}}{\bb{H}^k}
			&\leq
			C\norm{\bff{u}}{\bb{H}^k} \norm{\bff{v}}{\bb{H}^k} \norm{\bff{w}}{\bb{H}^{k+1}},
		\end{align}
		where $C=C(\Omega, k,d)$, but is independent of $\bff{u},\bff{v},\bff{w}$.
	\end{enumerate}
\end{lemma}
\begin{proof}
	The first inequality \eqref{equ:equivnorm-nabL2} is a direct consequence
	of integration by parts and H\"older's inequality. The inequality
	\eqref{equ:equivnorm-nabL2 young} is a result of Young's inequality.
	
	To prove~\eqref{equ:prod sobolev scal vec}, we first deduce from
	H\"older's inequality
	\begin{align*}
		\norm{\varphi \bff{v}}{\bb{W}^{1,p}}
		&\le
		\norm{\varphi \bff{v}}{\bb{L}^p}
		+
		\norm{\nabla(\varphi \bff{v})}{\bb{L}^p}
		\leq 
		\norm{\varphi}{L^{q_1}} \norm{\bff{v}}{\bb{L}^{r_1}}
		+
		\norm{\varphi}{L^{q_1}} \norm{\bff{v}}{\bb{W}^{1,r_1}}
		+
		\norm{\varphi}{W^{1,q_2}} \norm{\bff{v}}{\bb{L}^{r_2}}
		\\
		&\le
		2\norm{\varphi}{L^{q_1}} \norm{\bff{v}}{\bb{W}^{1,r_1}}
		+
		\norm{\varphi}{W^{1,q_2}} \norm{\bff{v}}{\bb{L}^{r_2}}.
	\end{align*}
	By symmetry, we obtain \eqref{equ:prod sobolev scal vec}.
	Inequalities~\eqref{equ:prod sobolev scal mat}, 
	\eqref{equ:prod sobolev vec dot}, and~\eqref{equ:prod sobolev mat dot}
	can be proved in exactly the same manner.

	
	Next, we prove \eqref{equ:prod 3}. By H\"older's inequality,
	\begin{align*}
		&\norm{(\bff{u}\odot \bff{v}) \circledast \bff{w}}{\bb{W}^{1,p}} 
		\\
		&\leq
		\norm{(\bff{u}\odot \bff{v}) \circledast \bff{w}}{\bb{L}^p}
		+
		\norm{\nabla\big((\bff{u}\odot \bff{v}) \circledast \bff{w}\big)}{\bb{L}^p}
		\\
		&\leq
		\norm{\bff{u}}{\bb{L}^{q_1}} \norm{\bff{v}}{\bb{L}^{r_1}} \norm{\bff{w}}{\bb{L}^{s_1}}
		+
		\norm{\bff{u}}{\bb{W}^{1,q_1}} \norm{\bff{v}}{\bb{L}^{r_1}} \norm{\bff{w}}{\bb{L}^{s_1}}
		+
		\norm{\bff{u}}{\bb{L}^{q_2}} \norm{\bff{v}}{\bb{W}^{1,r_2}} \norm{\bff{w}}{\bb{L}^{s_2}}
		+
		\norm{\bff{u}}{\bb{L}^{q_3}} \norm{\bff{v}}{\bb{L}^{r_3}} \norm{\bff{w}}{\bb{W}^{1,s_3}}
		\\
		&\leq 
		2 \norm{\bff{u}}{\bb{W}^{1,q_1}} \norm{\bff{v}}{\bb{L}^{r_1}} \norm{\bff{w}}{\bb{L}^{s_1}}
		+
		\norm{\bff{u}}{\bb{L}^{q_2}} \norm{\bff{v}}{\bb{W}^{1,r_2}} \norm{\bff{w}}{\bb{L}^{s_2}}
		+
		\norm{\bff{u}}{\bb{L}^{q_3}} \norm{\bff{v}}{\bb{L}^{r_3}} \norm{\bff{w}}{\bb{W}^{1,s_3}}.
	\end{align*}
	Inequality \eqref{equ:prod 3} then follows by symmetry.
	The proof of \eqref{equ:prod 3 vec dot vec mat} can be done in
	exactly the same manner.
	
	We will prove \eqref{equ:prod 3 vec dot mat} next. By H\"older's inequality,
	\begin{align*}
		\norm{\bff{u}(\bff{v}\cdot \nabla \bff{w})}{\bb{W}^{1,p}}
		&\leq
		\sum_{i=1}^d \Big( \norm{(\bff{v}\cdot \partial_i \bff{w}) \bff{u}}{\bb{L}^p} 
		+
		\norm{\nabla \big( (\bff{v}\cdot \partial_i \bff{w}) \bff{u} \big)}{\bb{L}^p} \Big)
		\\
		&\leq
		\sum_{i=1}^d \Big( \norm{\bff{u}}{\bb{L}^{q_1}} \norm{\bff{v}}{\bb{L}^{r_1}} \norm{\partial_i \bff{w}}{\bb{L}^{s_1}}
		+
		\norm{\nabla \bff{v}}{\bb{L}^{r_2}} \norm{\partial_i \bff{w}}{\bb{L}^{s_2}} \norm{\bff{u}}{\bb{L}^{q_2}}
		\\
		&\qquad \qquad
		+
		\norm{\bff{v}}{\bb{L}^{r_3}} \norm{\nabla \partial_i \bff{w}}{\bb{L}^{s_3}} \norm{\bff{u}}{\bb{L}^{q_3}}
		+
		\norm{\bff{v}}{\bb{L}^{r_1}} \norm{\partial_i \bff{w}}{\bb{L}^{s_1}} \norm{\nabla \bff{u}}{\bb{L}^{q_1}} \Big)
		\\
		&\leq
		2\norm{\bff{u}}{\bb{W}^{1,q_1}} \norm{\bff{v}}{\bb{L}^{r_1}}
		\norm{\bff{w}}{\bb{W}^{1,s_1}}
		+
		\norm{\bff{u}}{\bb{L}^{q_2}} \norm{\bff{v}}{\bb{W}^{1,r_2}} \norm{\bff{w}}{\bb{W}^{1,s_2}} 
		+
		\norm{\bff{u}}{\bb{L}^{q_3}} \norm{\bff{v}}{\bb{L}^{r_3}} \norm{\bff{w}}{\bb{W}^{2,s_3}}.
	\end{align*}
	Therefore, \eqref{equ:prod 3 vec dot mat} follows by symmetry.
	Next, let $\bff{v}=(v_1,v_2,v_3)$ and $\bff{w}=(w_1,w_2,w_3)$. Then
	\begin{align*}
		\norm{\varphi \bff{w}}{\bb{H}^k} 
		\leq
		\sum_{i=1}^3 \norm{\varphi w_i}{H^k}
		\leq
		\sum_{i=1}^3 C\norm{\varphi}{H^k} \norm{w_i}{H^k}
		\leq
		C \norm{\varphi}{H^k} \norm{\bff{w}}{\bb{H}^k},
	\end{align*}
	where we used the fact that
	\begin{align}\label{equ:prod H2 scalar}
		\norm{\varphi \psi}{H^k} \leq C \norm{\varphi}{H^k} \norm{\psi}{H^k}
	\end{align}
	for scalar-valued functions $\varphi$ and $\psi$, if $k>d/2$~\cite{AdaFou03}. Moreover, using \eqref{equ:prod H2 scalar} again,
	\begin{align*}
		\norm{\bff{v} \times \bff{w}}{\bb{H}^k}
		\leq
		\sum_{i,j=1}^3 \norm{v_i w_j}{H^k}
		\lesssim
		\sum_{i,j=1}^3 \norm{v_i}{H^k} \norm{w_j}{H^k}
		\leq
		\left(\sum_{i=1}^3 \norm{v_i}{H^k} \right) \left( \sum_{j=1}^3 \norm{w_j}{H^k} \right) 
		\leq
		\norm{\bff{v}}{\bb{H}^k} \norm{\bff{w}}{\bb{H}^k},
	\end{align*}
	and similarly we also have $\norm{\bff{v}\cdot \bff{w}}{\bb{H}^k} \leq C \norm{\bff{v}}{\bb{H}^k} \norm{\bff{w}}{\bb{H}^k}$, thus proving \eqref{equ:prod Hs mat dot}. Furthermore, \eqref{equ:prod Hs triple} follows by repeatedly applying \eqref{equ:prod Hs mat dot}.
	
	Next, by \eqref{equ:prod Hs scal vec} we have
	\begin{align*}
		\norm{\varphi \nabla \bff{w}}{\bb{H}^k}
		\leq
		\sum_{i=1}^d \norm{\varphi \partial_i \bff{w}}{\bb{H}^k}
		\leq
		\sum_{i=1}^d C \norm{\varphi}{H^k} \norm{\partial_i \bff{w}}{\bb{H}^k}
		\leq
		C \norm{\varphi}{H^k} \norm{\bff{w}}{\bb{H}^{k+1}},
	\end{align*}
	proving \eqref{equ:prod Hs scal mat}.
	Inequality \eqref{equ:prod Hs vec dot mat} can be shown in the same manner.

	Finally, \eqref{equ:prod Hs triple mat} follows from \eqref{equ:prod Hs mat dot} and \eqref{equ:prod Hs vec dot mat}. This completes the proof of the lemma.
\end{proof}

To facilitate the proof of our error analysis, we need several multilinear forms defined in the following lemmas.

\begin{lemma}
Given $\bff{\phi}, \bff{\eta} \in \bb{H}_{\bff{n}}^2$, let $\cal{B}(\bff{\phi}, \bff{\eta};\,
\cdot \, , \, \cdot) : \bb{H}^1\times \bb{H}^1 \to \bb{R}$ and $\cal{C}(\bff{\eta}; \,\cdot\, ,\,\cdot ): \bb{H}^1\times \bb{H}^1 \to \bb{R}$ be respectively defined by
\begin{align}\label{equ:bilinear B}
	\nonumber
	\cal{B}(\bff{\phi}, \bff{\eta}; \bff{v}, \bff{w}) 
	&:=
	\beta_3 \inpro{ (\bff{\phi} \cdot \bff{\eta}) \bff{v}}{ \bff{w}}_{\bb{L}^2}
	+
	\beta_5 \inpro{ (\bff{\phi} \cdot \bff{\eta}) \nabla \bff{v}}{ \nabla \bff{w}}_{\bb{L}^2}
	\\
	&\quad\;
	+ 
	\beta_5 \inpro{ \bff{\eta} (\bff{\phi}\cdot \nabla \bff{v}) }{ \nabla \bff{w}}_{\bb{L}^2}
	+ 
	\beta_5 \inpro{ \bff{\phi}(\bff{\eta} \cdot \nabla \bff{v}) }{
		\nabla \bff{w} }_{\bb{L}^2}
\end{align}
and
\begin{align}\label{equ:bilinear C}
	\cal{C}(\bff{\eta};\bff{v},\bff{w})
	:=
	-\beta_4 \inpro{\bff{\eta} \times \nabla \bff{v}}{ \nabla \bff{w} }_{\bb{L}^2}.
\end{align}
The following statements hold true:
\begin{enumerate}[(i)]
	\item The map $\cal{B}$ is symmetric with respect to the third and fourth variable.
	\item For all $\bff{v},\bff{w}\in \bb{H}^1$,
\begin{align}
	\label{equ:B bounded H1}
	|\cal{B}(\bff{\phi},\bff{\eta}; \bff{v}, \bff{w})| 
	&\leq
	(\beta_3+3\beta_5) \norm{\bff{\phi}}{\bb{L}^\infty}
	\norm{\bff{\eta}}{\bb{L}^\infty}
	\norm{\bff{v}}{\bb{H}^1} \norm{\bff{w}}{\bb{H}^1},
	\\
	\label{equ:C bounded H1}
	|\cal{C}(\bff{\eta}; \bff{v}, \bff{w})| 
	&\leq
	\beta_4 \norm{\bff{\eta}}{\bb{L}^\infty}
	\norm{\bff{v}}{\bb{H}^1} \norm{\bff{w}}{\bb{H}^1}.
\end{align}
	\item Moreover, if $\bff{v} \in \bb{H}_{\bff{n}}^2$ and $\bff{w} \in \bb{H}^1$, then 
\begin{align}
	\label{equ:B bounded W1,4}
	\nonumber
	|\cal{B}(\bff{\phi}, \bff{\eta}; \bff{v}, \bff{w})| 
	&\lesssim 
	\Big( \norm{\bff{\phi}}{\bb{W}^{1,4}} \norm{\bff{\eta}}{\bb{L}^\infty}
	+
	\norm{\bff{\phi}}{\bb{L}^\infty} \norm{\bff{\eta}}{\bb{W}^{1,4}}
	\Big) 
	\norm{\bff{v}}{\bb{W}^{1,4}} 
	\norm{\bff{w}}{\bb{L}^{2}} 
	\\
	&\quad
	+
	\norm{\bff{\phi}}{\bb{L}^\infty} \norm{\bff{\eta}}{\bb{L}^\infty}
	\norm{\bff{v}}{\bb{H}^2}
	\norm{\bff{w}}{\bb{L}^2}, 
	\\
	\label{equ:B bounded H2}
	|\cal{B}(\bff{\phi}, \bff{\eta}; \bff{v}, \bff{w})| 
	&\lesssim 
	\norm{\bff{\phi}}{\bb{H}^2}
	\norm{\bff{\eta}}{\bb{H}^2}
	\norm{\bff{v}}{\bb{H}^2}
	\norm{\bff{w}}{\bb{L}^2},
\end{align}
and
\begin{align}
	\label{equ:C bounded W1,4}
	|\cal{C}(\bff{\eta}; \bff{v}, \bff{w})| 
	&\lesssim
	\norm{\bff{\eta}}{\bb{W}^{1,4}} \norm{\bff{v}}{\bb{W}^{1,4}} 
	\norm{\bff{w}}{\bb{L}^{2}} 
	+
	 \norm{\bff{\eta}}{\bb{L}^\infty}
	\norm{\bff{v}}{\bb{H}^2}
	\norm{\bff{w}}{\bb{L}^2},
	\\
	\label{equ:C bounded H2}
	|\cal{C}(\bff{\eta}; \bff{v}, \bff{w})| 
	&\lesssim
	\norm{\bff{\eta}}{\bb{H}^2}
	\norm{\bff{v}}{\bb{H}^2}
	\norm{\bff{w}}{\bb{L}^2},
\end{align}
\end{enumerate}
where the constants depend only on $\beta_3,\beta_4,\beta_5$, and $\Omega$.
\end{lemma}

\begin{proof}
We can write
\begin{align*}
	\nonumber
	\cal{B}(\bff{\phi}, \bff{\eta}; \bff{v}, \bff{w}) 
	&=
	\beta_3 \inpro{ (\bff{\phi} \cdot \bff{\eta}) \bff{w}}{ \bff{v}}_{\bb{L}^2}
	+
	\beta_5 \inpro{ (\bff{\phi} \cdot \bff{\eta}) \nabla \bff{w}}{ \nabla \bff{v}}_{\bb{L}^2}
	\\
	&\quad
	+ 
	\beta_5 \inpro{ \bff{\eta} (\bff{\phi}\cdot \nabla \bff{w}) }{ \nabla \bff{v}}_{\bb{L}^2}
	+ 
	\beta_5 \inpro{ \bff{\phi}(\bff{\eta} \cdot \nabla \bff{w}) }{
		\nabla \bff{v} }_{\bb{L}^2}
	\\
	&=
	\cal{B}(\bff{\phi}, \bff{\eta}; \bff{w}, \bff{v}).
\end{align*}
Thus, $\cal{B}$ is symmetric in terms of $\bff{v}$ and $\bff{w}$.
By H\"{o}lder's inequality and the Sobolev embedding $\bb{H}^2 \subset \bb{L}^\infty$, we have
\begin{align*}
	\big| \cal{B}(\bff{\phi}, \bff{\eta}; \bff{v}, \bff{w}) \big|
	&\lesssim
	\norm{\bff{\phi}}{\bb{L}^\infty} \norm{\bff{\eta}}{\bb{L}^\infty}
	\norm{\bff{v}}{\bb{H}^1} \norm{\bff{w}}{\bb{H}^1},
	\\
	\big| \cal{C}(\bff{\phi}, \bff{\eta}; \bff{v}, \bff{w}) \big|
	&\lesssim
	\norm{\bff{\eta}}{\bb{L}^\infty}
	\norm{\bff{v}}{\bb{H}^1} \norm{\bff{w}}{\bb{H}^1},
\end{align*}
showing \eqref{equ:B bounded H1} and \eqref{equ:C bounded H1}, respectively. Next, performing integration by parts (and noting the boundary terms vanish since $\bff{\phi}, \bff{\eta},\bff{v}\in \bb{H}_{\bff{n}}^2$), and applying H\"{o}lder's inequality, we obtain
\begin{align*}
	\cal{B}(\bff{\phi}, \bff{\eta}; \bff{v}, \bff{w}) 
	&=
	\beta_3 \inpro{ (\bff{\phi} \cdot \bff{\eta}) \bff{v}}{ \bff{w}}_{\bb{L}^2}
	-
	\beta_5 \inpro{ \nabla \cdot \big( (\bff{\phi} \cdot \bff{\eta}) \nabla \bff{v} \big)}{ \bff{w}}_{\bb{L}^2}
	\\
	&\quad
	-
	\beta_5 \inpro{\nabla \cdot \big(\bff{\eta} (\bff{\phi}\cdot \nabla \bff{v}) \big)}{ \bff{w}}_{\bb{L}^2}
	- 
	\beta_5 \inpro{\nabla \cdot \big(\bff{\phi}(\bff{\eta} \cdot \nabla \bff{v}) \big)}{
	\bff{w} }_{\bb{L}^2}, \text{ and}
	\\
	\cal{C}(\bff{\eta}; \bff{v}, \bff{w})
	&=
	\beta_4 \inpro{\nabla \cdot(\bff{\eta} \times \nabla \bff{v})}{ \bff{w} }_{\bb{L}^2}.
\end{align*}
We can estimate each term on the right-hand side as follows. For the first term in $\cal{B}$, by H\"{o}lder's inequality and Sobolev embedding,
\begin{align*}
	\big|\beta_3  \inpro{ (\bff{\phi} \cdot \bff{\eta}) \bff{v}}{ \bff{w}}_{\bb{L}^2} \big|
	\lesssim 
	\norm{\bff{\phi}}{\bb{L}^\infty}
	\norm{\bff{\eta}}{\bb{L}^\infty}
	\norm{\bff{v}}{\bb{L}^2}
	\norm{\bff{w}}{\bb{L}^2}.
\end{align*}
For the second term, by \eqref{equ:prod 3 vec dot vec mat}, similarly we have
\begin{align*}
	\big| \beta_5 \inpro{ \nabla \cdot \big( (\bff{\phi} \cdot \bff{\eta}) \nabla \bff{v} \big)}{ \bff{w}}_{\bb{L}^2} \big|
	&\leq
	\norm{(\bff{\phi} \cdot \bff{\eta}) \nabla \bff{v}}{\bb{H}^1} \norm{\bff{w}}{\bb{L}^2}
	\\
	&\lesssim
	\norm{\bff{\phi}}{\bb{L}^\infty} \norm{\bff{\eta}}{\bb{L}^\infty} \norm{\bff{v}}{\bb{H}^2} \norm{\bff{w}}{\bb{L}^2}
	+
	\norm{\bff{\phi}}{\bb{L}^\infty} \norm{\bff{\eta}}{\bb{W}^{1,4}} \norm{\bff{v}}{\bb{W}^{1,4}}
	\norm{\bff{w}}{\bb{L}^2}
	\\
	&\quad
	+ 
	\norm{\bff{\phi}}{\bb{W}^{1,4}} \norm{\bff{\eta}}{\bb{L}^\infty} \norm{\bff{v}}{\bb{W}^{1,4}} 
	\norm{\bff{w}}{\bb{L}^2}.
\end{align*}
The third and the fourth terms in $\cal{B}$ can be estimated in a similar way by using \eqref{equ:prod 3 vec dot mat}.
For the functional $\mathcal{C}$, by H\"{o}lder's inequality and \eqref{equ:prod sobolev mat dot}, we have
\begin{align}\label{equ:beta4 cross}
	\nonumber
	\big| \beta_4 \inpro{\nabla \cdot(\bff{\eta} \times \nabla \bff{v})}{ \bff{w} }_{\bb{L}^2} \big|
	&\leq
	\norm{\bff{\eta} \times \nabla \bff{v}}{\bb{H}^1} \norm{\bff{w}}{\bb{L}^2}
	\\
	&\lesssim
	\big( \norm{\bff{\eta}}{\bb{L}^\infty} \norm{\nabla \bff{v}}{\bb{H}^1}
	+
	\norm{\nabla \bff{v}}{\bb{L}^4} \norm{\bff{\eta}}{\bb{W}^{1,4}} \big)
	\norm{\bff{w}}{\bb{L}^2}.
\end{align}
This concludes the proof of \eqref{equ:B bounded W1,4} and \eqref{equ:C bounded W1,4}. The proof of \eqref{equ:B bounded H2} and~\eqref{equ:C bounded H2} then follows by the Sobolev embedding $\bb{H}^2 \subset \bb{L}^\infty \cap
\bb{W}^{1,4}$, completing the proof of the lemma.
\end{proof}

\begin{lemma}\label{lem:A bou coe}
	Given $\alpha>0$ and $\bff{\phi} \in \bb{H}_{\bff{n}}^2$, let $\cal{A}(\bff{\phi};\,
	\cdot \, , \, \cdot) : \bb{H}_{\bff{n}}^2\times \bb{H}_{\bff{n}}^2 \to \bb{R}$ be defined by
	\begin{align}\label{equ:bilinear}
		\cal{A}(\bff{\phi}; \bff{v}, \bff{w})
		&:= 
		\cal{A}_0(\bff{v},\bff{w})
		+ \cal{B}(\bff{\phi},\bff{\phi};\bff{v},\bff{w})
		+ \cal{C}(\bff{\phi}; \bff{v}, \bff{w}),
	\end{align}
	where~$\cal{B}$ and~$\cal{C}$ are defined by~\eqref{equ:bilinear B}
	and~\eqref{equ:bilinear C}, respectively, and~$\cal{A}_0$ is defined by
	\begin{align*}
		\cal{A}_0(\bff{v},\bff{w})
		&:=
		\alpha \inpro{\bff{v}}{\bff{w}}_{\bb{L}^2}
		+ \beta_1 \inpro{\nabla \bff{v}}{\nabla \bff{w}}_{\bb{L}^2} 
		+ \beta_2 \inpro{\Delta \bff{v}}{\Delta \bff{w}}_{\bb{L}^2}.
	\end{align*}
	Then~$\cal{A}(\bff{\phi};\cdot,\cdot)$ is bounded, i.e., there exists $\kappa>0$ depending only on $\alpha, \beta_1, \ldots, \beta_5$, and
	$\norm{\bff{\phi}}{\bb{H}^2}$, such that
	\begin{equation}\label{equ:A bounded}
		|\cal{A}(\bff{\phi}; \bff{v}, \bff{w})| 
		\leq 
		\kappa \norm{\bff{v}}{\bb{H}^2} \norm{\bff{w}}{\bb{H}^2},
		\quad \forall \bff{v}, \bff{w} \in \bb{H}_{\bff{n}}^2.
	\end{equation}
	Moreover, $\cal{A}(\bff{\phi};\cdot,\cdot)$ is coercive, i.e., 
	\begin{equation}\label{equ:coercive}
		\cal{A}(\bff{\phi}; \bff{v}, \bff{v})
		\geq
		\mu \norm{\bff{v}}{\bb{H}^2}^2,
		\quad \forall \bff{v} \in \bb{H}_{\bff{n}}^2,
	\end{equation}
	for all $\alpha > 0$ if $\beta_1 \ge 0$, and for all $\alpha > \beta_1^2/\beta_2$
	if $\beta_1 < 0$. The constant $\mu$ depends
	on $\alpha$, $\beta_1$, $\beta_2$, and $\Omega$, but is independent
	of~$\bff{\phi}$. 
\end{lemma}

\begin{proof}
	For any $\bff{v},\bff{w} \in \bb{H}_{\bff{n}}^2$, by H\"{o}lder's inequality, \eqref{equ:B bounded H1}, \eqref{equ:C bounded H1}, and Sobolev embedding, we have
	\begin{align*}
		|\cal{A}(\bff{\phi}; \bff{v}, \bff{w})| 
		&\leq
		\abs{\cal{A}_0(\bff{v},\bff{w})}
		+ \abs{\cal{B}(\bff{\phi},\bff{\phi};\bff{v},\bff{w})}
		+ \abs{\cal{C}(\bff{\phi}; \bff{v}, \bff{w})}
		\\
		&\leq
		\big( \alpha + |\beta_1| + \beta_2 + \beta_3 \norm{\bff{\phi}}{\bb{H}^2}^2 + \beta_4 \norm{\bff{\phi}}{\bb{H}^2} + 3 \beta_5 \norm{\bff{\phi}}{\bb{H}^2}^2 \big)
		\norm{\bff{v}}{\bb{H}^2} \norm{\bff{w}}{\bb{H}^2}
		\\
		&=: \kappa \norm{\bff{v}}{\bb{H}^2} \norm{\bff{w}}{\bb{H}^2},
	\end{align*}
	for any $\alpha>0$, thus showing~\eqref{equ:A bounded}. 
	
	On the other hand, for any~$\bff{v} \in \bb{H}_{\bff{n}}^2$, noting
	\eqref{equ:equivnorm-nabL2} we deduce
	\begin{align}\label{equ:Avv}
		\cal{A}(\bff{\phi}; \bff{v}, \bff{v})
		&= 
		\alpha \norm{\bff{v}}{\bb{L}^2}^2
		+ \beta_1 \norm{\nabla \bff{v}}{\bb{L}^2}^2 
		+ \beta_2 \norm{\Delta \bff{v}}{\bb{L}^2}^2 
		+ \beta_3 \norm{|\bff{\phi}||\bff{v}|}{\bb{L}^2}^2 
		\nonumber\\
		&\quad
		+ \beta_5 \norm{|\bff{\phi}| |\nabla \bff{v}|}{\bb{L}^2}^2
		+ 2\beta_5 \norm{\bff{\phi} \cdot \nabla \bff{v}}{\bb{L}^2}^2,
	\end{align}
	so that
	\begin{align}\label{equ:A coe}
		\cal{A}(\bff{\phi}; \bff{v}, \bff{v})
		\geq
		\alpha \norm{\bff{v}}{\bb{L}^2}^2
		+ \beta_1 \norm{\nabla \bff{v}}{\bb{L}^2}^2 
		+ \beta_2 \norm{\Delta \bff{v}}{\bb{L}^2}^2.
	\end{align}
	If $\beta_1$ is non-negative, then for any $\alpha>0$ we have~\eqref{equ:coercive}
	with $\mu = \min\{\alpha,\beta_1,\beta_2\} $.
	If $\beta_1$ is negative, we derive from~\eqref{equ:A coe}
	and~\eqref{equ:equivnorm-nabL2} that
	\begin{align*} 
		\cal{A}(\bff{\phi}; \bff{v}, \bff{v})
		&\ge
		\alpha \norm{\bff{v}}{\bb{L}^2}^2
		-|\beta_1| \norm{\bff{v}}{\bb{L}^2} \norm{\Delta \bff{v}}{\bb{L}^2}
		+ \beta_2 \norm{\Delta \bff{v}}{\bb{L}^2}^2 
		\\
		&=
		\frac{3\alpha}{4} \norm{\bff{v}}{\bb{L}^2}^2
		+
		\Big(
		\frac{\sqrt{\alpha}}{2} \norm{\bff{v}}{\bb{L}^2}
		-
		\frac{|\beta_1|}{\sqrt{\alpha}} \norm{\Delta\bff{v}}{\bb{L}^2}
		\Big)^2
		+
		\Big(
		\beta_2 - \frac{\beta_1^2}{\alpha} 
		\Big)
		\norm{\Delta\bff{v}}{\bb{L}^2}^2
		\\
		&\geq
		\mu^\ast
		\Big(
		\norm{\bff{v}}{\bb{L}^2}^2
		+
		\norm{\Delta\bff{v}}{\bb{L}^2}^2
		\Big)
		\ge
		\mu^\ast C \norm{\bff{v}}{\bb{H}^2}^2,
	\end{align*}
	where $\mu^\ast = \min\{3\alpha/4,\beta_2-\beta_1^2/\alpha\}$ is positive if
	$\alpha>0$ is sufficiently large, and the constant~$C$ is given
	by~\eqref{equ:equivnorm-h2}. Inequality~\eqref{equ:coercive} then follows with
	$\mu=C\mu^\ast$.
\end{proof}

In the following, we state some assumptions on the exact solution which will be used when deriving various convergence estimates.

\subsection{Assumptions}
Let $\bff{j}\in \bb{L}^\infty$ be given such that $\nabla \cdot \bff{j}=0$ on $\Omega$ and $\bff{j}\cdot \bff{n}=0$ on $\partial\Omega$. Given $T>0$ and $\bff{u}_0\in\bb{H}_{\bff{n}}^2$, a weak
solution to \eqref{equ:llbar} is a function $\bff{u}:[0,T]\to \bb{H}_{\bff{n}}^2$ satisfying
for all $t\in [0,T]$ and $\bff{\phi}\in \bb{H}_{\bff{n}}^2$,
\begin{align}\label{equ:weakform}
	&\inpro{\partial_t \bff{u}(t)}{\bff{\phi}}_{\bb{L}^2} 
	+ 
	\cal{A}(\bff{u}(t); \bff{u}(t), \bff{\phi})
	-
	(\alpha+\beta_3) \inpro{\bff{u}(t)}{\bff{\phi}}_{\bb{L}^2}
	-
	\beta_6 \inpro{(\bff{j}\cdot \nabla)\bff{u}(t)}{\bff{\phi}}_{\bb{L}^2}
	= 0,
\end{align}
with $\bff{u}(0)= \bff{u}_0$. For simplicity, we will take $\norm{\bff{j}}{\bb{L}^\infty}=1$.

From this point onwards, we let
\begin{equation}\label{equ:r}
	r =
	\begin{cases}
		3, \quad & \text{if } d\in \{1,2\},
		\\
		9, \quad & \text{if } d=3,
	\end{cases}
\end{equation}
be the degree of piecewise polynomials in the conforming finite element space. Note that $r=9$ is required for $d=3$ to maintain $C^1$-conformity~\cite{Zen73}. We reiterate that it is computationally prohibitive in practice to apply the $C^1$-conforming numerical scheme for $d=3$, but we include the theoretical analysis in this case for the sake of completeness.

We assume that problem~\eqref{equ:llbar} possesses a
solution~$\bff{u}$ which satisfies
\begin{equation}\label{equ:ass 1}
	\norm{\bff{u}}{L^\infty(\bb{H}^{r+1})}
	+ \norm{\partial_t \bff{u}}{L^\infty(\bb{H}^{r+1})}
	+ \norm{\partial_t^2 \bff{u}}{L^\infty(\bb{L}^2)}
	\le K,
\end{equation}
and~$K>0$ depends on~$\bff{u}_0$. 
Furthermore, in Section \ref{sec:bdf2}, we will assume that the solution $\bff{u}$ satisfies
\begin{equation}\label{equ:ass 2}
	\norm{\bff{u}}{L^\infty(\bb{H}^{r+1})}
	+ \norm{\partial_t \bff{u}}{L^\infty(\bb{H}^{r+1})}
	+ \norm{\partial_t^2 \bff{u}}{L^\infty(\bb{H}^1)}
	+ \norm{\partial_t^3 \bff{u}}{L^\infty(\bb{L}^2)}
	\le K.
\end{equation}
These assumptions are needed so that, among other things, the optimal order approximation properties~\eqref{equ:fin approx} and~\eqref{equ:interp approx} of our finite element spaces are satisfied.

\section{Semi-discrete Galerkin Approximation}\label{sec:semidiscrete}

Let $\mathcal{T}_h$ be a shape-regular triangulation of $\Omega\subset \bb{R}^d$
into intervals (in 1D), triangles (in 2D), or tetrahedra (in 3D) with maximal mesh-size $h$. We
introduce a $C^1$-conforming finite element space $\bb{V}_h \subset \bb{H}_{\bff{n}}^2 \cap C^1(\overline{\Omega})$ which consists of
$C^1$-$P_r$ elements ($C^1$-piecewise polynomials of degree~$r$) defined on $\mathcal{T}_h$. For example, in 1D
we take $\bb{V}_h$ to be the space of $C^1$-$P_3$ Hermite finite elements~\cite{Cia78}, whereas we use the Hsieh--Clough--Tocher $C^1$-$P_3$ macroelements in 2D~\cite{Cia74, Man78}, and the \v{Z}en\'{i}\v{s}ek $C^1$-$P_9$ elements in 3D~\cite{Zen73, Zha09}.

Some essential approximation properties for our finite element spaces~\cite{Cia74, Cia78, DouDupPerSco79, Zha09} are described as follows. Due to shape regularity of the triangulation, the
following optimal order $\bb{H}^\beta$-approximation property for this $C^1$-conforming finite element holds: There exists a constant $C$ independent of $h$ and $\bff{v}$ such that
for any $\bff{v} \in \bb{H}^{r+1}$, with $0\le \beta \le 2$,
\begin{align}\label{equ:fin approx}
	\inf_{\bff{\chi} \in {\bb{V}}_h}  \norm{\bff{v} - \bff{\chi}}{\bb{H}^\beta} 
	\leq 
	C h^{r+1-\beta} \norm{\bff{v}}{\bb{H}^{r+1}}.
\end{align}
Furthermore, there exists a nodal interpolation operator~$\mathcal{I}_h: \bb{H}^{r+1} \to \bb{V}_h$ that satisfies
\begin{align}\label{equ:interp approx}
 	\norm{\bff{v} - \mathcal{I}_h(\bff{v})}{\bb{H}^\beta} 
 	\leq 
 	C h^{r+1-\beta} \|\bff{v}\|_{\bb{H}^{r+1}},
\end{align}
such that (see~\cite[Section 4.4]{BreSco08}) for $p\in [1,\infty)$ and $m\in [0,2]$,
\begin{align}\label{equ:interp bounded}
	\norm{\cal{I}_h(\bff{v})}{\bb{W}^{m,p}} \leq C\norm{\bff{v}}{\bb{W}^{m,p}},
	\quad \forall \bff{v}\in \bb{W}^{m,p}.
\end{align}
At some parts of the paper (which will be clearly indicated), we will also assume that the triangulation is quasi-uniform so that the inverse estimates hold, see~\cite[Theorem~4.5.11]{BreSco08}.

A semi-discrete Galerkin approximation to problem \eqref{equ:weakform} is defined as
follows: Find $\bff{u}_h: [0,T] \to \bb{V}_h$ such that for all
$t\in [0,T]$ and $\bff{\chi}\in \bb{V}_h$,
\begin{align}\label{equ:weaksemidisc}
	\inpro{\partial_t \bff{u}_h(t)}{\bff{\chi}}_{\bb{L}^2} 
	+
	\cal{A}(\bff{u}_h(t);\bff{u}_h(t), \bff{\chi})
	-
	(\alpha+\beta_3) \inpro{\bff{u}_h(t)}{\bff{\chi}}_{\bb{L}^2}
	-
	\beta_6 \inpro{(\bff{j}\cdot \nabla) \bff{u}_h(t)}{\bff{\chi}}_{\bb{L}^2}
	= 0,
\end{align}
and $\bff{u}_h(0)= \bff{u}_{0,h}$, where $\bff{u}_{0,h}$ is an approximation of
$\bff{u}_0$ in $\bb{V}_h$, say $\bff{u}_{0,h}= \mathcal{I}_h (\bff{u}_0)$, so
that
\begin{equation}\label{equ:u0 u0h}
	\norm{\bff{u}_{0,h}-\bff{u}_0}{\bb{H}^\beta} 
	\leq
	Ch^{r+1-\beta}.
\end{equation}
This system of ordinary differential equations~\eqref{equ:weaksemidisc} has a unique solution $\bff{u}_h$
defined on the interval $[0, t_h]\subseteq [0,T]$ by the Cauchy--Lipschitz
theorem; see~\cite{SoeTra23}. We will prove several stability results, which will be used to ensure
that the semi-discrete solution $\bff{u}_h$ can be continued globally to $[0,T]$
for any $T>0$ when $d=1$ or $d=2$ with any $\bff{u}_{0,h}\in \bb{V}_h$, or when $d=3$ with sufficiently small $\bff{u}_{0,h}$.

\begin{proposition}\label{pro:semidisc-est1}
Let $h>0$ and $T>0$ be given. For all $t\in [0,T]$ and initial data $\bff{u}_{0,h} \in \bb{V}_h$,
\begin{align}\label{equ:semidisc-est1}
	\norm{\bff{u}_h(t)}{\bb{L}^2}^2 
	+  
	\int_0^t \norm{\Delta \bff{u}_h(s)}{\bb{L}^2}^2 \ds  
	+ 
	\int_0^t \norm{\bff{u}_h(s)}{\bb{L}^4}^{4} \ds 
	\lesssim 
	\norm{\bff{u}_h(0)}{\bb{L}^2}^2,
\end{align}
where the constant depends on $T$, but is independent of $h$.
\end{proposition}

\begin{proof}
The proof is similar to that of \cite[Proposition~3.6]{SoeTra23} and is omitted.
\end{proof}

\begin{proposition}\label{pro:semidisc est2}
Let $h>0$ and $T>0$ be given. Then the following statements hold true.
\begin{enumerate}
	\renewcommand{\labelenumi}{\theenumi}
	\renewcommand{\theenumi}{{\rm (\roman{enumi})}}
	\item 
	If $d=1$ or $d=2$, for all $t\in [0,T]$ and initial data $\bff{u}_{0,h} \in \bb{V}_h$,
	\begin{align}\label{equ:semidisc-est2}
		\norm{\Delta \bff{u}_h(t)}{\bb{L}^2}^2 
		+ 
		\norm{\nabla \bff{u}_h(t)}{\bb{L}^2}^2 
		+ 
		\norm{\bff{u}_h(t)}{\bb{L}^4}^4 
		+ 
		\int_0^t \norm{\frac{\partial \bff{u}_h(s)}{\partial s}}{\bb{L}^2}^2 \ds 
		\lesssim 
		\norm{\bff{u}_h(0)}{\bb{H}^2}^2,
	\end{align}
	where the constant depends on $T$, but is independent of $h$.
	\item
	If $d=3$, then \eqref{equ:semidisc-est2} holds for all $t\in [0,T]$ when $\norm{\bff{u}_h(0)}{\bb{H}^2} \lesssim T^{-2/3}$.
\end{enumerate}
\end{proposition}

\begin{proof}
Taking $\bff{\chi}= \partial_t \bff{u}_h$ in \eqref{equ:weaksemidisc} gives
\begin{align*}
		\norm{\partial_t \bff{u}_h}{\bb{L}^2}^2 
		&+ \frac{\beta_1}{2} \ddt \norm{\nabla \bff{u}_h}{\bb{L}^2}^2 
		+ \frac{\beta_2}{2} \ddt \norm{\Delta \bff{u}_h}{\bb{L}^2}^2 
		+ \frac{\beta_3}{4} \ddt \norm{\bff{u}_h}{\bb{L}^4}^4 
		+ \beta_4 \inpro{\bff{u}_h \times
			\Delta\bff{u}_h}{\partial_t\bff{u}_h}_{\bb{L}^2}
		\\
		&+ \beta_5 \inpro{\nabla(|\bff{u}_h|^2 \bff{u}_h)}
		{\partial_t\nabla \bff{u}_h}_{\bb{L}^2} 
		= 
		\frac{\beta_3}{2} \ddt \norm{\bff{u}_h}{\bb{L}^2}^2
		+
		\beta_6 \inpro{(\bff{j}\cdot\nabla)\bff{u}_h}{\partial_t \bff{u}_h}_{\bb{L}^2}
\end{align*}
Similar argument as in \cite[Proposition~3.6]{SoeTra23} then yields
\begin{align}\label{equ:dtu} 
		\norm{\partial_t \bff{u}_h}{\bb{L}^2}^2 
		&+ \ddt \norm{\nabla \bff{u}_h}{\bb{L}^2}^2 
		+ \ddt \norm{\Delta \bff{u}_h}{\bb{L}^2}^2 
		+ \ddt \norm{\bff{u}_h}{\bb{L}^4}^4 
		+ \ddt \norm{\bff{u}_h\cdot\nabla\bff{u}_h}{\bb{L}^2}^2
		+ \ddt \norm{|\bff{u}_h| |\nabla\bff{u}_h|}{L^2}^2
		\nonumber\\
		&\lesssim
		\ddt \norm{\bff{u}_h}{\bb{L}^2}^2
		+
		\norm{\nabla \bff{u}_h}{\bb{L}^2}^2
		+
		\norm{\bff{u}_h}{\bb{L}^\infty}^2 \norm{\Delta\bff{u}_h}{\bb{L}^2}^2
		+
		\norm{\bff{u}_h}{\bb{L}^\infty}^2 \norm{\nabla\bff{u}_h}{\bb{L}^4}^4.
\end{align}
It remains to bound the last two terms on the right-hand side of \eqref{equ:dtu}.

\medskip
\noindent
\underline{Case 1: $d=1$}. The Gagliardo--Nirenberg inequality gives
\begin{align*}
	\norm{\bff{u}_h}{\bb{L}^\infty}^2
	&\lesssim
	\norm{\bff{u}_h}{\bb{L}^2}
	\norm{\bff{u}_h}{\bb{H}^1},
	\\
	\norm{\nabla \bff{u}_h}{\bb{L}^4}^4
	&\lesssim
	\norm{\nabla \bff{u}_h}{\bb{L}^2}^3
	\norm{\nabla \bff{u}_h}{\bb{H}^1}.
\end{align*}
Inserting these into~\eqref{equ:dtu}, using the Sobolev embedding,
and integrating over~$(0,t)$ yield
\begin{align}\label{equ:dsu} 
	\int_{0}^{t} \norm{\partial_s\bff{u}_h(s)}{\bb{L}^2}^2 \ds
	&+
	\norm{\Delta\bff{u}_h(t)}{\bb{L}^2}^2
	+
	\norm{\nabla \bff{u}_h(t)}{\bb{L}^2}^2
	+
	\norm{\bff{u}_h(t)}{\bb{L}^4}^4
	+ 
	\norm{\bff{u}_h(t)\cdot\nabla\bff{u}_h(t)}{\bb{L}^2}^2
	\nonumber\\
	&\lesssim
	1 + \int_{0}^{t} \Big(
	\norm{\bff{u}_h(s)}{\bb{H}^1}^2
	\norm{\Delta\bff{u}_h(s)}{\bb{L}^2}^2
	+
	\norm{\bff{u}_h(s)}{\bb{H}^1}^4
	\norm{\bff{u}_h(s)}{\bb{H}^2}
	\Big) \ds
	\nonumber\\
	&\lesssim
	1 + \int_{0}^{t} \Big(
	\norm{\Delta\bff{u}_h(s)}{\bb{L}^2}^2
	+
	\norm{\Delta\bff{u}_h(s)}{\bb{L}^2}^3
	+
	\norm{\Delta\bff{u}_h(s)}{\bb{L}^2}^4
	\Big) \ds,
\end{align}
where in the last step we used \eqref{equ:equivnorm-h2} and \eqref{equ:equivnorm-nabL2}. Noting that $\norm{\Delta \bff{u}_h}{L^2(\bb{L}^2)} \lesssim \norm{\bff{u}_h(0)}{\bb{L}^2}$ by
Proposition \ref{pro:semidisc-est1}, we have the required result by invoking the Gronwall inequality.

\medskip
\noindent
\underline{Case 2: $d=2$}. The Gagliardo--Nirenberg inequality gives
\begin{align*}
	\norm{\bff{u}_h}{\bb{L}^\infty}^2
	&\lesssim
	\norm{\bff{u}_h}{\bb{L}^2}
	\norm{\bff{u}_h}{\bb{H}^2},
	\\
	\norm{\nabla \bff{u}_h}{\bb{L}^4}^4
	&\lesssim
	\norm{\nabla \bff{u}_h}{\bb{L}^2}^2
	\norm{\nabla \bff{u}_h}{\bb{H}^1}^2.
\end{align*}
Inserting these into~\eqref{equ:dtu}, using the Sobolev embedding,
and integrating over~$(0,t)$ yield
\begin{align}\label{equ:dsu-2} 
	\int_{0}^{t} \norm{\partial_s\bff{u}_h(s)}{\bb{L}^2}^2 \ds
	&+
	\norm{\Delta\bff{u}_h(t)}{\bb{L}^2}^2
	+
	\norm{\nabla \bff{u}_h(t)}{\bb{L}^2}^2
	+
	\norm{\bff{u}_h(t)}{\bb{L}^4}^4
	+ 
	\norm{\bff{u}_h(t)\cdot\nabla\bff{u}_h(t)}{\bb{L}^2}^2
	\nonumber\\
	&\lesssim
	1 + \int_{0}^{t} \Big(
	\norm{\bff{u}_h(s)}{\bb{H}^2}^3
	+
	\norm{\nabla \bff{u}_h(s)}{\bb{L}^2}^2
	\norm{\bff{u}_h(s)}{\bb{H}^2}^3
	\Big) \ds
	\nonumber\\
	&\lesssim
	1 + \int_{0}^{t} \Big(
	\norm{\Delta\bff{u}_h(s)}{\bb{L}^2}^3
	+
	\norm{\Delta\bff{u}_h(s)}{\bb{L}^2}^4
	\Big) \ds,
\end{align}
where in the last step we used \eqref{equ:equivnorm-h2} and \eqref{equ:equivnorm-nabL2}. Invoking the Gronwall inequality (and noting Proposition \ref{pro:semidisc-est1}) yields the required result.

\medskip
\noindent
\underline{Case 3: $d=3$}. The Gagliardo--Nirenberg inequality gives
\begin{align*}
	\norm{\bff{u}_h}{\bb{L}^\infty}^2
	&\lesssim
	\norm{\bff{u}_h}{\bb{L}^2}^{1/2}
	\norm{\bff{u}_h}{\bb{H}^2}^{3/2},
	\\
	\norm{\nabla \bff{u}_h}{\bb{L}^4}^4
	&\lesssim
	\norm{\bff{u}_h}{\bb{L}^2}^{1/2}
	\norm{\bff{u}_h}{\bb{H}^2}^{7/2}.
\end{align*}
Inserting these into~\eqref{equ:dtu}, using the Sobolev embedding,
and integrating over~$(0,t)$ yield
\begin{align}\label{equ:dsu-3} 
	\int_{0}^{t} \norm{\partial_s\bff{u}_h(s)}{\bb{L}^2}^2 \ds
	&+
	\norm{\Delta\bff{u}_h(t)}{\bb{L}^2}^2
	+
	\norm{\nabla \bff{u}_h(t)}{\bb{L}^2}^2
	+
	\norm{\bff{u}_h(t)}{\bb{L}^4}^4
	+ 
	\norm{\bff{u}_h(t)\cdot\nabla\bff{u}_h(t)}{\bb{L}^2}^2
	\nonumber\\
	&\lesssim
	1 + \int_{0}^{t} 
	\norm{\bff{u}_h(s)}{\bb{H}^2}^5 \ds
	\lesssim
	1 + \int_{0}^{t}
	\norm{\Delta\bff{u}_h(s)}{\bb{L}^2}^5 \ds,
\end{align}
where in the last step we used \eqref{equ:equivnorm-h2}. Hence, by the Bihari--Gronwall's inequality (Theorem \ref{the:bihari}), the result follows for $t\in [0,T]$ if $\norm{\bff{u}_h(0)}{\bb{H}^2} \lesssim T^{-2/3}$.
This completes the proof of the proposition.
\end{proof}
%

To estimate the error in the semi-discrete approximation, we write the error at any time $t\in [0,T]$ as a sum of two terms:
\begin{equation}\label{equ:uh u xi rho}
	\bff{u}_h(t) - \bff{u}(t) 
	= \bff{\theta}(t) + \bff{\rho}(t)
\end{equation}
where
\begin{align}\label{equ:xi rho}
	\bff{\rho}(t) :=
	\widetilde{\bff{u}}_h(t) - \bff{u}(t)
	\quad\text{and}\quad
	\bff{\theta}(t) :=
	\bff{u}_h(t)- \widetilde{\bff{u}}_h(t),
\end{align}
with $\widetilde{\bff{u}}_h(t)$ being the elliptic projection of the exact
solution $\bff{u}(t)$ defined by
\begin{align}\label{equ:auxiliary}
\cal{A}(\bff{u}(t); \widetilde{\bff{u}}_h(t)- \bff{u}(t), \bff{\chi})=0 \quad \text{for all } \bff{\chi}\in \bb{V}_h.
\end{align}
Note that the elliptic projection of $\bff{u}(t)$ exists by the Lax--Milgram
theorem. We can now prove some estimates for this projection.



\begin{proposition}\label{pro:auxestimate}
Let $\widetilde{\bff{u}}_h$ be defined by \eqref{equ:auxiliary}, $\bff{u}$ be the solution of \eqref{equ:weakform} which
satisfies~\eqref{equ:ass 1}, and
$\bff{\rho}(t) = \widetilde{\bff{u}}_h(t) - \bff{u}(t)$ as defined in
\eqref{equ:xi rho}. Then for any $t\in [0,T]$,
\begin{equation}\label{equ:rho}
	\norm{\bff{\rho}(t)}{\bb{L}^2} 
	+ h \norm{\nabla \bff{\rho}(t)}{\bb{L}^2} 
	+ h^2 \norm{\Delta \bff{\rho}(t)}{\bb{L}^2} 
	\leq 
	C h^{r+1} \norm{\bff{u}}{L^\infty(\bb{H}^{r+1})},
\end{equation}
and
\begin{equation}\label{equ:partial rho}
	\norm{\partial_t\bff{\rho}(t)}{\bb{L}^2} 
	+ h \norm{\nabla \partial_t \bff{\rho}(t)}{\bb{L}^2} 
	+ h^{2} \norm{\Delta \partial_t \bff{\rho}(t)}{\bb{L}^2} 
	\leq 
	C  h^{r+1} \norm{\partial_t \bff{u}}{L^\infty(\bb{H}^{r+1})}.
\end{equation}
The constant $C$ depends on $\kappa,\mu$, and $T$, where $\kappa$ and $\mu$ were defined in \eqref{equ:A bounded} and \eqref{equ:coercive}, respectively.
\end{proposition}

\begin{proof}

For all $\bff{\chi} \in \bb{V}_h$, by the coercivity and the boundedness of $\cal{A}$, and \eqref{equ:auxiliary},
\begin{align*}
	\mu \norm{\widetilde{\bff{u}}_h(t) - \bff{u}(t)}{\bb{H}^2}^2 
	&\leq 
	\cal{A} (\bff{u}(t); \widetilde{\bff{u}}_h(t) - \bff{u}(t), \widetilde{\bff{u}}_h(t) - \bff{u}(t)) 
	\\
	&=
	\cal{A} (\bff{u}(t); \widetilde{\bff{u}}_h(t) - \bff{u}(t), \widetilde{\bff{u}}_h(t) - \bff{\chi} ) 
	-
	\cal{A} (\bff{u}(t); \widetilde{\bff{u}}_h(t) - \bff{u}(t), \bff{u}(t) - \bff{\chi}) 
	\\
	&\leq 
	\kappa \norm{\widetilde{\bff{u}}_h(t) - \bff{u}(t)}{\bb{H}^2} \norm{\bff{u}(t)-\bff{\chi}}{\bb{H}^2}.
\end{align*}
Therefore, by \eqref{equ:fin approx},
\begin{align}\label{equ:rho h2}
	\norm{\bff{\rho}(t)}{\bb{H}^2} 
	= 
	\norm{\widetilde{\bff{u}}_h(t) - \bff{u}(t)}{\bb{H}^2} 
	\leq 
	 \inf_{\bff{\chi} \in \bb{V}_h} \norm{\bff{u}(t)- \bff{\chi}}{\bb{H}^2} 
	\leq
	Ch^{r-1} \norm{\bff{u}(t)}{\bb{H}^{r+1}}.
\end{align}

To prove the $\bb{L}^2$-estimate, we use the Aubin--Nitsche duality argument.
For each $\bff{u}(t) \in \bb{H}^{4} \cap \bb{H}_{\bff{n}}^2$, let $\bff{\psi}(t)\in \bb{H}^{4} \cap \bb{H}_{\bff{n}}^2$ satisfy
\begin{equation}\label{equ:Au psi xi}
	\cal{A}(\bff{u}(t); \bff{\xi}, \bff{\psi}(t))
	= 
	\inpro{ \bff{\rho}(t) }{\bff{\xi}}_{\bb{L}^2} 
	\quad \text{for all } \bff{\xi}\in \bb{H}^2.
\end{equation}
For any $t\in [0,T]$, such $\bff{\psi}(t)$ exists by Lemma \ref{lem:aux H4} under the assumed conditions on $\Omega$. Moreover, by \eqref{equ:aux H4 est part 2},
\begin{align*}
	\norm{\bff{\psi}(t)}{\bb{H}^{4}} 
	\lesssim 
	\norm{\bff{\rho}(t)}{\bb{L}^2}.
\end{align*}
Therefore, taking $\bff{\xi}=\bff{\rho}(t)$ in~\eqref{equ:Au psi xi} and noting \eqref{equ:auxiliary},
we have for all $\bff{\chi}\in\bb{V}_h$,
\begin{align*}
	\norm{\bff{\rho}(t)}{\bb{L}^2}^2
	&= 
	\cal{A}(\bff{u}(t); \bff{\rho}(t), \bff{\psi}(t)) 
	=  
	\cal{A}(\bff{u}(t); \bff{\rho}(t), \bff{\psi}(t)-\bff{\chi})
	\\
	&\lesssim 
	 \norm{\bff{\rho}(t)}{\bb{H}^2} 
	\inf_{\bff{\chi}\in \bb{V}_h} \norm{\bff{\psi}(t)-\bff{\chi}}{\bb{H}^2} 
	\leq Ch^{r+1} \norm{\bff{u}(t)}{\bb{H}^{r+1}} \norm{\bff{\rho}(t)}{\bb{L}^2},
\end{align*}
where in the last step we used \eqref{equ:rho h2} and \eqref{equ:fin approx}. This implies
\[
	\norm{\bff{\rho}(t)}{\bb{L}^2} 
	\leq 
	C h^{r+1}
	\norm{\bff{u}(t)}{\bb{H}^{r+1}}
	\leq
	Ch^{r+1} \norm{\bff{u}}{L^\infty(\bb{H}^{r+1})}.
\]
The estimate \eqref{equ:rho} then follows by using \eqref{equ:equivnorm-nabL2}.

We now prove \eqref{equ:partial rho}. For ease of presentation, the dependence
of the functions on $t$ will sometimes be omitted. 
Note that by the coercivity of $\cal{A}$ and the definition of $\bff{\rho}$,
\begin{align}\label{equ:rho A}
	\mu \norm{\partial_t \bff{\rho}}{\bb{H}^2}^2 
	\leq 
	\cal{A} \big( \bff{u}; \partial_t \bff{\rho}, \partial_t \bff{\rho} \big)
	=
	\cal{A} \big( \bff{u}; \partial_t \bff{\rho}, \cal{I}_h (\partial_t \bff{\rho}) \big)
	+ 
	\cal{A} \big(\bff{u}; \partial_t \bff{\rho}, \cal{I}_h (\partial_t
	\bff{u}) - \partial_t \bff{u} \big),
\end{align}
where we used $\cal{I}_h(\partial_t \bff{\rho})= \partial_t \widetilde{\bff{u}}_h- \cal{I}_h(\partial_t \bff{u})$. 
We now estimate each term on the last line.
To this end, noting the time derivative of~\eqref{equ:auxiliary} and~\eqref{equ:bilinear}, we
differentiate equation \eqref{equ:auxiliary} with respect to $t$ to have,
for all $\bff{\chi}\in \bb{V}_h$,
\begin{equation}\label{equ:A2BC}
	\cal{A}(\bff{u}; \partial_t \bff{\rho}, \bff{\chi}) 
	+
	2\cal{B}(\bff{u}, \partial_t \bff{u}; \bff{\rho}, \bff{\chi})  
	+
	\cal{C}(\partial_t \bff{u}; \bff{\rho},\bff{\chi})  
	= 0.
\end{equation}
Thus, for the first term on the right-hand side of~\eqref{equ:rho A}, by the boundedness of $\cal{B}$, $\cal{C}$, and $\mathcal{I}_h$ we have
\begin{align}\label{equ:Buu}
	\big| \cal{A} \big( \bff{u}; \partial_t \bff{\rho}, \cal{I}_h (\partial_t \bff{\rho}) \big) \big|
	&\lesssim
	\big| \cal{B} \big( \bff{u}; \partial_t \bff{u}; \bff{\rho}, \cal{I}_h (\partial_t \bff{\rho}) \big) \big|
	+
	\big| \cal{C} \big(\partial_t \bff{u}; \bff{\rho}, \cal{I}_h (\partial_t \bff{\rho}) \big) \big|
	\nonumber\\
	&\lesssim
	\norm{\bff{\rho}}{\bb{H}^2} \norm{ \cal{I}_h (\partial_t \bff{\rho})}{\bb{H}^2}
	\lesssim
	h^{r-1} \norm{\partial_t \bff{\rho}}{\bb{H}^{r+1}},
\end{align}
where in the last step we used~\eqref{equ:rho}. For the second
term on the right-hand side of~\eqref{equ:rho A}, by the boundedness
of~$\cal{A}$ (see~\eqref{equ:A bounded}), the assumption on the exact
solution~$\bff{u}$ (see~\eqref{equ:ass 1}), and \eqref{equ:interp approx}, we
have
\begin{align*}
	\big| \cal{A} \big(\bff{u}; \partial_t \bff{\rho}, \cal{I}_h (\partial_t \bff{u})- \partial_t \bff{u} \big) \big|
	&\lesssim 
	\norm{\partial_t \bff{\rho}}{\bb{H}^2} \norm{\cal{I}_h (\partial_t \bff{u})- \partial_t \bff{u}}{\bb{H}^2}
	\lesssim 
	h^{r-1} \norm{\partial_t \bff{\rho}}{\bb{H}^2} \norm{\partial_t \bff{u}}{\bb{H}^4}.
\end{align*}
This, together with \eqref{equ:rho A} and \eqref{equ:Buu}, gives
\begin{align}\label{equ:partial t rho h2}
	\norm{\partial_t \bff{\rho}(t)}{\bb{H}^2} \leq 
	C h^{r-1} \norm{\partial_t \bff{u}}{L^\infty(\bb{H}^{r+1})},
\end{align}
where $C=C (\kappa, \mu, T )$.

For the estimate of~$\norm{\partial_t\bff{\rho}}{\bb{L}^2}$, we use duality
argument as before. For each $\bff{u}(t) \in \bb{H}^4 \cap \bb{H}_{\bff{n}}^2$, let $\bff{\psi}(t)\in
 \bb{H}^4 \cap \bb{H}_{\bff{n}}^2$ satisfy
\begin{align}\label{equ:partial t rho 2}
	\cal{A}(\bff{u}(t); \bff{\xi}, \bff{\psi}(t)) = \inpro{ \partial_t \bff{\rho}(t)}{ \bff{\xi}}_{\bb{L}^2} \quad \text{for all } \bff{\xi}\in \bb{H}^2.
\end{align}
The existence of $\bff{\psi}(t)$ is conferred by Lemma \ref{lem:aux H4}. Moreover,
by \eqref{equ:aux H4 est part 2},
\begin{align}\label{equ:psi H4 dt rho}
	\norm{\bff{\psi}(t)}{\bb{H}^4} 
	\lesssim \norm{\partial_t \bff{\rho}(t)}{\bb{L}^2},
\end{align}
where the constant is independent of~$t$ but depends on $T$.
Taking $\bff{\xi}= \partial_t \bff{\rho}$ in \eqref{equ:partial t rho 2}, we have
\begin{align}\label{equ:partial t rho 3}
	\cal{A}(\bff{u}(t); \partial_t \bff{\rho}, \bff{\psi}(t) )
	= \norm{\partial_t \bff{\rho}}{\bb{L}^2}^2.
\end{align}
Consequently, equations~\eqref{equ:partial t rho 3} and \eqref{equ:A2BC} yield
\begin{align*}
	\norm{\partial_t \bff{\rho}}{\bb{L}^2}^2
	&=
	\cal{A} (\bff{u}; \partial_t \bff{\rho}, \bff{\psi}-\bff{\chi}) 
	-
	2\cal{B} (\bff{u}, \partial_t \bff{u}; \bff{\rho}, \bff{\chi})
	-
	\cal{C}(\partial_t \bff{u};\bff{\rho}, \bff{\chi})
	\\
	&=
	\cal{A} (\bff{u}; \partial_t \bff{\rho}, \bff{\psi}-\bff{\chi}) 
	+ 
	2\cal{B} (\bff{u}, \partial_t \bff{u}; \bff{\rho}, \bff{\psi}-\bff{\chi})
	- 
	2\cal{B} (\bff{u},\partial_t \bff{u}; \bff{\rho}, \bff{\psi})
	+
	\cal{C}(\partial_t \bff{u}; \bff{\rho}, \bff{\psi}-\bff{\chi})
	-
	\cal{C}(\partial_t \bff{u}; \bff{\rho},\bff{\psi}).
\end{align*}
We estimate the first two terms on the right-hand side
above by using the boundedness of~$\cal{A}$, see~\eqref{equ:A bounded}, and of~$\cal{B}$, see~\eqref{equ:B bounded H1}. The third term is estimated by using~\eqref{equ:B bounded H2}, while the fourth term is bounded using H\"older's inequality. For the last term, we use integration by parts and~\eqref{equ:beta4 cross}. Thus we obtain
\begin{align*} 
	\norm{\partial_t \bff{\rho}}{\bb{L}^2}^2 
	&\lesssim
	\norm{\partial_t \bff{\rho}}{\bb{H}^2} \norm{\bff{\psi}- \bff{\chi}}{\bb{H}^2}
	+
	\norm{\bff{\rho}}{\bb{H}^1} \norm{\bff{\psi}- \bff{\chi}}{\bb{H}^1}
	+
	\norm{\bff{\rho}}{\bb{L}^2} \norm{\bff{\psi}}{\bb{H}^2}.
\end{align*}
By choosing $\bff{\chi} = \mathcal{I}_h (\bff{\psi})$ and
using successively~\eqref{equ:partial rho} and \eqref{equ:interp approx}, we obtain
\begin{align*}
	\norm{\partial_t \bff{\rho}}{\bb{L}^2}^2 
	&\lesssim
	h^{r-1} 
	(h^2 \norm{\bff{\psi}}{\bb{H}^4} )
	+ 
	h^r 
	(h^3 \norm{\bff{\psi}}{\bb{H}^4} ) 
	+ 
	h^{r+1} \norm{\bff{\psi}}{\bb{H}^2} 
	\\
	&\lesssim 
	h^{r+1} \norm{\partial_t \bff{\rho}}{\bb{L}^2} 
	+ 
	h^{r+3} \norm{\partial_t \bff{\rho}}{\bb{L}^2} 
	+ 
	h^{r+1} \norm{\partial_t \bff{\rho}}{\bb{L}^2},
\end{align*}
where in the last step we used \eqref{equ:psi H4 dt rho}. This implies
\[
	\norm{\partial_t \bff{\rho}(t)}{\bb{L}^2} \leq C h^{r+1} \norm{\partial_t \bff{u}}{L^\infty(\bb{H}^4)},
\]
The corresponding estimate for $\norm{\nabla \partial_t \bff{\rho}}{\bb{L}^2}$ then follows by interpolation \eqref{equ:equivnorm-nabL2}. This completes the proof of the proposition.
\end{proof}

The following results will be needed in the proof of the error estimate for the semi-discrete and fully discrete schemes. First, we show the stability of the elliptic projection.

\begin{lemma}\label{lem:auxbound}
	Let $\widetilde{\bff{u}}_h$ be defined as in \eqref{equ:auxiliary} with
	$\bff{u}\in\bb{H}^2$. Then
	\begin{align}\label{equ:aux H2}
		\norm{\widetilde{\bff{u}}_h(t)}{\bb{H}^2} 
		\leq C \norm{\bff{u}_0}{\bb{H}^2} \quad \text{for all } t\in [0,T].
	\end{align}
	Furthermore, if $\bff{u}\in \bb{H}^3$, and the triangulation
	$\cal{T}_h$ is \emph{quasi-uniform}, then for $p\leq 6$,
	\begin{align}\label{equ:aux H3}
		\norm{\widetilde{\bff{u}}_h(t)}{\bb{W}^{2,p}} 
		\leq C \norm{\bff{u}_0}{\bb{H}^3} \quad \text{for all } t\in [0,T].
	\end{align}
	Here, the constant $C$ depends on $T$, but is independent of $h$ and $\bff{u}$.
\end{lemma}

\begin{proof}
	Firstly, by the triangle inequality and Proposition \ref{pro:auxestimate},
	\begin{align*}
		\norm{\widetilde{\bff{u}}_h(t)}{\bb{H}^2}
		\leq
		\norm{\bff{\rho}(t)}{\bb{H}^2} 
		+
		\norm{\bff{u}(t)}{\bb{H}^2}
		\leq
		C \norm{\bff{u}_0}{\bb{H}^2}
		+ C \norm{\bff{u}_0}{\bb{H}^2},
	\end{align*}
	proving \eqref{equ:aux H2}.
	Next, if the triangulation $\cal{T}_h$ is quasi-uniform, then inverse estimates in $\Omega \subset
	\bb{R}^d$ (see~\cite[Theorem~4.5.11]{BreSco08}), give
	\begin{align*}
		\norm{\bff{\chi}}{\bb{W}^{2,p}} \leq Ch^{-d \left( \frac{1}{2}-\frac{1}{p} \right)} \norm{\bff{\chi}}{\bb{H}^2}, \quad \forall \bff{\chi}\in \bb{V}_h.
	\end{align*}
	By the triangle inequality, the above inverse estimate, \eqref{equ:rho}, and \eqref{equ:interp bounded}, we have
	\begin{align*}
		\norm{\widetilde{\bff{u}}_h}{\bb{W}^{2,p}}
		&\leq
		\norm{\widetilde{\bff{u}}_h - \mathcal{I}_h(\bff{u})}{\bb{W}^{2,p}} 
		+
		\norm{\mathcal{I}_h(\bff{u})}{\bb{W}^{2,p}}
		\\
		&\leq 
		Ch^{-d \left( \frac{1}{2}-\frac{1}{p} \right)} \norm{\widetilde{\bff{u}}_h - \mathcal{I}_h(\bff{u})}{\bb{H}^2}
		+
		\norm{\mathcal{I}_h(\bff{u})}{\bb{W}^{2,p}}
		\\
		&\leq Ch^{-d \left( \frac{1}{2}-\frac{1}{p} \right)} 
		\big(\norm{\bff{\rho}}{\bb{H}^2} + \norm{ \mathcal{I}_h(\bff{u})-\bff{u} }{\bb{H}^2} \big)
		+
		C \norm{\bff{u}}{\bb{W}^{2,p}}
		\\
		&\leq Ch^{-d \left( \frac{1}{2}-\frac{1}{p} \right)} \big(C h \norm{\bff{u}_0}{\bb{H}^3} + Ch \norm{\bff{u}}{\bb{H}^3} \big)
		+
		C \norm{\bff{u}}{\bb{H}^3}
		\leq C \norm{\bff{u}_0}{\bb{H}^3},
	\end{align*}
	since $d \left(\frac{1}{2}-\frac{1}{p}\right) \leq 1$ for $d\leq 3$ and $p\leq 6$. This shows \eqref{equ:aux H3}, completing the proof of the lemma.
\end{proof}

\begin{lemma}
Let $\mathcal{A}$ be as defined in \eqref{equ:bilinear} and $\bff{u}$ be the solution of \eqref{equ:weakform} which satisfies~\eqref{equ:ass 1}. Then
	\begin{align}
		\label{equ:A uh A tilde uh}
		\big| \cal{A}(\bff{u}; \widetilde{\bff{u}}_h, \bff{\chi})
		- 
		\cal{A}(\bff{u}_h; \widetilde{\bff{u}}_h, \bff{\chi}) \big|
		&\lesssim
		h^{2(r+1)}
		+
		\big(1+ \norm{\bff{u}_h}{\bb{H}^2}^2 \big)
		\big(\norm{\bff{\theta}}{\bb{L}^2}^2 + \norm{\bff{\chi}}{\bb{L}^2}^2 \big)
		+
		\epsilon \norm{\bff{\chi}}{\bb{H}^2}^2,
		\\
		\label{equ:A uh A tilde uh H2}
		\big| \cal{A}(\bff{u}; \widetilde{\bff{u}}_h, \bff{\chi})
		- 
		\cal{A}(\bff{u}_h; \widetilde{\bff{u}}_h, \bff{\chi}) \big|
		&\lesssim
		\big( 1+ \norm{\bff{u}_h}{\bb{H}^2}^2 \big) 
		\big( h^{2(r-2)} + \norm{\bff{\theta}}{\bb{H}^2}^2 \big)
		+ 
		\epsilon \norm{\bff{\chi}}{\bb{L}^2}^2.
	\end{align}
Furthermore, if the triangulation $\cal{T}_h$ is quasi-uniform, then
\begin{align}
	\label{equ:A uh A tilde uh s3}
	\big| \cal{A}(\bff{u}; \widetilde{\bff{u}}_h, \bff{\chi})
	- 
	\cal{A}(\bff{u}_h; \widetilde{\bff{u}}_h, \bff{\chi}) \big|
	&\lesssim
	\big( 1+ \norm{\bff{u}_h}{\bb{H}^2}^2 \big) 
	\big( h^{2r} + \norm{\bff{\theta}}{\bb{H}^1}^2 \big)
	+ 
	\epsilon \norm{\bff{\chi}}{\bb{L}^2}^2,
\end{align}
for any $\epsilon>0$, where the constants are independent of $h$.
\end{lemma}

\begin{proof}
Denoting
\[
R := \cal{A}(\bff{u}; \widetilde{\bff{u}}_h, \bff{\chi})
- 
\cal{A}(\bff{u}_h; \widetilde{\bff{u}}_h, \bff{\chi}),
\]
we have
	\begin{align}\label{equ:A u uh tilde}
		\nonumber
		R
		&=
		\cal{B}(\bff{u},\bff{u}; \widetilde{\bff{u}}_h, \bff{\chi}) - \cal{B}(\bff{u}_h, \bff{u}_h; \widetilde{\bff{u}}_h,\bff{\chi}) 
		+
		\cal{C}(\bff{u}; \widetilde{\bff{u}}_h, \bff{\chi}) - \cal{C}(\bff{u}_h;\widetilde{\bff{u}}_h, \bff{\chi}) 
		\\
		\nonumber
		&=
		\beta_3 \inpro{ (|\bff{u}_h|^2- |\bff{u}|^2) \widetilde{\bff{u}}_h}{ \bff{\chi} }_{\bb{L}^2}
		+
		\beta_4 \inpro{ (\bff{u}_h-\bff{u}) \times \nabla \widetilde{\bff{u}}_h}{\nabla \bff{\chi} }_{\bb{L}^2} 
		+
		\beta_5 \inpro{ (|\bff{u}|^2- |\bff{u}_h|^2) \nabla \widetilde{\bff{u}}_h}{ \nabla \bff{\chi} }_{\bb{L}^2}
		\\
		&\quad
		+ 
		2\beta_5 \inpro{ \bff{u} ( (\bff{u}-\bff{u}_h) \cdot \nabla \widetilde{\bff{u}}_h) }{ \nabla \bff{\chi}}_{\bb{L}^2}  
		+ 
		2 \beta_5 \inpro{(\bff{u}-\bff{u}_h) (\bff{u}_h \cdot \nabla \widetilde{\bff{u}}_h) }{ \nabla \bff{\chi} }_{\bb{L}^2}.
	\end{align}
Therefore, by using \eqref{equ:uh u xi rho}, Proposition \ref{pro:auxestimate},
Lemma~\ref{lem:auxbound}, and H\"{o}lder's inequality, we have
	\begin{align*}
		|R|
		&\leq
		\beta_3 \norm{\bff{u}+\bff{u}_h}{\bb{L}^\infty} \norm{\bff{\theta} + \bff{\rho}}{\bb{L}^2} 
		\norm{\widetilde{\bff{u}}_h}{\bb{L}^\infty}
		\norm{\bff{\chi}}{\bb{L}^2} 
		+ 
		\beta_4 \norm{\bff{\theta} + \bff{\rho}}{\bb{L}^2}
		\norm{\nabla \widetilde{\bff{u}}_h}{\bb{L}^4}
		\norm{\nabla \bff{\chi}}{\bb{L}^4} 
		\\
		&\quad
		+ 
		\beta_5 \norm{\bff{u}+\bff{u}_h}{\bb{L}^\infty}
		\norm{\bff{\theta} + \bff{\rho}}{\bb{L}^2} 
		\norm{\nabla \widetilde{\bff{u}}_h}{\bb{L}^4}
		\norm{\nabla \bff{\chi}}{\bb{L}^4}
		+
		2\beta_5 \norm{\bff{u}}{\bb{L}^\infty} 
		\norm{\bff{\theta} + \bff{\rho}}{\bb{L}^2}
		\norm{\nabla \widetilde{\bff{u}}_h}{\bb{L}^4}
		\norm{\nabla \bff{\chi}}{\bb{L}^4}
		\\
		&\quad
		+ 
		2\beta_5 \norm{\bff{\theta} + \bff{\rho}}{\bb{L}^2}
		\norm{\bff{u}_h}{\bb{L}^\infty}
		\norm{\nabla \widetilde{\bff{u}}_h}{\bb{L}^4}
		\norm{\nabla \bff{\chi}}{\bb{L}^4}
		\\
		&\lesssim
		\norm{\bff{\rho}}{\bb{L}^2} \norm{\bff{\chi}}{\bb{L}^2} 
		+ 
		\norm{\bff{\theta}}{\bb{L}^2} \norm{\bff{\chi}}{\bb{L}^2} 
		+
		\norm{\bff{u}_h}{\bb{L}^\infty} \norm{\bff{\rho}}{\bb{L}^2} \norm{\bff{\chi}}{\bb{L}^2}
		+
		\norm{\bff{u}_h}{\bb{L}^\infty} \norm{\bff{\theta}}{\bb{L}^2} \norm{\bff{\chi}}{\bb{L}^2}
		\\
		&\quad
		+
		\norm{\bff{\rho}}{\bb{L}^2} 
		\norm{\nabla \bff{\chi}}{\bb{L}^4}
		+ 
		\norm{\bff{\theta}}{\bb{L}^2} 
		\norm{\nabla \bff{\chi}}{\bb{L}^4}
		+
		\norm{\bff{u}_h}{\bb{L}^\infty} \norm{\bff{\rho}}{\bb{L}^2} \norm{\nabla \bff{\chi}}{\bb{L}^4}
		+
		\norm{\bff{u}_h}{\bb{L}^\infty} \norm{\bff{\theta}}{\bb{L}^2} \norm{\nabla \bff{\chi}}{\bb{L}^4}
		\\
		&\lesssim
		h^{2(r+1)}
		+ 
		\big(1+ \norm{\bff{u}_h}{\bb{H}^2}^2 \big) 
		\big( \norm{\bff{\theta}}{\bb{L}^2}^2 + \norm{\bff{\chi}}{\bb{L}^2}^2 \big)
		+ 
		\epsilon \norm{\bff{\chi}}{\bb{H}^2}^2,
	\end{align*}
where in the last step we used \eqref{equ:rho}, Young's inequality, and Sobolev embedding. This proves \eqref{equ:A uh A tilde uh}.

Next, we will show \eqref{equ:A uh A tilde uh H2}. Using integration by parts, we can write \eqref{equ:A u uh tilde} as
\begin{align*}
	R &=
	\beta_3 \inpro{ \big( (\bff{u}_h+ \bff{u}) \cdot (\bff{\theta}+\bff{\rho}) \big) \widetilde{\bff{u}}_h}{ \bff{\chi} }_{\bb{L}^2}
	-
	\beta_4 \inpro{ \nabla \cdot \big( (\bff{\theta}+\bff{\rho}) \times \nabla \widetilde{\bff{u}}_h \big)}{\bff{\chi} }_{\bb{L}^2} 
	\\
	&\quad
	-
	\beta_5 \inpro{ \nabla \cdot \big( (\bff{u}+ \bff{u}_h ) \cdot (\bff{\theta}+\bff{\rho}) \nabla \widetilde{\bff{u}}_h \big) }{\bff{\chi} }_{\bb{L}^2}
	-
	2\beta_5 \inpro{\nabla\cdot \big( \bff{u} ( (\bff{\theta}+\bff{\rho}) \cdot \nabla \widetilde{\bff{u}}_h) \big) }{ \bff{\chi}}_{\bb{L}^2} 
	\\
	&\quad 
	- 
	2 \beta_5 \inpro{\nabla \cdot \big( (\bff{\theta}+\bff{\rho}) (\bff{u}_h \cdot \nabla \widetilde{\bff{u}}_h) \big) }{ \bff{\chi} }_{\bb{L}^2}.
\end{align*}
H\"{o}lder's inequality yields
\begin{align}\label{equ:TTT} 
|R| 
&\le
\beta_3 \norm{\bff{u}+\bff{u}_h}{\bb{L}^\infty} \norm{\bff{\theta} + \bff{\rho}}{\bb{L}^2} 
	\norm{\widetilde{\bff{u}}_h}{\bb{L}^\infty}
	\norm{\bff{\chi}}{\bb{L}^2} 
	+ 
	\beta_4 \norm{(\bff{\theta} + \bff{\rho}) \times \nabla \widetilde{\bff{u}}_h}{\bb{H}^1}
	\norm{\bff{\chi}}{\bb{L}^2} 
	\nonumber\\
	&\quad
	+ 
	\beta_5 \norm{(\bff{u}+\bff{u}_h) \cdot (\bff{\theta} + \bff{\rho}) \nabla \widetilde{\bff{u}}_h}{\bb{H}^1}
	\norm{\bff{\chi}}{\bb{L}^2}
	+
	2\beta_5 \norm{\bff{u} \big((\bff{\theta} + \bff{\rho}) \cdot \nabla \widetilde{\bff{u}}_h \big)}{\bb{H}^1} 
	\norm{\bff{\chi}}{\bb{L}^2}
	\nonumber\\
	&\quad
	+ 
	2\beta_5 \norm{(\bff{\theta} + \bff{\rho}) (\bff{u}_h \cdot \nabla \widetilde{\bff{u}}_h)}{\bb{H}^1} 
	\norm{\bff{\chi}}{\bb{L}^2}
	\nonumber\\
	&= R_1 + \cdots + R_5.
\end{align}
It is easy to see that
\begin{align*} 
	R_1
	&\lesssim
	\big( 1 + \norm{\bff{u}_h}{\bb{L}^\infty} \big)
	\big( \norm{\bff{\rho}}{\bb{L}^2} + \norm{\bff{\theta}}{\bb{L}^2} \big)
	\norm{\bff{\chi}}{\bb{L}^2}
	\\
	&\lesssim
	\big( 1 + \norm{\bff{u}_h}{\bb{L}^\infty}^2 \big)
	\big( \norm{\bff{\rho}}{\bb{L}^2}^2 + \norm{\bff{\theta}}{\bb{L}^2}^2 \big)
	+
	\epsilon 
	\norm{\bff{\chi}}{\bb{L}^2}^2
	\\
	&\lesssim
	\big( 1 + \norm{\bff{u}_h}{\bb{H}^2}^2 \big)
	\big( h^{2(r+1)} + \norm{\bff{\theta}}{\bb{L}^2}^2 \big)
	+
	\epsilon 
	\norm{\bff{\chi}}{\bb{L}^2}^2.
\end{align*}
For the term $R_2$, we use~\eqref{equ:prod sobolev mat dot} (with $q_1=\infty$, $r_1=2$,
$q_2=r_2=4$), Lemma~\ref{lem:auxbound}, and the embedding
$\bb{H}^2\subset \bb{W}^{1,4}$ to have
\begin{align*} 
	R_2
	&\lesssim
	\big( \norm{\bff{\theta}+\bff{\rho}}{\bb{L}^\infty} \norm{\widetilde{\bff{u}}_h}{\bb{H}^2}
	+
	\norm{\bff{\theta}+\bff{\rho}}{\bb{W}^{1,4}} \norm{\widetilde{\bff{u}}_h}{\bb{W}^{1,4}} \big)
	\norm{\bff{\chi}}{\bb{L}^2}
	\\
	&\lesssim
	\big(
	\norm{\bff{\rho}}{\bb{H}^2}
	+
	\norm{\bff{\theta}}{\bb{H}^2}
	\big)
	\norm{\bff{\chi}}{\bb{L}^2}
	\lesssim
	h^{2(r-1)} + \norm{\bff{\theta}}{\bb{H}^2}^2 + \epsilon
	\norm{\bff{\chi}}{\bb{L}^2}^2.
\end{align*}
For the term $R_3$, by using \eqref{equ:prod 3 vec dot vec mat} (with $q_1=r_1=\infty$, $\alpha=2$, $q_2=q_3=r_2=r_3=\beta=6$) and the embedding~$\bb{H}^2
\subset \bb{W}^{1,6}$, we obtain
\begin{align*} 
	R_3
	&\lesssim
	\norm{\bff{u}+\bff{u_h}}{\bb{L}^\infty} 
	\norm{\bff{\theta}+\bff{\rho}}{\bb{L}^\infty} 
	\norm{\nabla \widetilde{\bff{u}}_h}{\bb{H}^1} 
	\norm{\bff{\chi}}{\bb{L}^2}
	+
	\norm{\bff{u}+\bff{u_h}}{\bb{L}^6} 
	\norm{\bff{\theta}+\bff{\rho}}{\bb{W}^{1,6}} 
	\norm{\nabla \widetilde{\bff{u}}_h}{\bb{L}^6} 
	\norm{\bff{\chi}}{\bb{L}^2}
	\\
	&\quad
	+
	\norm{\bff{u}+\bff{u_h}}{\bb{W}^{1,6}} 
	\norm{\bff{\theta}+\bff{\rho}}{\bb{L}^{6}} 
	\norm{\nabla \widetilde{\bff{u}}_h}{\bb{L}^6} 
	\norm{\bff{\chi}}{\bb{L}^2}
	\\
	&\lesssim
	\big(1 + \norm{\bff{u}_h}{\bb{H}^2}^2 \big)
	\big( h^{2(r-1)} + \norm{\bff{\theta}}{\bb{H}^2}^2 \big)
	+ \epsilon \norm{\bff{\chi}}{\bb{L}^2}^2.
\end{align*}
By using \eqref{equ:prod 3 vec dot mat}, we infer the same estimate for $R_4$ and $R_5$ as for $R_3$. Altogether, we deduce~\eqref{equ:A uh A tilde uh H2} from~\eqref{equ:TTT}.

Finally, we will show \eqref{equ:A uh A tilde uh s3}. If the triangulation $\cal{T}_h$ is quasi-uniform, then we can estimate \eqref{equ:TTT} as follows. Firstly, the term $R_1$ is estimated as before. For the term $R_2$, similar argument (noting \eqref{equ:aux H3} and the embedding $\bb{W}^{1,4}\subset \bb{L}^\infty$) yields
\begin{align*} 
	R_2
	&\lesssim
	\big( \norm{\bff{\theta}+\bff{\rho}}{\bb{L}^4} \norm{ \nabla \widetilde{\bff{u}}_h}{\bb{W}^{1,4}}
	+
	\norm{\bff{\theta}+\bff{\rho}}{\bb{H}^{1}} \norm{ \nabla \widetilde{\bff{u}}_h}{\bb{L}^{\infty}} \big)
	\norm{\bff{\chi}}{\bb{L}^2}
	\\
	&\lesssim
	\big(
	\norm{\bff{\rho}}{\bb{H}^1}
	+
	\norm{\bff{\theta}}{\bb{H}^1}
	\big)
	\norm{\widetilde{\bff{u}}_h}{\bb{W}^{2,4}}
	\norm{\bff{\chi}}{\bb{L}^2}
	\lesssim
	h^{2r} + \norm{\bff{\theta}}{\bb{H}^1}^2 
	+ 
	\epsilon \norm{\bff{\chi}}{\bb{L}^2}^2.
\end{align*}
For the term $R_3$, by using \eqref{equ:prod 3 vec dot vec mat} and \eqref{equ:aux H3}, again noting the embedding~$\bb{W}^{1,4} \subset \bb{L}^\infty$ and $\bb{H}^2 \subset \bb{W}^{1,6}$, we obtain
\begin{align*} 
	R_3
	&\lesssim
	\norm{\bff{u}+\bff{u_h}}{\bb{L}^6} 
	\norm{\bff{\theta}+\bff{\rho}}{\bb{L}^6} 
	\norm{\nabla \widetilde{\bff{u}}_h}{\bb{W}^{1,6}} 
	\norm{\bff{\chi}}{\bb{L}^2}
	+
	\norm{\bff{u}+\bff{u_h}}{\bb{L}^\infty} 
	\norm{\bff{\theta}+\bff{\rho}}{\bb{H}^{1}} 
	\norm{\nabla \widetilde{\bff{u}}_h}{\bb{L}^\infty} 
	\norm{\bff{\chi}}{\bb{L}^2}
	\\
	&\quad
	+
	\norm{\bff{u}+\bff{u_h}}{\bb{W}^{1,6}} 
	\norm{\bff{\theta}+\bff{\rho}}{\bb{L}^{6}} 
	\norm{\nabla \widetilde{\bff{u}}_h}{\bb{L}^6} 
	\norm{\bff{\chi}}{\bb{L}^2}
	\\
	&\lesssim
	\big(1 + \norm{\bff{u}_h}{\bb{H}^2}^2 \big)
	\big( h^{2r} + \norm{\bff{\theta}}{\bb{H}^1}^2 \big)
	+ \epsilon \norm{\bff{\chi}}{\bb{L}^2}^2.
\end{align*}
The same estimate as for $R_3$ holds for $R_4$ and $R_5$. Altogether, \eqref{equ:A uh A tilde uh s3} then follows from \eqref{equ:TTT}. This completes the proof of the lemma.
\end{proof}

We can now prove the error estimate for the semi-discrete scheme.

\begin{proposition}\label{pro:est theta}
Let $\bff{\theta}(t)= \bff{u}_h(t)- \widetilde{\bff{u}}_h(t)$ be defined
in \eqref{equ:xi rho} with $\bff{u}\in\bb{H}^4$. Then 
	\begin{equation}\label{equ:theta L2}
		\norm{\bff{\theta}}{L^\infty(\bb{L}^2)} 
		\leq
		Ch^{r+1}.
	\end{equation}
	Moreover, we have the following convergence.
	\begin{enumerate}
		\renewcommand{\labelenumi}{\theenumi}
		\renewcommand{\theenumi}{{\rm (\roman{enumi})}}
		\item For $d=1$ or $d=2$, we have
		\begin{equation}\label{equ:theta H1 H2}
			\norm{\nabla \bff{\theta}}{L^\infty(\bb{L}^2)} 
			+ 
			h \norm{\Delta\bff{\theta}}{L^\infty(\bb{L}^2)} 
			\leq Ch^r.
		\end{equation}
		Furthermore, if the triangulation $\cal{T}_h$ is quasi-uniform, then we have
		\begin{equation}\label{equ:theta quasi}
			\norm{\nabla \bff{\theta}}{L^\infty(\bb{L}^2)} 
			+ 
			h^{1/2} \norm{\Delta\bff{\theta}}{L^\infty(\bb{L}^2)} 
			\leq 
			Ch^{r+1/2},
		\end{equation}
		where the constants depend on $ \kappa $, $\mu$, $T$, and $K$.
		\item For $d=3$, estimate~\eqref{equ:theta H1 H2} holds if the triangulation
		$\cal{T}_h$ is quasi-uniform.
		The estimate~\eqref{equ:theta quasi} also hold if additionally
		the initial data satisfies $\norm{\bff{u}_0}{\bb{H}^2}
		\lesssim T^{-2/3}$.
	\end{enumerate}
\end{proposition}

\begin{proof}
By \eqref{equ:weaksemidisc} and \eqref{equ:weakform}, it follows that for any $\bff{\chi}\in \bb{V}_h$,
\begin{align}\label{equ:subtract weakform}
	\nonumber
	&\inpro{\partial_t \bff{\theta}}{ \bff{\chi} }_{\bb{L}^2} 
	+ 
	\cal{A}(\bff{u}_h; \bff{\theta}, \bff{\chi})
	\\ 
	\nonumber
	&= 
	\inpro{ \partial_t \bff{u}_h }{ \bff{\chi}}_{\bb{L}^2} 
	- 
	\inpro{ \partial_t \widetilde{\bff{u}}_h}{\bff{\chi}}_{\bb{L}^2}
	+ 
	\cal{A}(\bff{u}_h; \bff{u}_h, \bff{\chi}) 
	- 
	\cal{A}(\bff{u}_h; \widetilde{\bff{u}}_h, \bff{\chi}) 
	\\
	\nonumber
	&= 
	(\alpha+ \beta_3) \inpro{ \bff{u}_h}{ \bff{\chi}}_{\bb{L}^2} 
	+
	\beta_6 \inpro{(\bff{j}\cdot\nabla)\bff{u}_h}{\bff{\chi}}_{\bb{L}^2}
	- 
	\inpro{ \partial_t \bff{\rho}}{\bff{\chi}}_{\bb{L}^2}
	- 
	\inpro{ \partial_t \bff{u}}{ \bff{\chi}}_{\bb{L}^2}
	- 
	\cal{A}(\bff{u}_h; \widetilde{\bff{u}}_h, \bff{\chi}) 
	\\
	\nonumber
	&= 
	(\alpha+ \beta_3) \inpro{ \bff{u}_h}{\bff{\chi}}_{\bb{L}^2} 
	+
	\beta_6 \inpro{(\bff{j}\cdot\nabla)\bff{u}_h}{\bff{\chi}}_{\bb{L}^2}
	- 
	\inpro{\partial_t \bff{\rho}}{ \bff{\chi}}_{\bb{L}^2}
	+ 
	\cal{A}(\bff{u}; \bff{u}, \bff{\chi}) 
	- 
	(\alpha+ \beta_3) \inpro{ \bff{u}}{\bff{\chi}}_{\bb{L}^2} 
	- 
	\cal{A}(\bff{u}_h; \widetilde{\bff{u}}_h, \bff{\chi}) 
	\\
	&= 
	(\alpha+ \beta_3) \inpro{ \bff{u}_h-\bff{u}}{ \bff{\chi} }_{\bb{L}^2} 
	-
	\beta_6 \inpro{(\bff{u}_h-\bff{u}) \otimes \bff{\nu}}{\nabla\bff{\chi}}_{\bb{L}^2}
	- 
	\inpro{ \partial_t \bff{\rho}}{ \bff{\chi}}_{\bb{L}^2}
	+
	\cal{A}(\bff{u}; \widetilde{\bff{u}}_h, \bff{\chi})
	- 
	\cal{A}(\bff{u}_h; \widetilde{\bff{u}}_h, \bff{\chi}),
\end{align}
where in the last step we used $\bff{u}=\widetilde{\bff{u}}_h-\bff{\rho}$
and~\eqref{equ:auxiliary}, as well as~\eqref{equ:j nab u} and the assumptions on $\bff{j}$.

We will now derive a bound for $\norm{\bff{\theta}}{\bb{L}^2}$. To this end, take $\bff{\chi}= \bff{\theta}$ in the above equality.
By using the coercivity of~$\cal{A}$, H\"{o}lder's inequality, and
\eqref{equ:A uh A tilde uh}, we obtain
\begin{align*}
	&\frac{1}{2} \ddt \norm{\bff{\theta}}{\bb{L}^2}^2 
	+ 
	\mu \norm{\bff{\theta}}{\bb{H}^2}^2  
	\\
	&\leq 
	\big|(\alpha+\beta_3) \inpro{ \bff{u}_h -\bff{u}}{ \bff{\theta} }_{\bb{L}^2} \big|
	+
	\big| \beta_6 \inpro{(\bff{u}_h-\bff{u}) \otimes \bff{\nu}}{\nabla\bff{\theta}}_{\bb{L}^2} \big| 
	+
	\big| \inpro{\partial_t \bff{\rho}}{ \bff{\theta}}_{\bb{L}^2} \big|
	+ 
	\big| \cal{A}(\bff{u}; \widetilde{\bff{u}}_h, \bff{\chi})
	- 
	\cal{A}(\bff{u}_h; \widetilde{\bff{u}}_h, \bff{\chi}) 
	\big|
	\\
	&\lesssim
	\norm{\bff{\rho}}{\bb{L}^2} \norm{\bff{\theta}}{\bb{H}^1}
	+
	\norm{\bff{\theta}}{\bb{L}^2}^2
	+
	\norm{\partial_t \bff{\rho}}{\bb{L}^2} \norm{\bff{\theta}}{\bb{L}^2}
	+
	h^{2(r+1)}
	+ 
	\big(1+ \norm{\bff{u}_h}{\bb{L}^\infty}^2 \big) \norm{\bff{\theta}}{\bb{L}^2}^2 
	+ 
	\epsilon \norm{\bff{\theta}}{\bb{H}^2}^2
	\\
	&\lesssim
	h^{2(r+1)}
	+ 
	\big(1+ \norm{\bff{u}_h}{\bb{L}^\infty}^2 \big) \norm{\bff{\theta}}{\bb{L}^2}^2 
	+ 
	\epsilon \norm{\bff{\theta}}{\bb{H}^2}^2,
\end{align*}
where in the last step we also used Young's inequality, \eqref{equ:rho},
and \eqref{equ:partial rho}. Integrating over $(0,t)$, choosing $ \epsilon $
sufficiently small, and rearranging the equation, we obtain
\begin{align*}
	\norm{\bff{\theta}(t)}{\bb{L}^2}^2 
	+ 
	\int_0^t \norm{ \bff{\theta}(\tau)}{\bb{H}^2}^2 \dtau
	&\leq 
	Ch^{2(r+1)}
	+ 
	\norm{\bff{\theta}(0)}{\bb{L}^2}^2
	+
	\int_0^t \big(1+ \norm{\bff{u}_h(\tau)}{\bb{L}^\infty}^2 \big) \norm{\bff{\theta}(\tau)}{\bb{L}^2}^2 \dtau.
\end{align*}
Gronwall's inequality and Sobolev embedding then yield 
\begin{align*} 
	\norm{\bff{\theta}(t)}{\bb{L}^2}^2 
	+ 
	\int_0^t \norm{ \bff{\theta}(\tau)}{\bb{H}^2}^2 \dtau
	&\leq
	Ch^{2(r+1)}
	\exp\Big(
	\int_0^t \big(1+ \norm{\bff{u}_h(\tau)}{\bb{H}^2}^2 \big) \dtau
	\Big)
	\lesssim h^{2(r+1)},
\end{align*}
where in the last step we used Proposition \ref{pro:semidisc-est1}.
This proves the inequality for the first term in~\eqref{equ:theta L2}.

Next we want to derive a bound for
$\norm{\Delta \bff{\theta}}{\bb{L}^2}$ by taking $\bff{\chi}= \partial_t
\bff{\theta}$ in \eqref{equ:subtract weakform}. To this end, first note that
using \eqref{equ:bilinear}, we can write
\begin{align*}
	\cal{A}(\bff{u}_h; \bff{\theta}, \partial_t \bff{\theta})
	&=
	\frac{\alpha}{2} \ddt \norm{\bff{\theta}}{\bb{L}^2}^2
	+
	\frac{\beta_1}{2} \ddt \norm{\nabla \bff{\theta}}{\bb{L}^2}^2
	+
	\frac{\beta_2}{2} \ddt \norm{\Delta \bff{\theta}}{\bb{L}^2}^2
	+
	\cal{B}(\bff{u}_h, \bff{u}_h; \bff{\theta}, \partial_t \bff{\theta})
	+
	\cal{C}(\bff{u}_h; \bff{\theta}, \partial_t \bff{\theta}).
\end{align*}
Thus, \eqref{equ:subtract weakform} with $\bff{\chi}= \partial_t \bff{\theta}$ and
the above equation yield, after rearranging the equation,
\begin{align}\label{equ:dt theta H2 eqn}
	\nonumber
	&\norm{\partial_t \bff{\theta}}{\bb{L}^2}^2
	+
	\frac{\alpha}{2} \ddt \norm{\bff{\theta}}{\bb{L}^2}^2
	+
	\frac{\beta_1}{2} \ddt \norm{\nabla \bff{\theta}}{\bb{L}^2}^2
	+
	\frac{\beta_2}{2} \ddt \norm{\Delta \bff{\theta}}{\bb{L}^2}^2
	\\
	\nonumber
	&=
	(\alpha+\beta_3) \inpro{\bff{\theta}+\bff{\rho}}{\partial_t \bff{\theta}}_{\bb{L}^2}
	+
	\beta_6 \inpro{(\bff{j}\cdot\nabla)(\bff{\theta}+\bff{\rho})}{\partial_t \bff{\theta}}_{\bb{L}^2}
	-
	\inpro{\partial_t \bff{\rho}}{\partial_t \bff{\theta}}_{\bb{L}^2}
	-
	\cal{B}(\bff{u}_h, \bff{u}_h; \bff{\theta}, \partial_t \bff{\theta})
	\\
	\nonumber
	&\quad
	-
	\cal{C}(\bff{u}_h; \bff{\theta}, \partial_t \bff{\theta})
	+
	\cal{A}(\bff{u}; \widetilde{\bff{u}}_h, \partial_t \bff{\theta})
	-
	\cal{A}(\bff{u}_h; \widetilde{\bff{u}}_h, \partial_t \bff{\theta})
	\\
	\nonumber
	&\lesssim
	\norm{\bff{\theta} + \bff{\rho}}{\bb{H}^1} \norm{\partial_t \bff{\theta}}{\bb{L}^2}
	+
	\norm{\partial_t \bff{\rho}}{\bb{L}^2} \norm{\partial_t \bff{\theta}}{\bb{L}^2}
	+
	\big| \cal{B}(\bff{u}_h, \bff{u}_h; \bff{\theta}, \partial_t \bff{\theta}) \big|
	+
	\big| \cal{C}(\bff{u}_h; \bff{\theta}, \partial_t \bff{\theta}) \big|
	\\
	&\quad
	+
	\big| \cal{A}(\bff{u}; \widetilde{\bff{u}}_h, \partial_t \bff{\theta})
	-
	\cal{A}(\bff{u}_h; \widetilde{\bff{u}}_h, \partial_t \bff{\theta})
	\big|.
\end{align}
By using the corresponding estimates for~$\bff{\rho}$, $\partial_t\bff{\rho}$,
$\bff{\theta}$, and $\partial_t\bff{\theta}$, together with~\eqref{equ:B bounded
H2}, \eqref{equ:C bounded H2}, \eqref{equ:A uh A tilde uh H2}, and Young's inequality, we obtain
\begin{align}\label{equ:dt the} 
	\norm{\partial_t \bff{\theta}}{\bb{L}^2}^2
	&+
	\frac{\alpha}{2} \ddt \norm{\bff{\theta}}{\bb{L}^2}^2
	+
	\frac{\beta_1}{2} \ddt \norm{\nabla \bff{\theta}}{\bb{L}^2}^2
	+
	\frac{\beta_2}{2} \ddt \norm{\Delta \bff{\theta}}{\bb{L}^2}^2
	\lesssim
	h^{2(r-1)}
	+ 
	\norm{\bff{u}_h}{\bb{H}^2}^4
	\norm{\bff{\theta}}{\bb{H}^2}^2 
	+ 
	\epsilon \norm{\partial_t \bff{\theta}}{\bb{L}^2}^2.
\end{align}
The term~$\norm{\bff{u}_h}{\bb{H}^2}^4$ is bounded due to
Proposition~\ref{pro:semidisc est2}. Thus we deduce, after integrating, choosing sufficiently
small $ \epsilon>0 $, and using~\eqref{equ:u0 u0h},
\begin{align*}
	\norm{\bff{\theta}(t)}{\bb{H}^2}^2 
	+ 
	\int_0^t \norm{\partial_\tau \bff{\theta}(\tau)}{\bb{L}^2}^2 \dtau
	&\lesssim
	h^{2(r-1)}
	+
	\norm{\bff{\theta}(0)}{\bb{L}^2}^2
	+
	\int_{0}^{t} 
	\norm{\bff{\theta}(\tau)}{\bb{H}^2}^2 \dtau
	\lesssim
	h^{2(r-1)} 
	+
	\int_{0}^{t} 
	\norm{\bff{\theta}(\tau)}{\bb{H}^2}^2 \dtau.
\end{align*}
Gronwall's inequality yields the required estimate
$\norm{\bff{\theta}(t)}{\bb{H}^2} \leq Ch^{r-1}$. Interpolation then
yields~\eqref{equ:theta L2}.

Finally, we prove \eqref{equ:theta quasi}. In this case, continuing from
\eqref{equ:dt theta H2 eqn}, and proceeding in the same manner as before (but using
\eqref{equ:A uh A tilde uh s3} instead of \eqref{equ:A uh A tilde uh H2}), we
obtain, similarly to~\eqref{equ:dt the},
\begin{align*} 
	\norm{\partial_t \bff{\theta}}{\bb{L}^2}^2
	+
	\frac{\alpha}{2} \ddt \norm{\bff{\theta}}{\bb{L}^2}^2
	+
	\frac{\beta_1}{2} \ddt \norm{\nabla \bff{\theta}}{\bb{L}^2}^2
	+
	\frac{\beta_2}{2} \ddt \norm{\Delta \bff{\theta}}{\bb{L}^2}^2
	&\lesssim
	h^{2r} 
	+ 
	\norm{\bff{u}_h}{\bb{H}^2}^4
	\norm{\bff{\theta}}{\bb{H}^2}^2 
	+ 
	\epsilon \norm{\partial_t \bff{\theta}}{\bb{L}^2}^2.
\end{align*}
Similar argument then yields the estimate~\eqref{equ:theta quasi}, completing the proof of the proposition.
\end{proof}

\begin{remark}
Inequality~\eqref{equ:theta quasi} shows an estimate of $\nabla \bff{\theta}$ and
$\Delta \bff{\theta}$ to a superconvergent order (compared to that of $\nabla
\bff{\rho}$ and $\Delta \bff{\rho}$) if the triangulation $\cal{T}_h$ is
quasi-uniform.
\\
As a consequence, for $d=2$, by using the Sobolev embedding and the discrete Sobolev inequality \cite[Lemma~4.9.1]{BreSco08}, we also have the maximum norm estimates
\begin{align*}
	\norm{\bff{\theta}}{L^\infty(\bb{L}^{\infty})}
	&\lesssim
	h^r,
	\\
	\norm{\nabla \bff{\theta}}{L^\infty(\bb{L}^\infty)}
	&\lesssim
	h^r \abs{\log h}^{1/2}.
\end{align*}
\end{remark}

We are now ready to state and prove the main result of this section.
\begin{theorem}\label{the:semidisc error}
	Let $\bff{u}$ be the solution of \eqref{equ:weakform} which
	satisfies~\eqref{equ:ass 1}, and let~$\bff{u}_h$ be the solution of
	\eqref{equ:weaksemidisc}. Then 
	\begin{equation}\label{equ:uhu L2}
		\norm{\bff{u}_h- \bff{u}}{L^\infty(\bb{L}^2)} 
		\leq
		Ch^{r+1},
	\end{equation}	
Furthermore, for $d=1$ or $d=2$, we have
	\begin{equation}\label{equ:uhu}
		\norm{\nabla \bff{u}_h - \nabla\bff{u}}{L^\infty(\bb{L}^2)} 
		+ 
		h \norm{\Delta\bff{u}_h - \Delta\bff{u}}{L^\infty(\bb{L}^2)} 
		\leq 
		Ch^r.
	\end{equation}
	For $d=3$, estimate~\eqref{equ:uhu} also holds provided that one of the
	following assumptions holds:
	\begin{enumerate}
	\renewcommand{\labelenumi}{\theenumi}
	\renewcommand{\theenumi}{{\rm (\roman{enumi})}}
		\item The initial data satisfies
			\begin{equation}\label{equ:d3 ass}
				\norm{\bff{u}_0}{\bb{H}^2}
				\lesssim
				T^{-2/3}.
			\end{equation}
		\item The triangulation $\cal{T}_h$ is quasi-uniform.
	\end{enumerate}
	In \eqref{equ:uhu L2} and \eqref{equ:uhu}, the constants depend on $ \kappa $, $\mu$, $T$, and $K$.
\end{theorem}

\begin{proof}
	Recall that $\bff{u}_h(t) - \bff{u}(t)= \bff{\theta}(t) +
	\bff{\rho}(t)$ as defined in \eqref{equ:xi rho}. 
	The estimates for $\bff{\rho}$ is proved in Proposition
	\ref{pro:auxestimate}, while that for $\bff{\theta}$ is in Proposition \ref{pro:est theta}. 
	The theorem then follows by the triangle inequality.
\end{proof}

\section{Time Discretisation by the Linearised Euler Method}\label{sec:semi euler}

We shall now consider a discretisation in the time variable. In the sequel,
let~$k$ be the time step and $\bff{U}^n$ be the approximation in $\bb{V}_h$ of
$\bff{u}(t)$  at time $t=t_n:=nk$, $n=0,1,2,\ldots$. We denote $ \bff{u}^n :=
\bff{u}(t_n) $ and define, for any discrete vector-valued function $\bff{v}^n $,
\[
	\delta \bff{v}^n
	:=
	\frac{\bff{v}^{n}-\bff{v}^{n-1}}{k}, \quad n=1,2,3,\ldots.
\]
Note that under the
assumption~\eqref{equ:ass 1}, we have
\begin{equation}\label{equ:del un}
	\norm{ \delta\bff{u}^n}{\bb{L}^p}
	=
	\norm{ \frac{1}{k} \int_{t_{n-1}}^{t_n} \partial_t \bff{u}(t) \,
	\dt}{\bb{L}^p}
	\le C, \quad 1 \le p \le \infty, \ n=1,2,3,\ldots.
\end{equation}

We will describe a time discretisation scheme using a semi-implicit (linearised) Euler method as follows. We start with~$\bff{U}^0 =
\widetilde{\bff{u}}_h(0) \in \bb{V}_h$ where $\widetilde{\bff{u}}_h(0)$ is
defined by~\eqref{equ:auxiliary}.  For $t_n\in [0,T]$, given $\bff{U}^{n-1} \in
\bb{V}_h$, define $\bff{U}^n$ by (recalling the definition of~$\cal{A(\cdot\,; \cdot\, ,\cdot)}$ and $ \alpha$
in~\eqref{equ:bilinear})
\begin{align}\label{equ:backwardeuler}
	&\inpro{ \delta \bff{U}^n}{ \bff{\chi} }_{\bb{L}^2}
	+ 
	\cal{A}(\bff{U}^{n-1}; \bff{U}^n, \bff{\chi})
	-
	(\alpha+\beta_3) \inpro{\bff{U}^n}{\bff{\chi}}_{\bb{L}^2}
	-
	\beta_6 \inpro{(\bff{j}\cdot \nabla)\bff{U}^n}{\bff{\chi}}_{\bb{L}^2}
	=
	0, \quad
	\bff{\chi} \in \bb{V}_h.
\end{align}
Note that this fully discrete scheme is linear at each time step.

Before analysing the above scheme, we first clarify the notion of stability that is used subsequently.
Since equation \eqref{equ:llbar} is not a gradient flow, we will adopt the definition of stability for a fully discrete scheme given in \cite{SchBer11}.

\begin{definition}\label{def:stable}
A discrete time-stepping scheme is \emph{unconditionally stable} in the $\bb{H}^\beta$ norm if for any $k\leq \nu$, where $\nu$ depends only on the coefficients of the equation,
\begin{equation}\label{equ:stable bdd}
\norm{\bff{U}^n}{\bb{H}^\beta} \leq C, \quad \forall n\in \{1,2,\ldots,\lfloor T/k \rfloor\}.
\end{equation}
The constant $C$ is independent of $n$ and $k$.

The method is \emph{conditionally stable} if \eqref{equ:stable bdd} holds under a restriction on $k$ in terms of $h$.
\end{definition}

First, we show that the above scheme is well-defined. For the stability estimate, we need to
introduce the following constant which is defined from the coefficients in~\eqref{equ:llbar}:
	\begin{equation}\label{equ:k con}
	\lambda :=
	\begin{cases}
		2\beta_1 / \left( \beta_6^2 + 2\beta_1 \beta_3 \right) &\quad \text{if $\beta_1\ge0$},
		\\
		4\beta_2 / \left( \beta_1^2 + \beta_6^2+ 4\beta_2 \beta_3 \right) &\quad \text{if
		$\beta_1<0$}.
	\end{cases}
	\end{equation}

\begin{proposition}\label{pro:backeuler exist}
	For $n\in \bb{N}$, given $\bff{U}^{n-1} \in \bb{V}_h$ and 
	$k<\lambda$,
	there exists a unique $\bff{U}^n \in \bb{V}_h$ that solves the fully
	discrete scheme \eqref{equ:backwardeuler}.
\end{proposition}

\begin{proof}
For each $\bff{\phi}\in \bb{V}_h$, define a bilinear form $\cal{S}(\bff{\phi};\, \cdot \, , \, \cdot): \bb{V}_h \times \bb{V}_h \to \bb{R}$ by
\begin{align*}
	\nonumber
	\cal{S}(\bff{\phi}; \bff{v}, \bff{w})
	:=
	(1-k(\alpha+\beta_3)) \inpro{\bff{v}}{\bff{w}}_{\bb{L}^2}
	-
	k\beta_6\inpro{(\bff{j}\cdot\nabla)\bff{v}}{\bff{w}}_{\bb{L}^2}
	+
	k \cal{A}(\bff{\phi}; \bff{v},\bff{w}).
\end{align*}	
Note that this bilinear form is bounded. Moreover, it follows from
Lemma~\ref{lem:A bou coe} that there exists $ \alpha>0 $
such that $\cal{A}$ is coercive and $k<1/(\alpha+\beta_3+\beta_6)$, if $k<\lambda$.
Hence, $\cal{S}$ is $\bb{V}_h$-coercive.
Now, equation \eqref{equ:backwardeuler} is equivalent to
\[
	\cal{S}(\bff{U}^{n-1}; \bff{U}^n, \bff{\chi}) 
	= \inpro{\bff{U}^{n-1}}{\bff{\chi}}_{\bb{L}^2}
	\quad \text{for all } \bff{\chi}\in \bb{V}_h.
\]
Therefore, given $\bff{U}^{n-1}\in \bb{V}_h$, the proposition follows from the
Lax--Milgram lemma.
\end{proof}

Next, we show that the scheme \eqref{equ:backwardeuler} is unconditionally
stable in $\bb{L}^2$ for~$k$ satisfying
\begin{equation}\label{equ:k lam}
k<\lambda/2.
\end{equation}

\begin{proposition}\label{pro:euler stab}
Let $T>0$ be given and let $\bff{U}^n$ be defined by \eqref{equ:backwardeuler} with initial data $\bff{U}^0 \in \bb{V}_h$.
Then for $k$ satisfying~\eqref{equ:k lam} and $n\in \{1,2,\ldots, \lfloor T/k \rfloor\}$,
\begin{align}\label{equ:euler stab L2}
	\norm{\bff{U}^n}{\bb{L}^2}^2 
	+
	\sum_{m=1}^n \norm{\bff{U}^m - \bff{U}^{m-1}}{\bb{L}^2}^2
	+
	k \sum_{m=1}^n \norm{\Delta \bff{U}^m}{\bb{L}^2}^2
	\lesssim 
	\norm{\bff{U}^0}{\bb{L}^2}^2,
\end{align}
where the constant depends on $T$, but is independent of $n$ and $k$.
\end{proposition}

\begin{proof}
Taking $\bff{\chi}= \bff{U}^n$ in \eqref{equ:backwardeuler}, and using the
identity 
\begin{equation}\label{equ:ab ide}
2\bff{a}\cdot (\bff{a}-\bff{b})= |\bff{a}|^2 -
|\bff{b}|^2 +  |\bff{a}-\bff{b}|^2,
\quad\forall\bff{a}, \bff{b}\in\bb{R}^3, 
\end{equation}
we have (after multiplying by $k$)
\begin{align*}
	&\norm{\bff{U}^n}{\bb{L}^2}^2 
	- 
	\norm{\bff{U}^{n-1}}{\bb{L}^2}^2
	+ 
	\norm{\bff{U}^n - \bff{U}^{n-1}}{\bb{L}^2}^2 
	+ 
	2 \beta_2 k \norm{\Delta \bff{U}^n}{\bb{L}^2}^2 
	+
	2 \beta_3 k\norm{|\bff{U}^{n-1}| |\bff{U}^n|}{\bb{L}^2}^2
	\\
	&\qquad
	+ 
	4\beta_5 k \norm{\bff{U}^{n-1} \cdot \nabla \bff{U}^n}{\bb{L}^2}^2  
	+
	2\beta_5 k \norm{|\bff{U}^{n-1}| |\nabla \bff{U}^n|}{\bb{L}^2}^2
	\\
	&=
	- 
	2\beta_1 k \norm{\nabla \bff{U}^n}{\bb{L}^2}^2 
	+ 
	2\beta_3 k \norm{\bff{U}^n}{\bb{L}^2}^2 
	+
	2 \beta_6 k \norm{\nabla \bff{U}^n}{\bb{L}^2} \norm{\bff{U}^n}{\bb{L}^2}.
\end{align*}
If $\beta_1 \geq 0$, then we have
\begin{align*}
	&\norm{\bff{U}^n}{\bb{L}^2}^2 
	- 
	\norm{\bff{U}^{n-1}}{\bb{L}^2}^2
	+ 
	\norm{\bff{U}^n - \bff{U}^{n-1}}{\bb{L}^2}^2 
	+
	\beta_1 k \norm{\nabla \bff{U}^n}{\bb{L}^2}^2
	+ 
	2\beta_2 k \norm{\Delta \bff{U}^n}{\bb{L}^2}^2
	\le
	\left(\frac{\beta_6^2}{\beta_1} + 2\beta_3\right) k \norm{\bff{U}^n}{\bb{L}^2}^2.
\end{align*}
Otherwise, if $\beta_1<0$, then using \eqref{equ:equivnorm-nabL2} we have
\begin{align*}
	&\norm{\bff{U}^n}{\bb{L}^2}^2 
	- 
	\norm{\bff{U}^{n-1}}{\bb{L}^2}^2
	+ 
	\norm{\bff{U}^n - \bff{U}^{n-1}}{\bb{L}^2}^2 
	+ 
	2\beta_2 k \norm{\Delta \bff{U}^n}{\bb{L}^2}^2
	\\
	&\leq
	k (\beta_1^2/ 2\beta_2) \norm{\bff{U}^n}{\bb{L}^2}^2
	+
	\frac12 \beta_2 k \norm{\Delta \bff{U}^n}{\bb{L}^2}^2
	+
	2\beta_3 k \norm{\bff{U}^n}{\bb{L}^2}^2
	+
	k (\beta_6^2/2\beta_2) \norm{\bff{U}^n}{\bb{L}^2}^2
	+
	\frac12 \beta_2 k\norm{\nabla \bff{U}^n}{\bb{L}^2}^2
	\\
	&=
	\left(\frac{\beta_1^2+\beta_6^2}{2\beta_2}+ 2\beta_3\right) k \norm{\bff{U}^n}{\bb{L}^2}^2
	+
	\beta_2 k \norm{\Delta \bff{U}^n}{\bb{L}^2}^2.
\end{align*}
In any case, this implies (after changing the index $n$ to $m$ and rearranging the equation)
\[
	\norm{\bff{U}^m}{\bb{L}^2}^2 
	- 
	\norm{\bff{U}^{m-1}}{\bb{L}^2}^2
	+ 
	\norm{\bff{U}^m - \bff{U}^{m-1}}{\bb{L}^2}^2 
	+ 
	\beta_2 k \norm{\Delta \bff{U}^m}{\bb{L}^2}^2
	\le
	\frac{2k}{\lambda} \norm{\bff{U}^m}{\bb{L}^2}^2,
\]
where $\lambda$ was defined in \eqref{equ:k con}. Summing this over $m\in \{1,2,\ldots,n\}$, we have
\begin{align*}
	\left(1-\frac{2k}{\lambda}\right) \norm{\bff{U}^n}{\bb{L}^2}^2
	+
	\sum_{m=1}^n \norm{\bff{U}^m - \bff{U}^{m-1}}{\bb{L}^2}^2 
	+
	\beta_2 k\sum_{m=1}^n \norm{\Delta \bff{U}^m}{\bb{L}^2}^2
	\leq 
	\norm{\bff{U}^0}{\bb{L}^2}^2
	+
	\frac{2k}{\lambda} \sum_{m=1}^{n-1} \norm{\bff{U}^m}{\bb{L}^2}^2.
\end{align*}
With $k<\lambda/2$, we obtain the required
result after invoking the discrete Gronwall inequality.
\end{proof}

We next show that the method is also stable in the $\bb{H}^2$ norm (in the sense of Definition \ref{def:stable}). Before
doing so, we need the following definitions which are analogous
to~\eqref{equ:xi rho} and~\eqref{equ:auxiliary}.

Let~$\bff{u}$ and $\bff{U}^n$ be solutions of \eqref{equ:weakform} and
\eqref{equ:backwardeuler}, respectively. We write
\begin{equation}\label{equ:theta n plus rho n disc}
\bff{U}^n - \bff{u}^n = 
(\bff{U}^n - \widetilde{\bff{u}}_h^n) + (\widetilde{\bff{u}}_h^n - \bff{u}^n) =: \bff{\theta}^n + \bff{\rho}^n,
\end{equation}
where $\widetilde{\bff{u}}_h^n := \widetilde{\bff{u}}_h(t_n)$ is the elliptic projection of $\bff{u}^n$. More precisely, $\widetilde{\bff{u}}_h^n$ satisfies
\begin{align}\label{equ:eq elliptic Un}
	\cal{A}(\bff{u}^n; \widetilde{\bff{u}}_h^n- \bff{u}^n, \bff{\chi}) = 0
	\quad \text{ for all } \bff{\chi} \in \bb{V}_h.
\end{align}
Note that (cf. \eqref{equ:del un})
\begin{align}\label{equ:d rho n}
	\norm{\delta \bff{\rho}^n}{\bb{L}^2} 
	= 
	\norm{\frac{1}{k} \int_{t_{n-1}}^{t_n} \partial_t \bff{\rho}(t) \,\dt}{\bb{L}^2} 
	\leq 
	C h^{r+1},
\end{align}
and that by Taylor's theorem
\begin{align}\label{equ:diff un dt un}
	\norm{\delta \bff{u}^n - \partial_t \bff{u}^n}{\bb{L}^2} 
	= 
	\norm{\frac{1}{2k} \int_{t_{n-1}}^{t_n} (t- t_{n-1}) \, \partial_{t}^2 \bff{u}(t)\,\dt }{\bb{L}^2} 
	\leq 
	Ck,
\end{align}
where $C$ depends on $\kappa,\mu, T,\norm{\bff{u}}{\bb{H}^{r+1}}, \norm{\partial_t \bff{u}}{\bb{H}^{r+1}}, \norm{\partial_{t}^2 \bff{u}}{\bb{L}^2}$. 

For the analysis, we need the following lemma.

\begin{lemma}
	\label{lem:A bil for}
If $\bff{u}$ is the solution of \eqref{equ:weakform} which
satisfies~\eqref{equ:ass 1}, then
\begin{align}
	\label{equ:A Un-1 utilde n}
	\nonumber
	\big| \cal{A}(\bff{U}^{n-1}; \widetilde{\bff{u}}_h^n, \bff{\chi}) 
	- 
	\cal{A}(\bff{u}^n; \widetilde{\bff{u}}_h^n, \bff{\chi}) \big|
	&\lesssim
	\big(1+ \norm{\bff{U}^{n-1}}{\bb{H}^2}^2 \big) h^{2(r+1)} 
	+ 
	k^2
	+
	\norm{\bff{\theta}^{n-1}}{\bb{L}^2}^2
	\\
	&\quad
	+
	\epsilon \norm{\bff{\theta}^{n-1}}{\bb{H}^2}^2
	+
	\epsilon \norm{\bff{\chi}}{\bb{H}^2}^2,
	\\
	\label{equ:A Un-1 utilde n L2}
	\big| \cal{A}(\bff{U}^{n-1}; \widetilde{\bff{u}}_h^n, \bff{\chi}) 
	- 
	\cal{A}(\bff{u}^n; \widetilde{\bff{u}}_h^n, \bff{\chi}) \big|
	&\lesssim
	\big(1+ \norm{\bff{U}^{n-1}}{\bb{H}^2}^2 \big)
	\big(\norm{\bff{\theta}^{n-1}}{\bb{H}^2}^2 + h^{2(r-1)} + k^2 \big) 
	+ 
	\epsilon \norm{\bff{\chi}}{\bb{L}^2}^2
\end{align}
for any $\epsilon>0$.
\end{lemma}

\begin{proof}
Firstly, note that
\begin{equation}\label{equ:Un un}
	\bff{U}^{n-1} - \bff{u}^n 
	= 
	\bff{\theta}^{n-1} + \bff{\rho}^{n-1} - k\cdot \delta\bff{u}^n
\end{equation}
and that $\bff{\rho}^{n-1}=\bff{\rho}(t_{n-1}) $, thus we deduce from~\eqref{equ:rho} 
and \eqref{equ:del un}
\begin{equation}\label{equ:Un un L2}
	\norm{\bff{U}^{n-1} - \bff{u}^n}{\bb{L}^2}
	\lesssim 
	\norm{\bff{\theta}^{n-1}}{\bb{L}^2} 
	+ h^{r+1} + k.
\end{equation}
Inequality~\eqref{equ:Un un} also implies
\begin{equation}\label{equ:Un2 un2}
	\Big||\bff{U}^{n-1}|^2 - |\bff{u}^n|^2 \Big|
	= 
	| \bff{U}^{n-1} + \bff{u}^n | \
	| \bff{\theta}^{n-1} + \bff{\rho}^{n-1} - k\cdot \delta\bff{u}^n |,
\end{equation}
so that Proposition~\ref{pro:euler stab}, assumption~\eqref{equ:ass 1},
and~\eqref{equ:Un un L2} give
\begin{equation}\label{equ:Un2 un2 L1}
	\norm{|\bff{U}^{n-1}|^2 - |\bff{u}^n|^2}{\bb{L}^1}
	\lesssim 
	\norm{\bff{\theta}^{n-1}}{\bb{L}^2} 
	+ h^{r+1} + k.
\end{equation}
Similarly, by using Sobolev embedding $\bb{H}^2\subset \bb{L}^\infty \cap \bb{W}^{1,4}$ and \eqref{equ:prod Hs mat dot}, we deduce
\begin{equation}\label{equ:Un2 un2 Linf}
	\norm{|\bff{U}^{n-1}|^2 - |\bff{u}^n|^2}{\bb{L}^\infty}
	\lesssim 
	\big(1 + \norm{\bff{U}^{n-1}}{\bb{H}^2}\big)
	\big(\norm{\bff{\theta}^{n-1}}{\bb{H}^2} + h^{r-1} + k \big)
\end{equation}
and
\begin{equation}\label{equ:Un2 un2 W14}
	\norm{|\bff{U}^{n-1}|^2 - |\bff{u}^n|^2}{\bb{W}^{1,4}}
	\lesssim 
	\big(1 + \norm{\bff{U}^{n-1}}{\bb{H}^2}\big)
	\big(\norm{\bff{\theta}^{n-1}}{\bb{H}^2} + h^{r-1} + k \big).
\end{equation}
Furthermore, interpolation and Young's inequalities yield, for any $\epsilon>0$,
\begin{align}\label{equ:A Un-1 gal nir}
	\norm{\bff{\theta}^{n-1}}{\bb{H}^1} \norm{\bff{\chi}}{\bb{H}^2}
	\lesssim
	\norm{\bff{\theta}^{n-1}}{\bb{L}^2}^{1/2} \norm{\bff{\theta}^{n-1}}{\bb{H}^2}^{1/2}
	\norm{\bff{\chi}}{\bb{H}^2}
	\lesssim
	\norm{\bff{\theta}^{n-1}}{\bb{L}^2}^2 
	+
	\epsilon \norm{\bff{\theta}^{n-1}}{\bb{H}^2}^2
	+
	\epsilon \norm{\bff{\chi}}{\bb{H}^2}^2.
\end{align}

Now, we have
\begin{align}\label{equ:A Un-1 A un}
	\nonumber
	&\big| \cal{A}(\bff{U}^{n-1}; \widetilde{\bff{u}}_h^n, \bff{\chi}) 
	- 
	\cal{A}(\bff{u}^n; \widetilde{\bff{u}}_h^n, \bff{\chi}) \big| 
	\\
	\nonumber
	&\leq
	\big| \cal{B}(\bff{U}^{n-1},\bff{U}^{n-1}; \widetilde{\bff{u}}_h^n, \bff{\chi}) - \cal{B}(\bff{u}^n, \bff{u}^n; \widetilde{\bff{u}}_h^n,\bff{\chi}) \big| 
	+
	\big| \cal{C}(\bff{U}^{n-1}; \widetilde{\bff{u}}_h^n, \bff{\chi}) - \cal{C}(\bff{u}^n;\widetilde{\bff{u}}_h^n, \bff{\chi}) \big| 
	\\
	\nonumber
	&\leq 
	\beta_3 \big| \inpro{(|\bff{U}^{n-1}|^2 - |\bff{u}^n|^2) \widetilde{\bff{u}}_h^n}{\bff{\chi}}_{\bb{L}^2} \big| 
	+ 
	\beta_4 \big| \inpro{(\bff{U}^{n-1}- \bff{u}^n) \times \nabla \widetilde{\bff{u}}_h^n}{\nabla \bff{\chi}}_{\bb{L}^2} \big|
	\\
	\nonumber
	&\quad 
	+ 
	\beta_5 \big| \inpro{(|\bff{U}^{n-1}|^2 - |\bff{u}^n|^2) \nabla \widetilde{\bff{u}}_h^n}{\nabla \bff{\chi}}_{\bb{L}^2} \big|
	+ 
	2\beta_5 \big| \inpro{\bff{U}^{n-1} \left( (\bff{U}^{n-1}-\bff{u}^n)\cdot \nabla \widetilde{\bff{u}}_h^n \right)}{\nabla \bff{\chi}}_{\bb{L}^2} \big|
	\\
	\nonumber
	&\quad 
	+ 
	2\beta_5 \big| \inpro{(\bff{U}^{n-1} -\bff{u}^n)(\bff{u}^n \cdot \nabla \widetilde{\bff{u}}_h^n)}{\nabla \bff{\chi}}_{\bb{L}^2} \big|
	\\
	&=:
	T_1 +\cdots + T_5.
\end{align}
It follows successively from~\eqref{equ:Un2 un2 L1}, ~\eqref{equ:aux H2}, and Young's
inequality that
\begin{align*}
	T_1 
	&\leq
	\beta_3 
	\norm{|\bff{U}^{n-1}|^2 - |\bff{u}^n|^2}{\bb{L}^1} 
	\norm{\widetilde{\bff{u}}_h^n}{\bb{L}^\infty} 
	\norm{\bff{\chi}}{\bb{L}^\infty}  
	\lesssim
	\big(
	\norm{\bff{\theta}^{n-1}}{\bb{L}^2} 
	+
	h^{r+1}
	+ k
	\big)
	\norm{\widetilde{\bff{u}}_h^n}{\bb{H}^2} 
	\norm{\bff{\chi}}{\bb{H}^2}  
	\\
	&\lesssim
	\big(
	\norm{\bff{\theta}^{n-1}}{\bb{L}^2} 
	+
	h^{r+1}
	+ k
	\big)
	\norm{\bff{\chi}}{\bb{H}^2}
	\lesssim
	\norm{\bff{\theta}^{n-1}}{\bb{L}^2}^2 + h^{2(r+1)}+ k^2 
	+ 
	\epsilon \norm{\bff{\chi}}{\bb{H}^2}^2.
\end{align*}
For the term $T_2$, H\"older's inequality, the Sobolev embedding $ \bb{H}^1
\subset \bb{L}^4$, \eqref{equ:Un un L2},
\eqref{equ:aux H2}, and Young's inequality give
\begin{align*}
	T_2
	&\leq 
	\beta_4 
	\norm{\bff{U}^{n-1}-\bff{u}^n}{\bb{L}^2}
	\norm{\nabla\widetilde{\bff{u}}_h^n}{\bb{L}^4}
	\norm{\nabla \bff{\chi}}{\bb{L}^4}
	\lesssim
	\big(
	\norm{\bff{\theta}^{n-1}}{\bb{L}^2} 
	+
	h^{r+1}
	+ k
	\big)
	\norm{\bff{\widetilde{\bff{u}}}_h^n}{\bb{H}^2}
	\norm{\bff{\chi}}{\bb{H}^2}
	\\
	&\lesssim
	\norm{\bff{\theta}^{n-1}}{\bb{L}^2}^2 + h^{2(r+1)}+ k^2 
	+ 
	\epsilon \norm{\bff{\chi}}{\bb{H}^2}^2.
\end{align*}
For the term $T_3$, we use successively~\eqref{equ:Un2 un2}, H\"older's
inequality, the Sobolev embedding $ \bb{H}^1 \subset \bb{L}^6 $,
\eqref{equ:del un}, \eqref{equ:aux H2}, \eqref{equ:ass 1}, 
\eqref{equ:rho}, \eqref{equ:A Un-1 gal nir}, and Young's inequality to obtain
\begin{align*} 
	T_3
	&\le
	\beta_5
	\big|
	\inpro{\big(\bff{\theta}^{n-1}-k\cdot \delta \bff{u}^n\big)
	\cdot
	\big(\bff{U}^{n-1}+\bff{u}^n\big) \nabla\widetilde{\bff{u}}_h^n}
	{\nabla\bff{\chi}}
	\big|
	+
	\beta_5
	\big|
	\inpro{\bff{\rho}^{n-1} \cdot \big(\bff{U}^{n-1}+\bff{u}^n\big)
	\nabla\widetilde{\bff{u}}_h^n}
	{\nabla\bff{\chi}}
	\big|
	\\
	&\lesssim
	\norm{\bff{\theta}^{n-1} + k \cdot \delta \bff{u}^n}{\bb{L}^6}
	\norm{\bff{U}^{n-1}+ \bff{u}^n}{\bb{L}^2}
	\norm{\nabla \widetilde{\bff{u}}_h^n}{\bb{L}^6}
	\norm{\nabla \bff{\chi}}{\bb{L}^6}
	\\
	&\quad
	+
	\norm{\bff{\rho}^{n-1}}{\bb{L}^2}
	\norm{\bff{U}^{n-1} + \bff{u}^n}{\bb{L}^\infty}
	\norm{\nabla \widetilde{\bff{u}}_h^n}{\bb{L}^4}
	\norm{\nabla \bff{\chi}}{\bb{L}^4}
	\\
	&\lesssim
	\big( \norm{\bff{\theta}^{n-1}}{\bb{H}^1} + k \big)
	\norm{\widetilde{\bff{u}}_h^n}{\bb{H}^2}
	\norm{\bff{\chi}}{\bb{H}^2}
	+
	\norm{\bff{\rho}^{n-1}}{\bb{L}^2}
	\big(
	\norm{\bff{U}^{n-1}}{\bb{H}^2} + 1
	\big)
	\norm{\widetilde{\bff{u}}_h^n}{\bb{H}^2}
	\norm{\bff{\chi}}{\bb{H}^2}
	\\
	&\lesssim
	\big( \norm{\bff{\theta}^{n-1}}{\bb{H}^1} + k \big)
	\norm{\bff{\chi}}{\bb{H}^2}
	+
	h^{r+1}
	\big(
	\norm{\bff{U}^{n-1}}{\bb{H}^2} + 1
	\big)
	\norm{\bff{\chi}}{\bb{H}^2}
	\\
	&\lesssim
	\norm{\bff{\theta}^{n-1}}{\bb{L}^2}^2
	+
	\epsilon \norm{\bff{\theta}^{n-1}}{\bb{H}^2}^2
	+
	\epsilon \norm{\bff{\chi}}{\bb{H}^2}^2
	+ k^2
	+
	h^{2(r+1)} \big(1+ \norm{\bff{U}^{n-1}}{\bb{H}^2}^2 \big).
\end{align*}
Similarly, the terms $T_4$ and $T_5$ are estimated as follows:
\begin{align*} 
	T_4
	&\le
	2 \beta_5
	\Big|
	\inpro{ \bff{U}^{n-1} 
		\big(
			(\bff{\theta}^{n-1} - k\cdot\delta\bff{u}^n)
			\cdot\nabla\widetilde{\bff{u}}_h^n	
		\big)
	}{\nabla\bff{\chi}}_{\bb{L}^2}
	\Big|
	+
	2 \beta_5
	\Big|
	\inpro{ \bff{U}^{n-1} 
		\big(
			\bff{\rho}^{n-1}\cdot\nabla\widetilde{\bff{u}}_h^n	
		\big)
	}{\nabla\bff{\chi}}_{\bb{L}^2}
	\Big|
	\\
	&\lesssim
	\norm{\bff{U}^{n-1}}{\bb{L}^2}
	\norm{\bff{\theta}^{n-1} - k\cdot\delta\bff{u}^n}{\bb{L}^6}
	\norm{\nabla\widetilde{\bff{u}}_h^n}{\bb{L}^6}
	\norm{\nabla\chi}{\bb{L}^6}
	+
	\norm{\bff{U}^{n-1}}{\bb{L}^\infty}
	\norm{\bff{\rho}^{n-1}}{\bb{L}^2}
	\norm{\nabla\widetilde{\bff{u}}_h^n}{\bb{L}^4}
	\norm{\nabla\chi}{\bb{L}^4}
	\\
	&\lesssim
	\big(
	\norm{\bff{\theta}^{n-1}}{\bb{H}^1} + k
	\big)
	\norm{\bff{\chi}}{\bb{H}^2}
	+
	h^{r+1}
	\norm{\bff{U}^{n-1}}{\bb{H}^2}
	\norm{\bff{\chi}}{\bb{H}^2}
	\\
	&\lesssim
	\norm{\bff{\theta}^{n-1}}{\bb{L}^2}^2
	+
	\epsilon \norm{\bff{\theta}^{n-1}}{\bb{H}^2}^2
	+
	\epsilon \norm{\bff{\chi}}{\bb{H}^2}^2
	+
	h^{2(r+1)} \norm{\bff{U}^{n-1}}{\bb{H}^2}^2,
\end{align*}
and
\begin{align*} 
	T_5
	&\lesssim
	\norm{\bff{U}^{n-1}-\bff{u}^n}{\bb{L}^2}
	\norm{\bff{u}^n}{\bb{L}^\infty}
	\norm{\nabla\widetilde{\bff{u}}_h^n}{\bb{L}^4}
	\norm{\nabla\bff{\chi}}{\bb{L}^4}
	\lesssim
	\big(
	\norm{\bff{\theta}^{n-1}}{\bb{L}^2} + h^{r+1} + k
	\big)
	\norm{\bff{\chi}}{\bb{H}^2}
	\\
	&\lesssim
	\norm{\bff{\theta}^{n-1}}{\bb{L}^2}^2 + h^{2(r+1)} + k^2
	+ \epsilon \norm{\bff{\chi}}{\bb{H}^2}^2.
\end{align*}
Altogether, \eqref{equ:A Un-1 utilde n} follows from \eqref{equ:A Un-1 A un}.

Next, we prove \eqref{equ:A Un-1 utilde n L2}. Proceeding similarly as in
\eqref{equ:A Un-1 A un} and performing integration by parts on the terms with
coefficients $\beta_4$ and $\beta_5$, we obtain
\begin{align}\label{equ:A Un-1 A un L2}
	\nonumber
	&\big| \cal{A}(\bff{U}^{n-1}; \widetilde{\bff{u}}_h^n, \bff{\chi}) 
	- 
	\cal{A}(\bff{u}^n; \widetilde{\bff{u}}_h^n, \bff{\chi}) \big| 
	\\
	\nonumber
	&\leq 
	\beta_3 \big| \inpro{(|\bff{U}^{n-1}|^2 - |\bff{u}^n|^2) \widetilde{\bff{u}}_h^n}{\bff{\chi}}_{\bb{L}^2} \big| 
	+ 
	\beta_4 \big| \inpro{\nabla \cdot \big((\bff{U}^{n-1}- \bff{u}^n) \times \nabla \widetilde{\bff{u}}_h^n \big)}{\bff{\chi}}_{\bb{L}^2} \big|
	\\
	\nonumber
	&\quad 
	+ 
	\beta_5 \big| \inpro{\nabla \cdot \big( (|\bff{U}^{n-1}|^2 - |\bff{u}^n|^2) \nabla \widetilde{\bff{u}}_h^n \big)}{\bff{\chi}}_{\bb{L}^2} \big|
	+ 
	2\beta_5 \big| \inpro{\nabla \cdot \big( \bff{U}^{n-1} ( (\bff{U}^{n-1}-\bff{u}^n)\cdot \nabla \widetilde{\bff{u}}_h^n) \big)}{\bff{\chi}}_{\bb{L}^2} \big|
	\\
	\nonumber
	&\quad 
	+ 
	2\beta_5 \big| \inpro{\nabla\cdot \big( (\bff{U}^{n-1} -\bff{u}^n)(\bff{u}^n \cdot \nabla \widetilde{\bff{u}}_h^n) \big)}{\bff{\chi}}_{\bb{L}^2} \big|
	\\
	&=
	R_1 +\cdots + R_5.
\end{align}
By using \eqref{equ:Un2 un2}, H\"{o}lder's inequality, Sobolev embedding, 
\eqref{equ:rho}, \eqref{equ:del un}, \eqref{equ:aux H2},
Proposition~\ref{pro:euler stab}, and Young's inequality, we deduce for any
$ \epsilon>0 $ 
\begin{align*}
	R_1
	&\leq
	\beta_3 \big(\norm{\bff{\theta}^{n-1}}{\bb{L}^\infty}+ \norm{\bff{\rho}^{n-1}}{\bb{L}^\infty} 
	+ k \norm{\delta \bff{u}^n}{\bb{L}^\infty} \big)
	\norm{\bff{U}^{n-1}+ \bff{u}^n}{\bb{L}^2} 
	\norm{\widetilde{\bff{u}}_h^n}{\bb{L}^\infty} 
	\norm{\bff{\chi}}{\bb{L}^2}  
	\\
	&\lesssim
	\big(
	\norm{\bff{\theta}^{n-1}}{\bb{H}^2}
	+
	\norm{\bff{\rho}^{n-1}}{\bb{H}^2}
	+
	k
	\big)
	\norm{\widetilde{\bff{u}}_h^n}{\bb{H}^2} 
	\norm{\bff{\chi}}{\bb{L}^2}
	\\
	&\lesssim
	\norm{\bff{\theta}^{n-1}}{\bb{H}^2}^2
	+
	h^{2(r-1)}+ k^2 + \epsilon \norm{\bff{\chi}}{\bb{L}^2}^2.
\end{align*}
Before estimating the remaining terms $R_2$, \ldots, $R_5$, we note that
similarly to~\eqref{equ:del un} we can show with the help of the Sobolev
embedding $\bb{H}^2 \subset \bb{W}^{1,4}$ and the assumption \eqref{equ:ass 1}, that
\[
	\norm{\delta\bff{u}^n}{\bb{W}^{1,4}} \le C, \quad n=1,2,3,\ldots.
\]
Thus, the term~$R_2$ can be estimated by using H\"{o}lder's inequality,
\eqref{equ:prod sobolev mat dot} and \eqref{equ:aux H2} as follows
\begin{align*}
	R_2
	&\leq
	\beta_4 \norm{(\bff{U}^{n-1}-\bff{u}^n) \times \nabla \widetilde{\bff{u}}_h^n}{\bb{H}^1}
	\norm{\bff{\chi}}{\bb{L}^2}
	\\
	&\lesssim
	\big(\norm{\bff{\theta}^{n-1}}{\bb{L}^\infty} + \norm{\bff{\rho}^{n-1}}{\bb{L}^\infty} + k \norm{\delta \bff{u}^n}{\bb{L}^\infty} \big)  
	\norm{\widetilde{\bff{u}}_h^n}{\bb{H}^2} 
	\norm{\bff{\chi}}{\bb{L}^2} 
	\\ 
	&\quad
	+ 
	\big(\norm{\bff{\theta}^{n-1}}{\bb{W}^{1,4}} + \norm{\bff{\rho}^{n-1}}{\bb{W}^{1,4}} + k \norm{\delta \bff{u}^n}{\bb{W}^{1,4}} \big)  
	\norm{\widetilde{\bff{u}}_h^n}{\bb{W}^{1,4}} 
	\norm{\bff{\chi}}{\bb{L}^2}
	\\
	&\lesssim
	\norm{\bff{\theta}^{n-1}}{\bb{H}^2}^2
	+
	h^{2(r-1)}+ k^2 + \epsilon \norm{\bff{\chi}}{\bb{L}^2}^2.
\end{align*}
For the term $R_3$, by H\"older's inequality, \eqref{equ:prod sobolev scal mat}, \eqref{equ:Un2 un2 Linf}, \eqref{equ:Un2 un2 W14} and \eqref{equ:aux H2} we have
\begin{align*}
	R_3
	&\leq
	\beta_5
	\norm{(|\bff{U}^{n-1}|^2-|\bff{u}^n|^2)\nabla\widetilde{\bff{u}}_h^n}
	{\bb{H}^1}
	\norm{\bff{\chi}}{\bb{L}^2}
	\\
	&\lesssim
	\norm{(|\bff{U}^{n-1}|^2-|\bff{u}^n|^2)}{\bb{L}^\infty}
	\norm{\nabla\widetilde{\bff{u}}_h^n}{\bb{H}^{1}}
	\norm{\bff{\chi}}{\bb{L}^2}
	+
	\norm{(|\bff{U}^{n-1}|^2-|\bff{u}^n|^2)}{\bb{W}^{1,4}}
	\norm{\nabla\widetilde{\bff{u}}_h^n}{\bb{L}^{4}}
	\norm{\bff{\chi}}{\bb{L}^2}
	\\
	&\lesssim
	\big(1 + \norm{\bff{U}^{n-1}}{\bb{H}^2}\big)
	\big(\norm{\bff{\theta}^{n-1}}{\bb{H}^2} + h^{r-1} + k \big)
	\norm{\widetilde{\bff{u}}_h^n}{\bb{H}^2}
	\norm{\bff{\chi}}{\bb{L}^2}
	\\
	&\lesssim
	\big(1+ \norm{\bff{U}^{n-1}}{\bb{H}^2}^2 \big)
	\big(\norm{\bff{\theta}^{n-1}}{\bb{H}^2}^2 + h^{2(r-1)} + k^2 \big) 
	+ 
	\epsilon \norm{\bff{\chi}}{\bb{L}^2}^2,
\end{align*}
where in the last step we used Young's inequality.
The terms $R_4$ and $R_5$ can be estimated similarly as in $R_3$. This completes the proof of the lemma.
\end{proof}
By using Lemma~\ref{lem:A bil for}, we prove the following result, which is
analogous to~\eqref{equ:theta L2}.

\begin{proposition}\label{pro:the n L2}
Assume that $\bff{u}$
satisfies~\eqref{equ:ass 1}. Then for $h,k>0$ with~$k$ satisfying~\eqref{equ:k
lam}, and for any $n\in \{1,2,\ldots, \lfloor T/k \rfloor\}$,
\begin{equation}\label{equ:theta n euler L2}
	\norm{\bff{\theta}^n}{\bb{L}^2} \leq C(h^{r+1} +k).
\end{equation}
\end{proposition}
\begin{proof}
By using successively the definition $ \bff{\theta}^n = \bff{U}^n -
\widetilde{\bff{u}}_h^n $, equations~\eqref{equ:backwardeuler},
\eqref{equ:weakform}, and \eqref{equ:eq elliptic Un}, we have, for all
$\bff{\chi}\in \bb{V}_h$,
\begin{align}\label{equ:eq dtheta n}
	\nonumber
	&\inpro{\delta \bff{\theta}^n}{\bff{\chi}}_{\bb{L}^2} 
	+ 
	\cal{A}(\bff{U}^{n-1}; \bff{\theta}^n, \bff{\chi}) 
	\\
	\nonumber
	&= 
	\inpro{\delta \bff{U}^n}{\bff{\chi}}_{\bb{L}^2} 
	+ 
	\cal{A}(\bff{U}^{n-1}; \bff{U}^n, \bff{\chi}) 
	- 
	\inpro{\delta \widetilde{\bff{u}}_h^n}{\bff{\chi}}_{\bb{L}^2} 
	- 
	\cal{A}(\bff{U}^{n-1}; \widetilde{\bff{u}}_h^n, \bff{\chi}) 
	\\
	\nonumber
	&= 
	(\alpha+\beta_3) \inpro{\bff{U}^n}{\bff{\chi}}_{\bb{L}^2}
	- 
	\beta_6 \inpro{(\bff{j}\cdot\nabla) \bff{U}^n}{\bff{\chi}}_{\bb{L}^2} 
	- 
	\inpro{\delta \widetilde{\bff{u}}_h^n - \partial_t \bff{u}^n}{\bff{\chi}}_{\bb{L}^2} 
	- 
	\inpro{\partial_t \bff{u}^n}{\bff{\chi}}_{\bb{L}^2} 
	\\
	\nonumber
	&\quad 
	- 
	\Big( \cal{A}(\bff{U}^{n-1}; \widetilde{\bff{u}}_h^n,
	\bff{\chi}) - \cal{A}(\bff{u}^n;
	\widetilde{\bff{u}}_h^n, \bff{\chi}) \Big) 
	- 
	\cal{A}(\bff{u}^n; \widetilde{\bff{u}}_h^n, \bff{\chi}) 
	\\
	\nonumber
	&= 
	(\alpha+\beta_3) \inpro{\bff{U}^n}{\bff{\chi}}_{\bb{L}^2} 
	- 
	\beta_6 \inpro{(\bff{j}\cdot\nabla) \bff{U}^n}{\bff{\chi}}_{\bb{L}^2}
	-
	\inpro{\delta \widetilde{\bff{u}}_h^n - \partial_t \bff{u}^n}{\bff{\chi}}_{\bb{L}^2} 
	\\
	\nonumber
	&\quad
	- 
	(\alpha+\beta_3) \inpro{\bff{u}^n}{\bff{\chi}}_{\bb{L}^2} 
	- 
	\beta_6 \inpro{(\bff{j}\cdot\nabla) \bff{u}^n}{\bff{\chi}}_{\bb{L}^2}
	\\
	\nonumber
	&\quad 
	+ 
	\cal{A}(\bff{u}^n; \bff{u}^n, \bff{\chi}) 
	- 
	\Big( \cal{A}(\bff{U}^{n-1}; \widetilde{\bff{u}}_h^n,
	\bff{\chi}) - \cal{A}(\bff{u}^n;
	\widetilde{\bff{u}}_h^n, \bff{\chi}) \Big) 
	- 
	\cal{A}(\bff{u}^n; \widetilde{\bff{u}}_h^n, \bff{\chi}) 
	\\
	&= 
	(\alpha+\beta_3) \inpro{\bff{U}^n-\bff{u}^n}{\bff{\chi}}_{\bb{L}^2}
	+
	\beta_6 \inpro{(\bff{U}^n-\bff{u}^n)\otimes \bff{j}}{\nabla \bff{\chi}}_{\bb{L}^2}
	- 
	\inpro{\delta \bff{\rho}^n}{\bff{\chi}}_{\bb{L}^2}
	- 
	\inpro{\delta \bff{u}^n- \partial_t \bff{u}^n}{\bff{\chi}}_{\bb{L}^2}
	\nonumber\\
	&\quad
	- 
	\Big( \cal{A}(\bff{U}^{n-1}; \widetilde{\bff{u}}_h^n, \bff{\chi}) 
	- 
	\cal{A}(\bff{u}^n; \widetilde{\bff{u}}_h^n, \bff{\chi}) \Big).
\end{align}
Now, taking $\bff{\chi}= \bff{\theta}^n$, then using H\"{o}lder's and Young's
inequality, Lemma \ref{lem:auxbound} and \eqref{equ:A Un-1 utilde n}, we have
\begin{align*}
	&\frac{1}{2k} \left( \norm{\bff{\theta}^n}{\bb{L}^2}^2 - \norm{\bff{\theta}^{n-1}}{\bb{L}^2}^2 \right)
	+
	\frac{1}{2k} \norm{\bff{\theta}^n - \bff{\theta}^{n-1}}{\bb{L}^2}^2 
	+ 
	\mu \norm{ \bff{\theta}^n}{\bb{H}^2}^2  
	\\
	&\leq 
	\big| (\alpha+\beta_3) \inpro{\bff{\theta}^n + \bff{\rho}^n}{\bff{\theta}^n}_{\bb{L}^2} \big| 
	+
	\big| \beta_6 \inpro{(\bff{\theta}^n + \bff{\rho}^n)\otimes \bff{j}}{\nabla \bff{\theta}^n}_{\bb{L}^2} \big| 
	+ 
	\big| \inpro{\delta \bff{\rho}^n}{\bff{\theta}^n}_{\bb{L}^2} \big|
	\\
	&\quad
	+ 
	\big| \inpro{\delta \bff{u}^n- \partial_t \bff{u}^n}{\bff{\theta}^n}_{\bb{L}^2} \big|
	+ 
	\big| \cal{A}(\bff{U}^{n-1}; \widetilde{\bff{u}}_h^n, \bff{\theta}^n) 
	- 
	\cal{A}(\bff{u}^n; \widetilde{\bff{u}}_h^n, \bff{\theta}^n) \big|
	\\
	&\lesssim 
	\norm{\bff{\theta}^n + \bff{\rho}^n}{\bb{L}^2} \norm{\bff{\theta}^n}{\bb{H}^1} 
	+ 
	\norm{\delta \bff{\rho}^n}{\bb{L}^2} \norm{\bff{\theta}^n}{\bb{L}^2} 
	+ 
	\norm{\delta \bff{u}^n - \partial_t \bff{u}^n}{\bb{L}^2} \norm{\bff{\theta}^n}{\bb{L}^2} 
	\\
	&\quad 
	+ 
	\big(1+ \norm{\bff{U}^{n-1}}{\bb{H}^2}^2 \big) h^{2(r+1)}
	+ 
	k^2
	+
	\norm{\bff{\theta}^{n-1}}{\bb{L}^2}^2
	+
	\epsilon \norm{\bff{\theta}^{n-1}}{\bb{H}^2}^2
	+
	\epsilon \norm{\bff{\theta}^n}{\bb{H}^2}^2
	\\
	&\lesssim
	h^{2(r+1)}
	+
	\norm{\bff{\theta}^n - \bff{\theta}^{n-1}}{\bb{L}^2}^2
	+
	\norm{\bff{\theta}^{n-1}}{\bb{L}^2}^2
	+
	\norm{\delta \bff{\rho}^n}{\bb{L}^2}^2
	+
	\norm{\delta \bff{u}^n - \partial_t \bff{u}^n}{\bb{L}^2}^2
	\\
	&\quad 
	+ 
	\big(1+ \norm{\bff{U}^{n-1}}{\bb{H}^2}^2 \big) h^{2(r+1)} 
	+ 
	k^2
	+
	\norm{\bff{\theta}^{n-1}}{\bb{L}^2}^2
	+
	\epsilon \norm{\bff{\theta}^{n-1}}{\bb{H}^2}^2
	+
	\epsilon \norm{\bff{\theta}^n}{\bb{H}^2}^2
	\\
	&\leq
	C\big(1+ \norm{\bff{U}^{n-1}}{\bb{H}^2}^2 \big) h^{2(r+1)}
	+
	Ck^2
	+
	C\norm{\bff{\theta}^n-\bff{\theta}^{n-1}}{\bb{L}^2}^2
	+
	C\norm{\bff{\theta}^{n-1}}{\bb{L}^2}^2
	+
	C\epsilon \norm{\bff{\theta}^{n-1}}{\bb{H}^2}^2
	+
	C\epsilon \norm{\bff{\theta}^{n}}{\bb{H}^2}^2
\end{align*}
for any $\epsilon>0$, where in the last step we used \eqref{equ:d rho n} and \eqref{equ:diff un dt un}. Choosing $\epsilon=\mu/2C$, then multiplying by $k$ and summing the inequalities over $m\in \{1,2,\ldots,n\}$, we obtain
\begin{align*}
	\norm{\bff{\theta}^n}{\bb{L}^2}^2
	\leq
	\norm{\bff{\theta}^0}{\bb{L}^2}^2
	+ 
	Ch^{2(r+1)} \left( 1+ k \sum_{m=1}^n \norm{\bff{U}^{m-1}}{\bb{H}^2}^2 \right)
	+
	C k^2 
	+
	\sum_{m=1}^n C k \norm{\bff{\theta}^{m-1}}{\bb{L}^2}^2.
\end{align*}
Since $\bff{U}^0= \widetilde{\bff{u}}_h(0)$, by the discrete Gronwall inequality (noting the stability of $\bff{U}^n$ in Proposition~\ref{pro:euler stab}) we have
\begin{align}\label{equ:theta n L2}
	\norm{\bff{\theta}^n}{\bb{L}^2}^2 
	\leq 
	C (h^{2(r+1)}  + k^2 ),
\end{align}
where $C=C\big(\kappa,\mu, T,\norm{\bff{u}}{\bb{H}^{r+1}}, \norm{\partial_t \bff{u}}{\bb{H}^{r+1}}, \norm{\partial_{t}^2 \bff{u}}{\bb{L}^2}\big)$, from which \eqref{equ:theta n euler L2} then follows.
\end{proof}

We are now ready to prove the stability of the semi-implicit Euler method
in the $\bb{H}^2$ norm (with a time-step restriction) in the case where the triangulation $\cal{T}_h$ is quasi-uniform.

\begin{proposition}\label{pro:Un stable H2}
Let $T>0$ be given and $\bff{U}^n$ be defined by \eqref{equ:backwardeuler}. Assume that the triangulation $\cal{T}_h$ is quasi-uniform. If $k=O(h^2)$, then for any initial data $\bff{U}^0\in \bb{V}_h$ and $n\in \{1,2,\ldots, \lfloor T/k\rfloor \}$,
\begin{equation}\label{equ:H2 sta}
	\norm{\bff{U}^n}{\bb{H}^2}
	\leq
	C \norm{\bff{U}^0}{\bb{H}^2},
\end{equation}
where $C$ depends on $T$, but is independent of $n$ and $k$.
\end{proposition}

\begin{proof}
The identity~$\bff{U}^n=\bff{\theta}^n + \widetilde{\bff{u}}_h^{n+1}$, \eqref{equ:theta n euler L2}, \eqref{equ:aux H2}, and the inverse estimates give
\begin{align*}
	\norm{\bff{U}^n}{\bb{H}^2} 
	\leq 
	\norm{\bff{\theta}^n}{\bb{H}^2} 
	+
	\norm{\widetilde{\bff{u}}_h^{n+1}}{\bb{H}^2}
	\leq
	Ch^{-2} \norm{\bff{\theta}^n}{\bb{L}^2}
	+ C \norm{\bff{U}^0}{\bb{H}^2}
	\lesssim
	(1+kh^{-2}) \norm{\bff{U}^0}{\bb{H}^2}
	\lesssim 
	\norm{\bff{U}^0}{\bb{H}^2}
\end{align*}
if we assume $k=O(h^2)$, as required.
\end{proof}

In fact, we show that without assuming quasi-uniformity of $\cal{T}_h$, the $\bb{H}^2$ norm of discrete solution $\bff{U}^n$ remains bounded, albeit with
some drawback which will be explained later in Remark~\ref{rem:dra}.

\begin{proposition}\label{pro:Un bounded H2}
	Let $T>0$ and~$k$ satisfies~\eqref{equ:k lam}. Let
	$\bff{U}^n$ be defined by
	\eqref{equ:backwardeuler} with initial data $\bff{U}^0 \in \bb{V}_h$.
	\begin{enumerate}
		\renewcommand{\labelenumi}{\theenumi}
		\renewcommand{\theenumi}{{\rm (\roman{enumi})}}
		\item 
		If $d\leq 2$, then 
		for
		any $n\in \{1,2,\ldots, \lfloor T/k \rfloor\}$,
		\begin{equation}\label{equ:Un bounded H2}
			\norm{\bff{U}^n}{\bb{H}^2}^2
			\leq
			C
			\norm{\bff{U}^0}{\bb{H}^2}^2
			\exp\left( C\norm{\bff{U}^0}{\bb{L}^2}^2 \right),
		\end{equation}
		where $C$ depends on 
		$T$ but is independent of $n$ and $k$.
		\item
		If $d= 3$, then
		\begin{equation}\label{equ:Un bounded H2 d3}
			\norm{\bff{U}^n}{\bb{H}^2}^2
			\leq
			C
			\left(C_0\norm{\bff{U}^0}{\bb{H}^2}^{-3} - t_M \right)^{-\frac{2}{3}},
		\end{equation}
		for $n\in \{1,2,\ldots,M\}$, where
		\begin{align}\label{equ:M max}
			M:= \max \big\{n:\;  n < \lfloor T/k \rfloor, \;
			nk\le C_0 \norm{\bff{U}^0}{\bb{H}^2}^{-3}, \; 
			k \norm{\bff{U}^m-\bff{U}^{m-1}}{\bb{H}^2}^3 
			\le C_0,\; \forall m\leq n \big\}.
		\end{align}
		Here~$C_0$ depends on~$C_0'$, the constant given by the Gagliardo--Nirenberg inequality;
		see~\eqref{equ:gal nir 2d}. The constant $C$ depends on $t_M:=Mk \le T$.
	\end{enumerate}
\end{proposition}

\begin{proof}
	By taking $\bff{\chi}= \delta \bff{U}^n$ in \eqref{equ:backwardeuler}
	and using~\eqref{equ:bilinear}, we have (after rearranging the equation)
	\begin{align*} 
		\norm{\delta \bff{\bff{U}}^n}{\bb{L}^2}^2
		+
		\beta_1 \inpro{\nabla\bff{U}^n}{\nabla\delta\bff{U}^n}_{\bb{L}^2}
		+
		\beta_2 \inpro{\Delta\bff{U}^n}{\Delta\delta\bff{U}^n}_{\bb{L}^2}
		&=
		\beta_3 \inpro{\bff{U}^n}{\delta\bff{U}^n}_{\bb{L}^2}
		+
		\beta_6 \inpro{(\bff{j}\cdot \nabla)\bff{U}^n}{\delta \bff{U}^n}_{\bb{L}^2}
		\\
		&\quad
		-
		\cal{B}(\bff{U}^{n-1}, \bff{U}^{n-1}; \bff{U}^n, \delta\bff{U}^n).
	\end{align*}
	For simplicity of notations, we assume that all the constants in the above
	equation equal~$1$. (The case of negative~$\beta_1$ can be dealt with in the same
	manner as in the proof of Proposition~\ref{pro:euler stab}.)
	By using \eqref{equ:ab ide} for the second and third term on the left-hand side,
	H\"older's inequality for the first term on the right-hand side, and
	\eqref{equ:B bounded W1,4} for the last term on the right-hand side, we have
	(after using Young's inequality)
	\begin{align*}
		S&:=
		\norm{\delta \bff{\bff{U}}^n}{\bb{L}^2}^2
		+
		\frac{1}{2k} 
		\Big( 
		\norm{\nabla \bff{U}^n}{\bb{L}^2}^2 
		- \norm{\nabla \bff{U}^{n-1}}{\bb{L}^2}^2 
		+
		\norm{\nabla \bff{U}^n -\nabla \bff{U}^{n-1}}{\bb{L}^2}^2
		\Big)
		\\
		&\quad
		+
		\frac{1}{2k} 
		\Big( 
		\norm{\Delta \bff{U}^n}{\bb{L}^2}^2 
		- 
		\norm{\Delta \bff{U}^{n-1}}{\bb{L}^2}^2 
		+
		\norm{\Delta \bff{U}^n -\Delta \bff{U}^{n-1}}{\bb{L}^2}^2
		\Big)
		\\
		&\le
		\norm{\bff{U}^n}{\bb{H}^1} \norm{\delta \bff{U}^n}{\bb{L}^2}
		+
		\Big( \big(1+\norm{\bff{U}^{n-1}}{\bb{L}^\infty} \big) \norm{\bff{U}^{n-1}}{\bb{W}^{1,4}}
		\norm{\bff{U}^n}{\bb{W}^{1,4}} \Big)
		\norm{\delta \bff{U}^n}{\bb{L}^2}
		\\
		&\quad
		+
		\Big( \big(1+\norm{\bff{U}^{n-1}}{\bb{L}^\infty}^2 \big)
		\norm{\bff{U}^n}{\bb{H}^2} \Big)
		\norm{\delta \bff{U}^n}{\bb{L}^2}
		\\
		&\le
		\norm{\bff{U}^n}{\bb{H}^1}^2 
		+ 
		\epsilon \norm{\delta \bff{U}^n}{\bb{L}^2}^2
		+
		\big(1+\norm{\bff{U}^{n-1}}{\bb{L}^\infty}^2 \big)
		\norm{\bff{U}^{n-1}}{\bb{W}^{1,4}}^2
		\norm{\bff{U}^n}{\bb{W}^{1,4}}^2
		+ 
		\big(1+\norm{\bff{U}^{n-1}}{\bb{L}^\infty}^4 \big)
		\norm{\bff{U}^n}{\bb{H}^2}^2.
	\end{align*}
	Note that by the Gagliardo--Nirenberg inequalities
	\begin{align}\label{equ:gal nir 2d}
		\norm{\bff{v}}{\bb{L}^\infty}^2 
		\le 
		C_0'
		\norm{\bff{v}}{\bb{L}^2}^{2-\frac{d}{2}}
		\norm{\bff{v}}{\bb{H}^2}^{\frac{d}{2}}
		\quad\text{and}\quad
		\norm{\bff{v}}{\bb{W}^{1,4}}^2
		\le
		C_0'
		\norm{\bff{v}}{\bb{L}^2}^{1-\frac{d}{4}}
		\norm{\bff{v}}{\bb{H}^2}^{1+\frac{d}{4}},
		\quad\forall\bff{v}, \; d\le3.
	\end{align}
	Hence, it follows from the above inequalities and Proposition~\ref{pro:euler stab} that
	\begin{align} \label{equ:S d}
		S
		&\lesssim
		\norm{\bff{U}^n}{\bb{H}^1}^2
		+ 
		\epsilon \norm{\delta \bff{U}^n}{\bb{L}^2}^2
		+
		\big(1+ \norm{\bff{U}^{n-1}}{\bb{L}^2}^{2-\frac{d}{2}} \norm{\bff{U}^{n-1}}{\bb{H}^2}^{\frac{d}{2}} \big)
		\norm{\bff{U}^{n-1}}{\bb{L}^2}^{1-\frac{d}{4}}
		\norm{\bff{U}^{n-1}}{\bb{H}^2}^{1+\frac{d}{4}}
		\norm{\bff{U}^n}{\bb{L}^2}^{1-\frac{d}{4}}
		\norm{\bff{U}^n}{\bb{H}^2}^{1+\frac{d}{4}}
		\nonumber\\
		&\quad
		+
		\big(1+\norm{\bff{U}^{n-1}}{\bb{L}^2}^{4-d}
		\norm{\bff{U}^{n-1}}{\bb{H}^2}^d \big)
		\norm{\bff{U}^n}{\bb{H}^2}^2
		\nonumber\\
		&\lesssim
		\norm{\bff{U}^n}{\bb{H}^1}^2
		+ 
		\epsilon \norm{\delta \bff{U}^n}{\bb{L}^2}^2
		+
		\norm{\bff{U}^0}{\bb{L}^2}^{2-\frac{d}{2}}
		\norm{\bff{U}^{n-1}}{\bb{H}^2}^{1+\frac{d}{4}}
		\norm{\bff{U}^n}{\bb{H}^2}^{1+\frac{d}{4}}
		+
		\norm{\bff{U}^0}{\bb{L}^2}^{4-d}
		\norm{\bff{U}^{n-1}}{\bb{H}^2}^{1+\frac{3d}{4}}
		\norm{\bff{U}^n}{\bb{H}^2}^{1+\frac{d}{4}}
		\nonumber\\
		&\quad
		+
		\big(1+\norm{\bff{U}^0}{\bb{L}^2}^{4-d}
		\norm{\bff{U}^{n-1}}{\bb{H}^2}^d \big)
		\norm{\bff{U}^n}{\bb{H}^2}^2,
	\end{align}
	where the constant depends on~$C_0'$ given in~\eqref{equ:gal nir 2d}.
	For any $\epsilon>0$, applying Young's inequality we obtain
	\begin{align}\label{equ:S Un H2}
		S &\lesssim
		\begin{cases}
			\norm{\bff{U}^n}{\bb{H}^1}^2
			+
			\epsilon \norm{\delta \bff{U}^n}{\bb{L}^2}^2
			+
			\big(1+\norm{\bff{U}^0}{\bb{L}^2}^4\big) \norm{\bff{U}^{n-1}}{\bb{H}^2}^{2}
			+
			\big(1+ \norm{\bff{U}^0}{\bb{L}^2}^2 \big)
			\norm{\bff{U}^{n-1}}{\bb{H}^2}^{2} \norm{\bff{U}^n}{\bb{H}^2}^{2}
			\ &
			\text{if $d=1$},
			\\[1ex]
			\norm
			{\bff{U}^n}{\bb{H}^1}^2
			+
			\epsilon \norm{\delta \bff{U}^n}{\bb{L}^2}^2
			+
			\big(1+\norm{\bff{U}^0}{\bb{L}^2}^2\big) \norm{\bff{U}^{n-1}}{\bb{H}^2}^{4}
			+
			\big(1+\norm{\bff{U}^0}{\bb{L}^2}^2 \big)
			\norm{\bff{U}^{n-1}}{\bb{H}^2}^{2} \norm{\bff{U}^n}{\bb{H}^2}^{2}
			\ &
			\text{if $d=2$},
			\\[1ex]
			\norm{\bff{U}^n}{\bb{H}^1}^2
			+
			\epsilon \norm{\delta \bff{U}^n}{\bb{L}^2}^2
			+
			\big(1+\norm{\bff{U}^0}{\bb{L}^2}\big)
			\norm{\bff{U}^{n-1}}{\bb{H}^2}^{5}
			+
			\big(1+\norm{\bff{U}^0}{\bb{L}^2}\big) \norm{\bff{U}^n-\bff{U}^{n-1}}{\bb{H}^2}^{5}
			\ &
			\text{if $d=3$}.
		\end{cases}
	\end{align}
	We will show the derivation of~\eqref{equ:S Un H2} in detail only for $d=2$. The other cases are similar. If $d=2$, then from~\eqref{equ:S d} we have
	\begin{align*}
		S 
		&\lesssim
		\norm{\bff{U}^n}{\bb{H}^1}^2
		+
		\epsilon \norm{\delta \bff{U}^n}{\bb{L}^2}^2
		+
		\norm{\bff{U}^{n-1}}{\bb{H}^2}^{\frac32}
		\norm{\bff{U}^{n}}{\bb{H}^2}^{\frac32}
		+
		\norm{\bff{U}^{n-1}}{\bb{H}^2}^{\frac52}
		\norm{\bff{U}^{n}}{\bb{H}^2}^{\frac32}
		+
		\left(1+\norm{\bff{U}^{n-1}}{\bb{H}^2}^2 \right) \norm{\bff{U}^n}{\bb{H}^2}^2
		\\
		&\lesssim
		\norm{\bff{U}^n}{\bb{H}^1}^2
		+
		\epsilon \norm{\delta \bff{U}^n}{\bb{L}^2}^2
		+
		\left(1+ \norm{\bff{U}^{n-1}}{\bb{H}^2}^{2}
		\norm{\bff{U}^{n}}{\bb{H}^2}^{2}\right) 
		+
		\left( \norm{\bff{U}^{n-1}}{\bb{H}^2}^4 + \norm{\bff{U}^{n-1}}{\bb{H}^2}^{2}
		\norm{\bff{U}^{n}}{\bb{H}^2}^{2}\right),
	\end{align*}
	which implies~\eqref{equ:S d} for $d=2$.
	Now, for the case $d=1$ and $2$, summing over $m\in
	\{1,2\ldots,n\}$ and multiplying by $k$,
	we deduce that
	\begin{align*}
		\norm{\bff{U}^n}{\bb{H}^2}^2
		+
		k \sum_{m=1}^n \norm{\delta \bff{\bff{U}}^m}{\bb{L}^2}^2
		\lesssim
		\norm{\bff{U}^0}{\bb{H}^2}^2
		+ 
		k \sum_{m=1}^n \big(\norm{\bff{U}^m}{\bb{H}^2}^2
		+
		\norm{\bff{U}^{m-1}}{\bb{H}^2}^2 \big) \norm{\bff{U}^{m-1}}{\bb{H}^2}^2 .
	\end{align*}
	Therefore, the discrete Gronwall inequality gives
	\begin{align}\label{equ:nab Un Del Un H2}
		\norm{\bff{U}^n}{\bb{H}^2}^2
		&\lesssim
		\norm{\bff{U}^0}{\bb{H}^2}^2 \exp \left( Ck \sum_{m=0}^n \norm{\bff{U}^m}{\bb{H}^2}^2 \right)
		\lesssim
		\norm{\bff{U}^0}{\bb{H}^2}^2
		\exp\left( C\norm{\bff{U}^0}{\bb{L}^2}^2 \right),
	\end{align}
	for some positive constant $C$, where in the last step we used
	Proposition~\ref{pro:euler stab} again.
	
	For the case $d=3$, summing \eqref{equ:S Un H2} over $m\in
	\{1,2\ldots,n\}$ and noting the definition of $M$, we have for all $n\leq M$,
	\begin{align*}
		\norm{\bff{U}^n}{\bb{H}^2}^2
		\lesssim
		\norm{\bff{U}^0}{\bb{H}^2}^2
		+
		k \sum_{m=0}^{n-1} \norm{\bff{U}^m}{\bb{H}^2}^5.
	\end{align*}
	We now invoke the discrete Bihari--Gronwall inequality (Theorem \ref{the:disc bihari}) with $\varphi(x)=x^2$, $\psi(x)=x^5$ to obtain
	\begin{align}\label{equ:Un H2 disc bihari}
		\norm{\bff{U}^n}{\bb{H}^2} 
		\lesssim
		\left(\norm{\bff{U}^0}{\bb{H}^2}^{-3} - Mk \right)^{-\frac{1}{3}}
		\lesssim
		\left(C_0\norm{\bff{U}^0}{\bb{H}^2}^{-3} - t_M \right)^{-\frac{1}{3}}
	\end{align}
	for $n\in \{1,2,\ldots,M\}$. Note that the right-hand side is positive due to the
	definition of~$M$ in~\eqref{equ:M max}. This completes the proof of the proposition.
\end{proof}

\begin{remark}\label{rem:dra}
	For a given $T$, in case $d\leq 2$, the bound in~\eqref{equ:Un bounded
	H2} holds independently of the number of iterations $n$ or time-step
	size $k$. However, note that it is possibly much larger than that in~\eqref{equ:H2
	sta}.
	For $d=3$, without the assumption of quasi-uniformity, the bound in~\eqref{equ:Un bounded H2 d3}
	is only guaranteed to hold for sufficiently small initial data or only up to time $t_M\le T$. This means we may not have global stability.
\end{remark}

Based on Proposition~\ref{pro:Un stable H2} and Proposition~\ref{pro:Un bounded
	H2}, we now prove convergence rates
	for~$\norm{\bff{\theta}^n}{\bb{H}^\beta}$ for $\beta=1$ and $2$,
which are analogous to~\eqref{equ:theta L2}.
To this end, for a given $T>0$,
define
\begin{equation}\label{equ:N def}
	N =
	\begin{cases}
		\min\big\{
		\lfloor T/k \rfloor , M \big\}
		\quad &\text{if $\cal{T}_h$ is not quasi-uniform and $d=3$,}
		\\[1ex]
		\lfloor T/k \rfloor \quad &\text{otherwise,}
	\end{cases}
\end{equation}
where $M$ is defined in~\eqref{equ:M max}.
We also define
\begin{align}\label{equ:alpha N}
	\alpha_N := \max \big\{\norm{\bff{U}^n}{\bb{H}^2} : n=1,2,\ldots,N \big\}.
\end{align}
Note that there are explicit bounds on $\alpha_N$ for three different cases:
\begin{enumerate}
	\item \label{it:quasi} the triangulation $\cal{T}_h$ is quasi-uniform
		and $k=O(h^2)$ with $d=1,2,3$,
	\item \label{it:2d} the triangulation is not quasi-uniform and $d\leq 2$,
	\item \label{it:3d} the triangulation is not quasi-uniform and $d=3$.
\end{enumerate}
In cases~\eqref{it:quasi} and~\eqref{it:2d},
the constant $\alpha_N$ depends only on $\norm{\bff{U}^0}{\bb{H}^2}$ and $T$, but is independent of $N$
and $k$; see~\eqref{equ:H2 sta} and \eqref{equ:Un bounded H2}.
Meanwhile, in case~\eqref{it:3d}, the constant~$\alpha_N$ depends on $\norm{\bff{U}^0}{\bb{H}^2}$
and~$t_M=Mk$.

\begin{proposition}\label{pro:theta n euler Hr}
Assume that $\bff{u}$ satisfies~\eqref{equ:ass 1}.
Then the inequality
\begin{align}\label{equ:theta n euler Hr}
	\norm{\bff{\theta}^n}{\bb{H}^\beta}
	\leq 
	C(h^{r+1-\beta} +k), \quad \beta=1,2
\end{align}
holds for $n\in \{1,2,\ldots, N\}$, where $C$ depends on $\kappa, \mu, T, \alpha_N, \norm{\bff{u}}{L^\infty({\bb{H}^{r+1}})},
\norm{\partial_t \bff{u}}{L^\infty({\bb{H}^{r+1}})}, 
\norm{\partial_{t}^2 \bff{u}}{L^\infty({\bb{L}^2})}$. Here, $N$ and $\alpha_N$ are defined by \eqref{equ:N def} and \eqref{equ:alpha N}, respectively.
\end{proposition}

\begin{proof}
Note that the definition \eqref{equ:bilinear} of $\cal{A}$ gives
\begin{align*}
	\cal{A}(\bff{U}^{n-1}; \bff{\theta}^n, \delta \bff{\theta}^n)
	&=
	\frac{\alpha}{2k} \big( \norm{\bff{\theta}^n}{\bb{L}^2}^2 - \norm{\bff{\theta}^{n-1}}{\bb{L}^2}^2 \big)
	+
	\frac{\alpha}{2k} \norm{\bff{\theta}^n - \bff{\theta}^{n-1}}{\bb{L}^2}^2
	\\
	&\quad
	+
	\frac{\beta_1}{2k} \big( \norm{\nabla \bff{\theta}^n}{\bb{L}^2}^2 - \norm{\nabla \bff{\theta}^{n-1}}{\bb{L}^2}^2 \big)
	+
	\frac{\beta_1}{2k} \norm{\nabla \bff{\theta}^n - \nabla \bff{\theta}^{n-1}}{\bb{L}^2}^2
	\\
	&\quad
	+
	\frac{\beta_2}{2k} \big( \norm{\Delta \bff{\theta}^n}{\bb{L}^2}^2 - \norm{\Delta \bff{\theta}^{n-1}}{\bb{L}^2}^2 \big)
	+
	\frac{\beta_2}{2k} \norm{\Delta \bff{\theta}^n - \Delta \bff{\theta}^{n-1}}{\bb{L}^2}^2
	\\
	&\quad
	+ \cal{B}(\bff{U}^{n-1}, \bff{U}^{n-1}; \bff{\theta}^n, \delta \bff{\theta}^n)
	+ \cal{C}(\bff{U}^{n-1}; \bff{\theta}^n, \delta \bff{\theta}^n).
\end{align*}
Now, we take $\bff{\chi}= \delta \bff{\theta}^n$ in equations \eqref{equ:eq dtheta n}. Using the above, H\"{o}lder's inequality, \eqref{equ:A Un-1 utilde n L2}, \eqref{equ:B bounded H2} (and noting \eqref{equ:equivnorm-nabL2 young}), we have
\begin{align*}
	\nonumber
	S&:= 
	\norm{\delta \bff{\theta}^n}{\bb{L}^2}^2
	+
	\frac{1}{k} \big( \norm{\bff{\theta}^n}{\bb{L}^2}^2 - \norm{\bff{\theta}^{n-1}}{\bb{L}^2}^2 \big)
	+
	\frac{1}{k} \big( \norm{\Delta \bff{\theta}^n}{\bb{L}^2}^2 - \norm{\Delta \bff{\theta}^{n-1}}{\bb{L}^2}^2 \big)
	+
	\frac{1}{k} \norm{\bff{\theta}^n - \bff{\theta}^{n-1}}{\bb{H}^2}^2 
	\\
	\nonumber
	&\lesssim
	\big| \inpro{\bff{\theta}^n + \bff{\rho}^n}{\delta \bff{\theta}^n}_{\bb{L}^2} \big| 
	+
	\big| \inpro{(\bff{j}\cdot \nabla) (\bff{\theta}^n + \bff{\rho}^n)}{\delta \bff{\theta}^n}_{\bb{L}^2} \big| 
	+ 
	\big| \inpro{\delta \bff{\rho}^n}{\delta \bff{\theta}^n}_{\bb{L}^2} \big|
	+ 
	\big| \inpro{\delta \bff{u}^n- \partial_t \bff{u}^n}{\delta \bff{\theta}^n}_{\bb{L}^2} \big|
	\\
	\nonumber
	&\quad
	+ 
	\big| \cal{A}(\bff{U}^{n-1}; \widetilde{\bff{u}}_h^n, \delta \bff{\theta}^n) 
	- 
	\cal{A}(\bff{u}^n; \widetilde{\bff{u}}_h^n, \delta \bff{\theta}^n) \big|
	+
	\big| \cal{B}(\bff{U}^{n-1}, \bff{U}^{n-1}; \bff{\theta}^n, \delta \bff{\theta}^n) \big|
	+
	\big| \cal{C}(\bff{U}^{n-1}; \bff{\theta}^n, \delta \bff{\theta}^n) \big|
	\\
	\nonumber
	&\lesssim 
	\norm{\bff{\theta}^n + \bff{\rho}^n}{\bb{H}^1} \norm{\delta \bff{\theta}^n}{\bb{L}^2} 
	+ 
	\norm{\delta \bff{\rho}^n}{\bb{L}^2} \norm{\delta \bff{\theta}^n}{\bb{L}^2} 
	+ 
	\norm{\delta \bff{u}^n - \partial_t \bff{u}^n}{\bb{L}^2} \norm{\delta \bff{\theta}^n}{\bb{L}^2} 
	\\
	\nonumber
	&\quad
	+
	\big(1+ \norm{\bff{U}^{n-1}}{\bb{H}^2}^2 \big)
	\big(\norm{\bff{\theta}^{n-1}}{\bb{H}^2}^2 + h^{2(r-1)} + k^2 \big) 
	+ 
	\epsilon \norm{\delta \bff{\theta}^n}{\bb{L}^2}^2
	+
	\big(1+\norm{\bff{U}^{n-1}}{\bb{H}^2}^2 \big) \norm{\bff{\theta}^n}{\bb{H}^2} \norm{\delta \bff{\theta}^n}{\bb{L}^2}
	\\
	&\lesssim	
	h^{2r} +k^2+ \epsilon \norm{\delta \bff{\theta}^n}{\bb{L}^2}^2
	+
	(1+\alpha_N^2) \big( \norm{\bff{\theta}^{n-1}}{\bb{H}^2}^2 + h^{2(r-1)}+ k^2 \big)
	\\
	&\quad
	+ 
	(1+\alpha_N^4) \big(\norm{\bff{\theta}^n-\bff{\theta}^{n-1}}{\bb{H}^2}^2 + \norm{\bff{\theta}^{n-1}}{\bb{H}^2}^2 \big)
	+ 
	\epsilon \norm{\delta \bff{\theta}^n}{\bb{L}^2}^2
\end{align*}
for any $\epsilon>0$, where in the last step we used \eqref{equ:d rho n},
\eqref{equ:diff un dt un}, \eqref{equ:theta n L2}, and Young's inequality.
This implies
\begin{align*}
	S
	&\lesssim
	h^{2(r-1)} + k^2 	
	+
	\norm{\bff{\theta}^n-\bff{\theta}^{n-1}}{\bb{H}^2}^2 
	+ 
	\norm{\Delta \bff{\theta}^{n-1}}{\bb{L}^2}^2
	+ 
	\epsilon \norm{\delta \bff{\theta}^n}{\bb{L}^2}^2.
\end{align*}
Taking $\epsilon>0$ and $k$ satisfying~\eqref{equ:k lam}, then summing the above over
$m\in \{1,2,\ldots,n\}$ (where $n\leq N$), we obtain
\begin{align*}
	k \norm{\delta \bff{\theta}^n}{\bb{L}^2}^2
	+
	\norm{\Delta \bff{\theta}^n}{\bb{L}^2}^2
	\lesssim
	\norm{\bff{\theta}^0}{\bb{H}^2}^2
	+ 
	h^{2(r-1)}
	+
	k^2 
	+
	\sum_{m=1}^n k \norm{\Delta \bff{\theta}^{m-1}}{\bb{L}^2}^2.
\end{align*}
The discrete Gronwall inequality yields
\begin{equation*}
	\norm{\Delta \bff{\theta}^n}{\bb{L}^2}^2
	\leq 
	C(h^{2(r-1)}
	+
	k^2).
\end{equation*}
Inequality \eqref{equ:theta n euler Hr} then follows by interpolation.
%
\end{proof}

We can now prove the main theorem of this section.

\begin{theorem}\label{the:backeulerrate}
Let $\bff{u}$ and $\bff{U}^n$ be solutions of \eqref{equ:weakform} and
\eqref{equ:backwardeuler}, respectively. Assume that $\bff{u}$
satisfies~\eqref{equ:ass 1}. Let $N$ and $\alpha_N$ are defined by \eqref{equ:N def} and \eqref{equ:alpha N}, respectively.
Then for $t_n=nk$, $n=1,\ldots,N$,
\begin{align}\label{equ:Un utn Hr euler}
	\norm{\bff{U}^n- \bff{u}(t_n)}{\bb{H}^\beta}
	\leq 
	C(h^{r+1-\beta} +k), \quad \beta=0,1,2,
\end{align}
where the constant $C$ depends on $\kappa$, $\mu$, $T$,
	$\norm{\bff{u}}{L^\infty({\bb{H}^{r+1}})}$,
	$\norm{\partial_t \bff{u}}{L^\infty({\bb{H}^{r+1}})}$, and
	$\norm{\partial_{t}^2 \bff{u}}{L^\infty({\bb{L}^2})}$.
\end{theorem}

\begin{proof}
Note that $\bff{U}^n- \bff{u}(t_n)= \bff{\theta}^n + \bff{\rho}^n$.
Inequality
\eqref{equ:Un utn Hr euler} for $\beta=0$ follows by applying the triangle inequality,
\eqref{equ:theta n euler L2}, and \eqref{equ:rho}. The cases $\beta=1$ and $2$ follow from \eqref{equ:theta n euler Hr} and \eqref{equ:rho}. This
completes the proof of the theorem.
\end{proof}

\section{Time Discretisation by the Linearised BDF2 Method}\label{sec:bdf2}

To obtain higher order accuracy in time, we shall consider the BDF2 (second order backward differentiation formula) time discretisation scheme. First, for any discrete vector-valued function $\bff{v}^n$, define
\[
	\cal{D}\bff{v}^n := \delta \bff{v}^n + \frac{k}{2} \delta^2 \bff{v}^n
	=
	\delta \bff{v}^n + \frac{1}{2} (\delta \bff{v}^n- \delta \bff{v}^{n-1})
	=
	\frac{3\bff{v}^n - 4\bff{v}^{n-1} + \bff{v}^{n-2}}{2k}
	,
	\quad
	n=2, 3, \ldots,
\]
where $\delta^2 \bff{v}^n= \delta(\delta \bff{v}^n)$.
Under the assumption \eqref{equ:ass 2}, by using \eqref{equ:del un}, we have for all
$p\in [1,\infty]$ and $n=1,2,\ldots$,
\begin{align}\label{equ:Du Lp}
	\norm{\cal{D} \bff{u}^n}{\bb{L}^p}
	=
	\norm{\delta \bff{u}^n + \frac{1}{2} (\delta \bff{u}^n- \delta \bff{u}^{n-1})}{\bb{L}^p}
	\leq
	\frac{3}{2} \norm{\delta \bff{u}^n}{\bb{L}^p}
	+
	\frac{1}{2} \norm{\delta \bff{u}^{n-1}}{\bb{L}^p}
	\leq C.
\end{align}
We start with $\bff{U}^0= \widetilde{\bff{u}}_h(0)\in \bb{V}_h$. Since the BDF2
method we are going to introduce, see~$\eqref{equ:bdf2}$ below, is a two-step method, we need to define $\bff{U}^1$ by employing one time step of the Crank--Nicolson method to maintain a second-order accuracy in time (or by one time step of the semi-implicit Euler method in~\eqref{equ:backwardeuler} with very small step-size). More precisely,
\begin{equation}\label{equ:step one cn}
	\inpro{ \delta \bff{U}^1}{ \bff{\chi} }_{\bb{L}^2}
	+ 
	\cal{A}\left(\overline{\bff{U}^1}; \overline{\bff{U}^1}, \bff{\chi} \right)
	-
	(\alpha+\beta_3) \inpro{\overline{\bff{U}^1}}{\bff{\chi}}_{\bb{L}^2}
	-
	\beta_6 \inpro{(\bff{j}\cdot \nabla)\overline{\bff{U}^1}}{\bff{\chi}}_{\bb{L}^2}
	=
	0, \quad
	\bff{\chi} \in \bb{V}_h,
\end{equation}
where $\overline{\bff{U}^1}:= (\bff{U}^0 + \bff{U}^1)/2$. It is straightforward to verify (using a method similar to the one in the previous section) that
\begin{equation}\label{equ:bdf step 1}
	\norm{\bff{\theta}^1}{\bb{H}^\beta}
	\leq 
	C(h^{r+1-\beta}+k^2), \quad \beta=0,1,2.
\end{equation}

Next, for $n\geq 2$ be such that $t_n \in [0,T]$, we
compute~$\bff{U}^n\in\bb{V}_h$ from $\bff{U}^{n-1}$ and $\bff{U}^{n-2}$ by
\begin{align}\label{equ:bdf2}
	&\inpro{ \cal{D} \bff{U}^n}{ \bff{\chi} }_{\bb{L}^2} 
	+ 
	\cal{A}(\bff{U}^{n-1}; \bff{U}^n, \bff{\chi})
	-
	(\alpha+\beta_3) \inpro{\bff{U}^n}{\bff{\chi}}_{\bb{L}^2}
	-
	\beta_6 \inpro{(\bff{j}\cdot\nabla)\bff{U}^n}{\bff{\chi}}_{\bb{L}^2}
	= 0,
	\quad \bff{\chi}\in \bb{V}_h,
\end{align}
which is a linear scheme.
We now show that the scheme \eqref{equ:bdf2} is well-defined.

\begin{proposition}
For $n \geq 2$, given $\bff{U}^{n-1}$ and $\bff{U}^{n-2}\in \bb{V}_h$ and $k<3\lambda/2$ (where $\lambda$ was defined in \eqref{equ:k con}),
there exists a unique $\bff{U}^n \in \bb{V}_h$ that solves the fully
discrete scheme \eqref{equ:bdf2}.
\end{proposition}

\begin{proof}
For each $\bff{\phi}\in \bb{V}_h$, define a bilinear form $\cal{F}(\bff{\phi}; \cdot,\cdot): \bb{V}_h \times \bb{V}_h \to \bb{R}$ by
\[
	\cal{F}(\bff{\phi}; \bff{v},\bff{w})
	:=
	(3-2k(\alpha+\beta_3)) \inpro{\bff{v}}{\bff{w}}_{\bb{L}^2}
	-
	2k\beta_6 \inpro{(\bff{j}\cdot\nabla)\bff{v}}{\bff{w}}_{\bb{L}^2}
	+
	2k \cal{A}(\bff{\phi}; \bff{v},\bff{w}).
\]
Equation \eqref{equ:bdf2} is then equivalent to
\[
	\cal{F}(\bff{U}^{n-1}; \bff{U}^n, \bff{\chi})
	=
	\inpro{4\bff{U}^{n-1}- \bff{U}^{n-2}}{\bff{\chi}}_{\bb{L}^2}.
\]
Therefore, the result follows by the same argument as in Proposition \ref{pro:backeuler exist}.
\end{proof}

Next, we show that the scheme \eqref{equ:bdf2} is unconditionally stable in $\bb{L}^2$ for $k$ satisfying \eqref{equ:k lam}. An important identity that will often be used is
\begin{equation}\label{equ:abc ide}
	2 \bff{a} \cdot (3\bff{a}-4\bff{b}+\bff{c})
	=
	|\bff{a}|^2- |\bff{b}|^2
	+
	|2\bff{a}-\bff{b}|^2 - |2\bff{b}-\bff{c}|^2
	+
	|\bff{a}-2\bff{b}+\bff{c}|^2,
	\quad \forall \bff{a},\bff{b},\bff{c} \in \bb{R}^3.
\end{equation}

\begin{proposition}\label{pro:bdf stab}
	Let $T>0$ be given and let $\bff{U}^n$ be defined by \eqref{equ:bdf2} with initial data $\bff{U}^0 \in \bb{V}_h$.
	Then for $k$ satisfying \eqref{equ:k lam} and $n\in \{1,2,\ldots, \lfloor T/k \rfloor\}$,
	\begin{align}\label{equ:bdf stab L2}
		\norm{\bff{U}^n}{\bb{L}^2}^2 
		+
		\sum_{m=1}^n \norm{2\bff{U}^m - \bff{U}^{m-1}}{\bb{L}^2}^2
		+
		k \sum_{m=1}^n \norm{\Delta \bff{U}^m}{\bb{L}^2}^2
		\lesssim 
		\norm{\bff{U}^0}{\bb{L}^2}^2,
	\end{align}
	where the constant depends on $T$, but is independent of $n$ and $k$.
\end{proposition}

\begin{proof}
	First, taking $\bff{\chi}=\overline{\bff{U}^1}$ in \eqref{equ:step one cn} and applying the same argument as in Proposition \ref{pro:euler stab}, we have
	\begin{align*}
		&\norm{\bff{U}^1}{\bb{L}^2}^2 - \norm{\bff{U}^0}{\bb{L}^2}^2
		+
		2\beta_2 k \norm{\Delta \overline{\bff{U}^1}}{\bb{L}^2}^2
		+
		2\beta_3 k \norm{\overline{\bff{U}}^n}{\bb{L}^4}^4
		+
		2\beta_5 k \norm{\abs{\overline{\bff{U}^1}} \abs{\nabla \overline{\bff{U}^1}}}{\bb{L}^2}^2
		+
		4\beta_5 k \norm{\overline{\bff{U}^1} \cdot \nabla \overline{\bff{U}^1}}{\bb{L}^2}^2
		\\
		&\leq 
		\frac{2k}{\lambda} \norm{\overline{\bff{U}^1}}{\bb{L}^2}^2
		+
		\beta_2 k \norm{\Delta \overline{\bff{U}^1}}{\bb{L}^2}^2,
	\end{align*}
	where $\lambda$ was defined in \eqref{equ:k con}. Therefore, for $k$ satisfying \eqref{equ:k lam},
	\[
		\norm{\bff{U}^1}{\bb{L}^2}^2 \leq C \norm{\bff{U}^0}{\bb{L}^2}^2.
	\]
	Next, taking $\bff{\chi}= \bff{U}^n$ in \eqref{equ:bdf2} and using \eqref{equ:abc ide},
	we have (after multiplying by $k$)
	\begin{align*}
		&\norm{\bff{U}^n}{\bb{L}^2}^2 
		- 
		\norm{\bff{U}^{n-1}}{\bb{L}^2}^2
		+ 
		\norm{2\bff{U}^n - \bff{U}^{n-1}}{\bb{L}^2}^2 
		-
		\norm{2\bff{U}^{n-1} - \bff{U}^{n-2}}{\bb{L}^2}^2
		+
		\norm{\bff{U}^n- 2 \bff{U}^{n-1} + \bff{U}^{n-2}}{\bb{L}^2}^2
		\\
		&\;
		+ 
		2 \beta_2 k \norm{\Delta \bff{U}^n}{\bb{L}^2}^2 
		+
		2 \beta_3 k\norm{|\bff{U}^{n-1}| |\bff{U}^n|}{\bb{L}^2}^2
		+ 
		4\beta_5 k \norm{\bff{U}^{n-1} \cdot \nabla \bff{U}^n}{\bb{L}^2}^2  
		+
		2\beta_5 k \norm{|\bff{U}^{n-1}| |\nabla \bff{U}^n|}{\bb{L}^2}^2
		\\
		&=
		- 
		2\beta_1 k \norm{\nabla \bff{U}^n}{\bb{L}^2}^2 
		+ 
		2\beta_3 k \norm{\bff{U}^n}{\bb{L}^2}^2
		+
		2\beta_6 k \norm{\nabla \bff{U}^n}{\bb{L}^2} \norm{\bff{U}^n}{\bb{L}^2}.
	\end{align*}
Similar argument as in the proof of Proposition \ref{pro:euler stab} then yields the result.
\end{proof}

Next, we show that the method is also stable in the $\bb{H}^2$ norm (in the sense of Definition \ref{def:stable}). Analogously to \eqref{equ:d rho n} and \eqref{equ:diff un dt un}, we have by \eqref{equ:Du Lp} and \eqref{equ:d rho n},
\begin{align}\label{equ:d rho n bdf}
	\norm{\cal{D} \bff{\rho}^n}{\bb{L}^2} 
	\leq
	\frac{3}{2} \norm{\delta \bff{\rho}^n}{\bb{L}^2}
	+
	\frac{1}{2} \norm{\delta \bff{\rho}^{n-1}}{\bb{L}^2}
	\leq 
	C h^{r+1},
\end{align}
and that by Taylor's theorem
\begin{align}\label{equ:diff un dt un bdf}
	\norm{\cal{D} \bff{u}^n - \partial_t \bff{u}^n}{\bb{L}^2} 
	\leq
	Ck \int_{t_{n-2}}^{t_n} \norm{\partial_{t}^3 \bff{u}(t)}{\bb{L}^2}\,\dt
	\leq 
	Ck^2,
\end{align}
where $C$ depends on $\kappa,\mu, T$, and $K$, with $K$ as defined in~\eqref{equ:ass 2}. 

By using Lemma~\ref{lem:A bil for}, we have the following result, which is analogous to Proposition \ref{pro:the n L2}.

\begin{proposition}\label{pro:the n bdf L2}
	Assume that $\bff{u}$
	satisfies~\eqref{equ:ass 2}. Then for $h,k>0$ with $k$ satisfying \eqref{equ:k lam}, and for any $n\in \{1,2,\ldots, \lfloor T/k \rfloor\}$,
	\begin{equation}\label{equ:theta n bdf L2}
		\norm{\bff{\theta}^n}{\bb{L}^2} \leq C(h^{r+1} +k^2).
	\end{equation}
\end{proposition}
\begin{proof}
	Similar calculations as in \eqref{equ:eq dtheta n} (replacing $\delta$ there by $\cal{D}$) shows that for all
	$\bff{\chi}\in \bb{V}_h$,
	\begin{align}\label{equ:eq dtheta n bdf}
		\nonumber
		\inpro{\cal{D} \bff{\theta}^n}{\bff{\chi}}_{\bb{L}^2} 
		+ 
		\cal{A}(\bff{U}^{n-1}; \bff{\theta}^n, \bff{\chi}) 
		&= 
		(\alpha+\beta_3) \inpro{\bff{U}^n-\bff{u}^n}{\bff{\chi}} 
		-
		\beta_6 \inpro{(\bff{U}^n-\bff{u}^n)\otimes \bff{j}}{\nabla \bff{\chi}}_{\bb{L}^2}
		- 
		\inpro{\cal{D} \bff{\rho}^n}{\bff{\chi}}_{\bb{L}^2}
		\\
		&\quad 
		- 
		\inpro{\cal{D} \bff{u}^n- \partial_t \bff{u}^n}{\bff{\chi}}_{\bb{L}^2}
		- 
		\Big( \cal{A}(\bff{U}^{n-1}; \widetilde{\bff{u}}_h^n, \bff{\chi}) 
		- 
		\cal{A}(\bff{u}^n; \widetilde{\bff{u}}_h^n, \bff{\chi}) \Big).
	\end{align}
	Now, take $\bff{\chi}= \bff{\theta}^n$. Using the identity \eqref{equ:abc ide} and applying similar argument as in the proof of Proposition \ref{pro:the n L2} yield
	\begin{align*}
		&\frac{1}{2k} \left( \norm{\bff{\theta}^n}{\bb{L}^2}^2 - \norm{\bff{\theta}^{n-1}}{\bb{L}^2}^2 \right)
		+
		\frac{1}{2k} \left( \norm{2\bff{\theta}^n - \bff{\theta}^{n-1}}{\bb{L}^2}^2 -
		\norm{2\bff{\theta}^{n-1} - \bff{\theta}^{n-2}}{\bb{L}^2}^2 \right)
		\\
		&\quad 
		+ 
		\frac{1}{2k} \norm{\bff{\theta}^n - 2\bff{\theta}^{n-1} 
		+ \bff{\theta}^{n-2}}{\bb{L}^2}^2
		+
		\mu \norm{ \bff{\theta}^n}{\bb{H}^2}^2  
		\\
		&\leq 
		\big| (\alpha+\beta_3) \inpro{\bff{\theta}^n + \bff{\rho}^n}{\bff{\theta}^n}_{\bb{L}^2} \big| 
		+
		\big| \beta_6 \inpro{(\bff{\theta}^n + \bff{\rho}^n)\otimes \bff{j}}{\nabla\bff{\theta}^n}_{\bb{L}^2} \big| 
		+ 
		\big| \inpro{\cal{D} \bff{\rho}^n}{\bff{\theta}^n}_{\bb{L}^2} \big|
		\\
		&\quad
		+ 
		\big| \inpro{\cal{D} \bff{u}^n- \partial_t \bff{u}^n}{\bff{\theta}^n}_{\bb{L}^2} \big|
		+ 
		\big| \cal{A}(\bff{U}^{n-1}; \widetilde{\bff{u}}_h^n, \bff{\theta}^n) 
		- 
		\cal{A}(\bff{u}^n; \widetilde{\bff{u}}_h^n, \bff{\theta}^n) \big|
		\\
		&\leq
		C\big(1+ \norm{\bff{U}^{n-1}}{\bb{H}^2}^2 \big) h^{2(r+1)}
		+
		Ck^2
		+
		C\norm{2\bff{\theta}^n-\bff{\theta}^{n-1}}{\bb{L}^2}^2
		+
		C\norm{\bff{\theta}^{n-1}}{\bb{L}^2}^2
		+
		C\epsilon \norm{\bff{\theta}^{n-1}}{\bb{H}^2}^2
		+
		C\epsilon \norm{\bff{\theta}^{n}}{\bb{H}^2}^2
	\end{align*}
	for any $\epsilon>0$, where in the last step we used \eqref{equ:d rho n bdf} and \eqref{equ:diff un dt un bdf}. Choosing $\epsilon=\mu/2C$, then multiplying by $k$ and summing the inequalities over $m\in \{2,3,\ldots,n\}$, we obtain
	\begin{align*}
		\norm{\bff{\theta}^n}{\bb{L}^2}^2
		\leq
		\norm{\bff{\theta}^1}{\bb{L}^2}^2
		+
		\norm{2\bff{\theta}^1 - \bff{\theta}^0}{\bb{L}^2}^2
		+ 
		Ch^{2(r+1)} \left( 1+ k \sum_{m=2}^n \norm{\bff{U}^{m-1}}{\bb{H}^2}^2 \right)
		+
		C k^4 
		+
		\sum_{m=2}^n C k \norm{\bff{\theta}^{m-1}}{\bb{L}^2}^2.
	\end{align*}
	Since $\bff{U}^0= \widetilde{\bff{u}}_h(0)$ and \eqref{equ:bdf step 1} hold, by the discrete Gronwall inequality (noting the stability of $\bff{U}^n$ in Proposition~\ref{pro:bdf stab}) we have
	\begin{align}\label{equ:theta n hk bdf L2}
		\norm{\bff{\theta}^n}{\bb{L}^2}^2 
		\leq 
		C (h^{2(r+1)}  + k^4 ),
	\end{align}
	where $C$ depends on $\kappa,\mu, T$, and $K$ (with $K$ as defined in~\eqref{equ:ass 2}). Inequality~\eqref{equ:theta n euler L2} then follows.
\end{proof}

Next, we discuss the stability of the semi-implicit BDF2 scheme in the $\bb{H}^2$ norm, first in the case where the triangulation is quasi-uniform.

\begin{proposition}\label{pro:Un stable H2 bdf}
	Let $T>0$ be given and $\bff{U}^n$ be defined by \eqref{equ:bdf2}. Assume that the triangulation $\cal{T}_h$ is quasi-uniform. If $k=O(h)$, then for any initial data $\bff{U}^0\in \bb{V}_h$ and $n\in \{1,2,\ldots, \lfloor T/k\rfloor \}$,
	\begin{equation}\label{equ:H2 sta bdf}
		\norm{\bff{U}^n}{\bb{H}^2}
		\leq
		C \norm{\bff{U}^0}{\bb{H}^2},
	\end{equation}
	where $C$ depends on $T$, but is independent of $n$ and $k$.
\end{proposition}

\begin{proof}
	The identity~$\bff{U}^n=\bff{\theta}^n + \widetilde{\bff{u}}_h^{n+1}$, \eqref{equ:theta n bdf L2}, \eqref{equ:aux H2}, and the inverse estimates give
	\begin{align*}
		\norm{\bff{U}^n}{\bb{H}^2} 
		\leq 
		\norm{\bff{\theta}^n}{\bb{H}^2} 
		+
		\norm{\widetilde{\bff{u}}_h^{n+1}}{\bb{H}^2}
		\leq
		Ch^{-2} \norm{\bff{\theta}^n}{\bb{L}^2}
		+ C \norm{\bff{U}^0}{\bb{H}^2}
		\lesssim
		(1+k^2 h^{-2}) \norm{\bff{U}^0}{\bb{H}^2}
		\lesssim 
		\norm{\bff{U}^0}{\bb{H}^2}
	\end{align*}
	if we assume $k=O(h)$, as required.
\end{proof}

In the following, we show that without assuming quasi-uniformity of $\cal{T}_h$, the energy
and the $\bb{H}^2$ norm of the discrete solution $\bff{U}^n$ remain bounded, analogously to Proposition \ref{pro:Un bounded H2}. To do this, we need the identity
\begin{align}\label{equ:id abc ab}
	4\cal{D}\bff{v}^n \cdot \delta \bff{v}^n
	=
	4 \abs{\delta \bff{v}^n}^2
	+
	\left( \abs{\delta \bff{v}^n}^2
	-
	\abs{\delta \bff{v}^{n-1}}^2 \right)
	+
	\abs{\delta \bff{v}^n - \delta \bff{v}^{n-1}}^2.
\end{align}

\begin{proposition}\label{pro:Un bounded H2 bdf}
	Let $T>0$ be given and $\bff{U}^n$ be defined by
	\eqref{equ:bdf2} with initial data $\bff{U}^0 \in \bb{V}_h$.
	\begin{enumerate}
		\renewcommand{\labelenumi}{\theenumi}
		\renewcommand{\theenumi}{{\rm (\roman{enumi})}}
		\item 
		If $d\leq 2$, then for sufficiently small $k$ 
		and $n\in \{1,2,\ldots, \lfloor T/k \rfloor\}$,
		\begin{equation}\label{equ:Un bounded H2 bdf}
			\norm{\bff{U}^n}{\bb{H}^2}^2
			\leq
			B
			\exp\left( C\norm{\bff{U}^0}{\bb{L}^2}^2 \right),
		\end{equation}
		where $C$ depends on 
		$T$, while $B$ depends on~$\norm{\bff{U}^0}{\bb{H}^2}$ and~$\norm{\bff{U}^1}{\bb{H}^2}$, but both $C$ and $B$ are independent of $n$ and $k$.
		\item
		If $d= 3$, then
		\begin{equation}\label{equ:Un bounded H2 d3 bdf}
			\norm{\bff{U}^n}{\bb{H}^2}^2
			\leq
			C
			\left(C_0\norm{\bff{U}^0}{\bb{H}^2}^{-3} - t_M \right)^{-\frac{2}{3}},
		\end{equation}
		for $n\in \{1,2,\ldots,M\}$, where
		\begin{align}\label{equ:M max bdf}
			M:= \max \big\{n:\; n < \lfloor T/k \rfloor, \;
			nk\le C_0 \norm{\bff{U}^0}{\bb{H}^2}^{-3}, \; 
			k \norm{\bff{U}^m-\bff{U}^{m-1}}{\bb{H}^2}^3 
			\le C_0,\; \forall m\leq n \big\}.
		\end{align}
		Here~$C_0$ depends on~$C_0'$, the constant given by the Gagliardo--Nirenberg inequality;
		see~\eqref{equ:gal nir 2d}. The constant $C$ depends on $t_M:=Mk \le T$.
	\end{enumerate}
\end{proposition}

\begin{proof}
	First, we aim to show that
	\begin{equation}\label{equ:k dU1}
	k \norm{\delta \bff{U}^1}{\bb{L}^2}^2 \leq B,
	\end{equation}
	for a constant $B$ depending on $\norm{\bff{U}^0}{\bb{H}^2}$ and $\norm{\bff{U}^1}{\bb{H}^2}$.
	To this end, we take $\bff{\chi}=\delta \bff{U}^1$ in \eqref{equ:step one cn}, then apply \eqref{equ:B bounded H2} and Young's inequality to obtain
	\begin{align*}
		&\norm{\delta \bff{U}^1}{\bb{L}^2}^2 
		+
		\frac{\beta_1}{2k} \big( \norm{\nabla \bff{U}^1}{\bb{L}^2}^2 - \norm{\nabla \bff{U}^{0}}{\bb{L}^2}^2 \big)
		+
		\frac{\beta_2}{2k} \big( \norm{\Delta \bff{U}^1}{\bb{L}^2}^2 - \norm{\Delta \bff{U}^{0}}{\bb{L}^2}^2 \big)
		\\
		&\leq
		\beta_3
		\abs{ \inpro{\overline{\bff{U}^1}}{\delta \bff{U}^1}_{\bb{L}^2} }
		+
		\beta_6
		\abs{ \inpro{(\bff{j}\cdot\nabla) \overline{\bff{U}^1}}{\delta \bff{U}^1}_{\bb{L}^2}}
		+
		\abs{ \cal{B}(\overline{\bff{U}^1}, \overline{\bff{U}^1}; \overline{\bff{U}^1}, \delta \bff{U}^1) }
		\\
		&\lesssim
		\norm{\overline{\bff{U}^1}}{\bb{H}^1}^2 
		+ 
		\epsilon \norm{\delta \bff{U}^1}{\bb{L}^2}^2
		+
		\left(1+\norm{\overline{\bff{U}^1}}{\bb{H}^2} \right) \norm{\overline{\bff{U}^1}}{\bb{H}^2}^2
		\norm{\delta \bff{U}^1}{\bb{L}^2}
		\\
		&\lesssim
		1+ \norm{\overline{\bff{U}^1}}{\bb{H}^2}^6
		+
		\epsilon \norm{\delta \bff{U}^1}{\bb{L}^2}^2
	\end{align*}
	for any $\epsilon>0$. Choosing $\epsilon>0$ sufficiently small, then rearranging and multiplying by $k$ give
	\[
	k \norm{\delta \bff{U}^1}{\bb{L}^2}^2 
	+
	\norm{\bff{U}^1}{\bb{H}^2}^2
	\lesssim
	\norm{\bff{U}^0}{\bb{H}^2}^2
	+
	k \left(1+ \norm{\overline{\bff{U}^1}}{\bb{H}^2}^6 \right)
	\leq
	B,
	\]
	where $B$ is a constant depending on $\norm{\bff{U}^0}{\bb{H}^2}$ and $\norm{\bff{U}^1}{\bb{H}^2}$, thus showing \eqref{equ:k dU1}.

	Next, by taking $\bff{\chi}= \delta \bff{U}^n$ in \eqref{equ:bdf2}
	and using~\eqref{equ:bilinear}, we have
	\begin{align*} 
		&\inpro{\cal{D}\bff{U}^n}{\delta \bff{U}^n}_{\bb{L}^2}
		+
		\beta_1 \inpro{\nabla\bff{U}^n}{\nabla\delta\bff{U}^n}_{\bb{L}^2}
		+
		\beta_2 \inpro{\Delta\bff{U}^n}{\Delta\delta\bff{U}^n}_{\bb{L}^2}
		\\
		&=
		\beta_3 \inpro{\bff{U}^n}{\delta\bff{U}^n}_{\bb{L}^2}
		+
		\beta_6 \inpro{(\bff{j}\cdot\nabla)\bff{U}^n}{\delta \bff{U}^n}_{\bb{L}^2}
		-
		\cal{B}(\bff{U}^{n-1}, \bff{U}^{n-1}; \bff{U}^n, \delta\bff{U}^n)
		-
		\cal{C}(\bff{U}^{n-1};\bff{U}^n, \delta \bff{U}^n).
	\end{align*}
	For simplicity of notations, we assume that all the constants in the above
	equation equal~$1$. By using \eqref{equ:id abc ab} and applying similar argument as in the proof of Proposition \ref{pro:Un bounded H2}, we have
	\begin{align*}
		&\norm{\delta \bff{\bff{U}}^n}{\bb{L}^2}^2
		+
		\frac{1}{4} \left( \norm{\delta \bff{U}^n}{\bb{L}^2}^2
		-
		\norm{\delta \bff{U}^{n-1}}{\bb{L}^2}^2 \right)
		+
		\frac{1}{4} \norm{\delta \bff{U}^n-\delta \bff{U}^{n-1}}{\bb{L}^2}^2
		\\
		&\quad 
		+
		\frac{1}{2k} 
		\left( 
		\norm{\nabla \bff{U}^n}{\bb{L}^2}^2 
		- \norm{\nabla \bff{U}^{n-1}}{\bb{L}^2}^2 
		+
		\norm{\nabla \bff{U}^n -\nabla \bff{U}^{n-1}}{\bb{L}^2}^2
		\right)
		\\
		&\quad
		+
		\frac{1}{2k} 
		\left( 
		\norm{\Delta \bff{U}^n}{\bb{L}^2}^2 
		- 
		\norm{\Delta \bff{U}^{n-1}}{\bb{L}^2}^2 
		+
		\norm{\Delta \bff{U}^n -\Delta \bff{U}^{n-1}}{\bb{L}^2}^2
		\right)
		\\
		&\le
		\norm{\bff{U}^n}{\bb{H}^1}^2 
		+ 
		\epsilon \norm{\delta \bff{U}^n}{\bb{L}^2}^2
		+
		\big(1+\norm{\bff{U}^{n-1}}{\bb{L}^\infty}^2 \big)
		\norm{\bff{U}^{n-1}}{\bb{W}^{1,4}}^2
		\norm{\bff{U}^n}{\bb{W}^{1,4}}^2
		+ 
		\big(1+\norm{\bff{U}^{n-1}}{\bb{L}^\infty}^4 \big)
		\norm{\bff{U}^n}{\bb{H}^2}^2.
	\end{align*}
	For the case $d=1$ and $2$, summing over $m\in
	\{2,3,\ldots,n\}$ and multiplying by $k$ (noting \eqref{equ:S Un H2} and \eqref{equ:k dU1}), we deduce
	\begin{align*}
		\norm{\bff{U}^n}{\bb{H}^2}^2
		+
		k \sum_{m=2}^n \norm{\delta \bff{\bff{U}}^m}{\bb{L}^2}^2
		\lesssim
		B
		+ 
		k \sum_{m=2}^n \big(\norm{\bff{U}^m}{\bb{H}^2}^2
		+
		\norm{\bff{U}^{m-1}}{\bb{H}^2}^2 \big) \norm{\bff{U}^{m-1}}{\bb{H}^2}^2 .
	\end{align*}
	Therefore, by the discrete Gronwall inequality and Proposition~\ref{pro:bdf stab}, we have
	\begin{align}\label{equ:nab Un Del Un H2 bdf}
		\norm{\bff{U}^n}{\bb{H}^2}^2
		&\lesssim
		B \exp \left( Ck \sum_{m=1}^n \norm{\bff{U}^m}{\bb{H}^2}^2 \right)
		\lesssim
		B
		\exp\left( C\norm{\bff{U}^0}{\bb{L}^2}^2 \right),
	\end{align}
	for some positive constants $B$ and $C$, thus proving \eqref{equ:Un bounded H2 bdf}.
	The case $d=3$ and the rest of the proof are similar to the proof of \eqref{equ:Un bounded H2 d3}, and so the details are omitted.
\end{proof}

We can also prove convergence rates for $\norm{\bff{\theta}^n}{\bb{H}^\beta}$ for $\beta=1$ and $2$, which are analogous to Proposition \ref{pro:theta n euler Hr}. See also the discussion preceding it.

\begin{proposition}\label{pro:theta n bdf Hr}
	Assume that $\bff{u}$ satisfies~\eqref{equ:ass 2}.
	Then the inequality
	\begin{align}\label{equ:theta n bdf Hr}
		\norm{\bff{\theta}^n}{\bb{H}^\beta}
		\leq 
		C(h^{r+1-\beta} +k^2), \quad \beta=1,2
	\end{align}
	holds for $n\in \{1,2,\ldots, N\}$. Here~$N$ and~$\alpha_N$ are defined by~\eqref{equ:N def} and~\eqref{equ:alpha N}, respectively.
\end{proposition}

\begin{proof}
The proof is similar to that of Proposition \ref{pro:theta n euler Hr} (noting \eqref{equ:id abc ab}), thus we omit the details.
\end{proof}

Finally, we have the main theorem of this section.

\begin{theorem}\label{the:bdf rate}
	Let $\bff{u}$ and $\bff{U}^n$ be solutions of \eqref{equ:weakform} and
	\eqref{equ:bdf2}, respectively. Assume that $\bff{u}$
	satisfies~\eqref{equ:ass 2}. Then for $t_n=nk$, $n=1,\ldots,N$,
	\begin{align}\label{equ:Un utn Hr bdf}
		\norm{\bff{U}^n- \bff{u}(t_n)}{\bb{H}^\beta}
		\leq 
		C(h^{r+1-\beta} +k^2), \quad \beta=0,1,2,
	\end{align}
	where the constant $C$ depends on $\kappa$, $\mu$, $T$, $\norm{\bff{u}}{L^\infty({\bb{H}^4})}$,
	$\norm{\partial_t \bff{u}}{L^\infty({\bb{H}^4})}$, 
	$\norm{\partial_{t}^2 \bff{u}}{L^\infty({\bb{H}^1})}$, and
	$\norm{\partial_{t}^3 \bff{u}}{L^\infty({\bb{L}^2})}$.
\end{theorem}

\begin{proof}
	The inequality follows from Proposition \ref{pro:the n bdf L2}, Proposition \ref{pro:theta n bdf Hr}, inequality~\eqref{equ:rho}, and the triangle inequality (see the proof of Theorem \ref{the:backeulerrate} for details).
\end{proof}

\section{Numerical Experiments}\label{sec:numeric exp}

We perform numerical simulations for the semi-implicit Euler scheme using the open-source package \textsc{FreeFEM} \cite{Hec12} to verify its order of convergence experimentally. These numerical experiments were done using the Hsieh--Clough--Tocher (HCT) element in the domain $\Omega \subset \bb{R}^2$ and $t\in [0,T]$. 
Since the exact solution of the equation is not known, we use extrapolation to empirically verify the order of convergence. Let $\bff{u}_h$ be the finite element solution with spatial step size $h$ and a very small fixed time step size. For $s=0,1,2,$ and each $h>0$, define the extrapolated spatial order of convergence in $L^\infty(\bb{H}^s)$-norm:
\begin{align*}
	\text{rate}_{\bb{H}^s}
	:=
	\log_2 \left[ \frac{\norm{\bff{e}_{2h}}{L^\infty(\bb{H}^s)}}
	{\norm{\bff{e}_h}{L^\infty(\bb{H}^s)}} \right],
	\text{ where } 
	\norm{\bff{e}_h}{L^\infty(\bb{H}^s)}:= \norm{\bff{u}_h - \bff{u}_{h/2}}{L^\infty(\bb{H}^s)}.
\end{align*}
The order of convergence in the $L^\infty(\bb{L}^2)$-norm, $L^\infty(\bb{H}^1)$ and $L^\infty(\bb{H}^2)$-norms are found to be $4, 3$, and $2$, respectively, confirming the theoretical results in Theorem~\ref{the:semidisc error} and Theorem~\ref{the:backeulerrate}.

\subsection{Experiment 1 (Landau--Lifshitz--Baryakhtar)}
Let $\bff{u}:\Omega\to \bb{R}^3$ be vector-valued. The coefficients in \eqref{equ:llbar} are taken to be $\beta_1=0.1, \beta_2=0.2, \beta_3=0.2, \beta_4=0.1$, $\beta_5=0.2, \beta_6=0$. The initial data $\bff{u}_0$ is given by
\begin{align}\label{equ:initial data}
	\bff{u}_0(x,y)
	=
	\big( 20 \sin^2(\pi x) \sin^2(\pi y), \,
	30 \sin^2(2\pi x) \sin^2(\pi y), \,
	40 \sin^2(\pi x) \sin^2(2\pi y) \big),
\end{align}
where $(x,y)\in [0,1]\times[0,1]$. The time step size $k= 1\times 10^{-9}$.

\begin{figure}[!hbt]
\begin{center}
	\begin{tikzpicture}
		\begin{axis}[
			title=Plot of $\|\bff{e}_h\|$ against $1/h$,
			height=0.45\textwidth,
			width=0.45\textwidth,
			xlabel= $1/h$,
			ylabel= $\|\bff{e}_h\|$,
			xmode=log,
			ymode=log,
			legend pos=outer north east,
			legend cell align=left,
			]
			\addplot+[mark=*,red] coordinates {(2,0.5539)(4,0.2563)(8,0.1751)(16,0.02277)(32,0.0008286)(64,0.00009154)(128,0.000006583)};
			\addplot+[mark=x,blue] coordinates {(2,6.226)(4,6.338)(8,5.886)(16,1.197)(32,0.09312)(64,0.01485)(128,0.001693)};
			\addplot+[mark=square,green] coordinates {(2,87.90)(4,218.9)(8,312.4)(16,107.6)(32,19.92)(64,6.931)(128,1.572)};
			\addplot+[dashed,no marks,red,domain=20:128]{450/x^4};
			\addplot+[dashed,no marks,blue,domain=20:128]{1400/x^3};
			\addplot+[dashed,no marks,green,domain=20:128]{9800/x^2};
			\legend{$L^\infty(\bb{L}^2)$-norm of $\bff{e}_h$,$L^\infty(\bb{H}^1)$-norm of $\bff{e}_h$,$L^\infty(\bb{H}^2)$-norm of $\bff{e}_h$,order 4 reference line,order 3 reference line,order 2 reference line}
		\end{axis}
	\end{tikzpicture}
\end{center}
	\caption{Order of convergence for experiment 1.}
	\label{fig:order}
\end{figure}

The plot of $\norm{\bff{e}_h}{}$ against $1/h$ is shown in Figure \ref{fig:order}.
Furthermore, snapshots of the magnetic spin field $\bff{u}$ at some indicated times (with initial data \eqref{equ:initial data}) are shown in Figure \ref{fig:snapshots field}.

\begin{figure}[!hbt]
	\centering
	\begin{subfigure}[b]{0.3\textwidth}
		\label{fig:time 0}
		\centering
		\includegraphics[width=\textwidth]{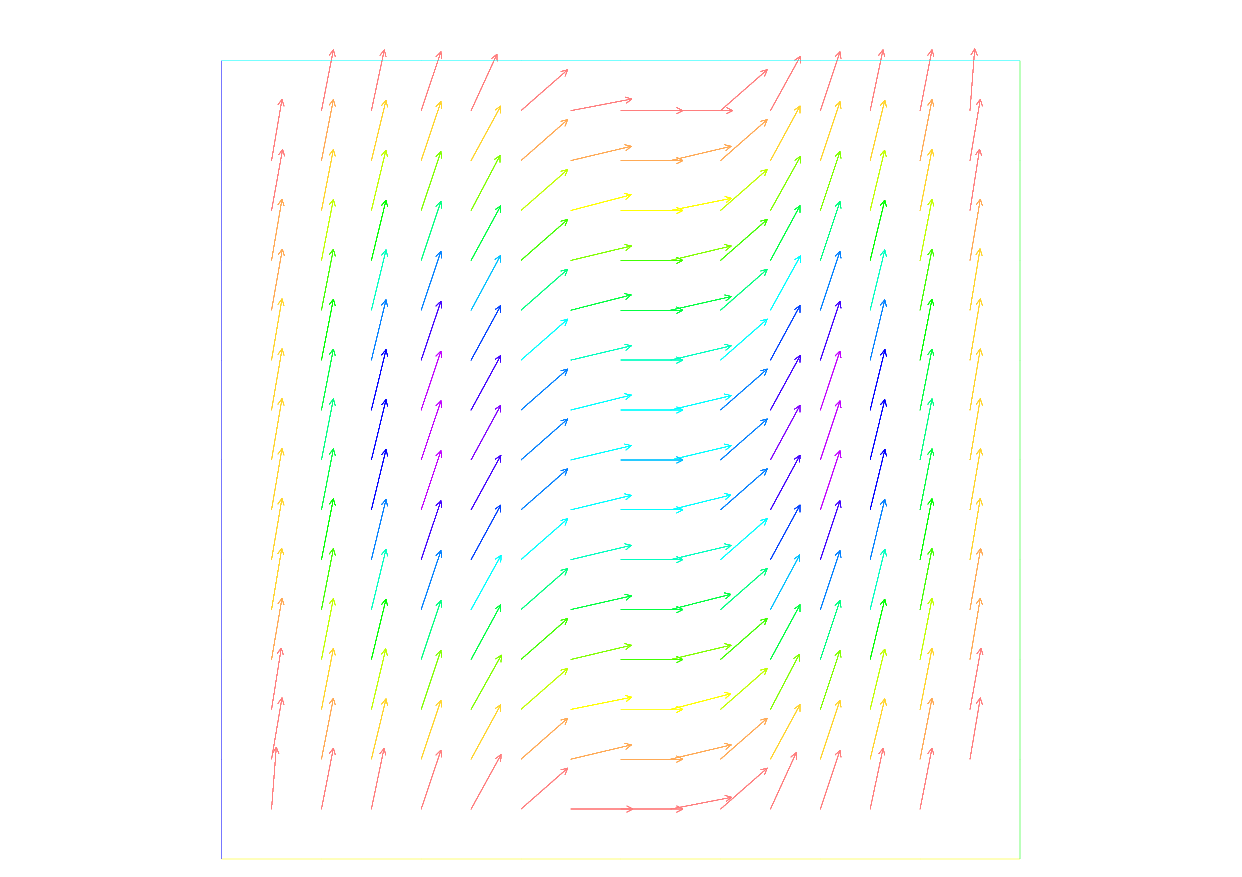}
		\caption{$t=0$}
	\end{subfigure}
	\begin{subfigure}[b]{0.3\textwidth}
		\label{fig:time 1e-04}
		\centering
		\includegraphics[width=\textwidth]{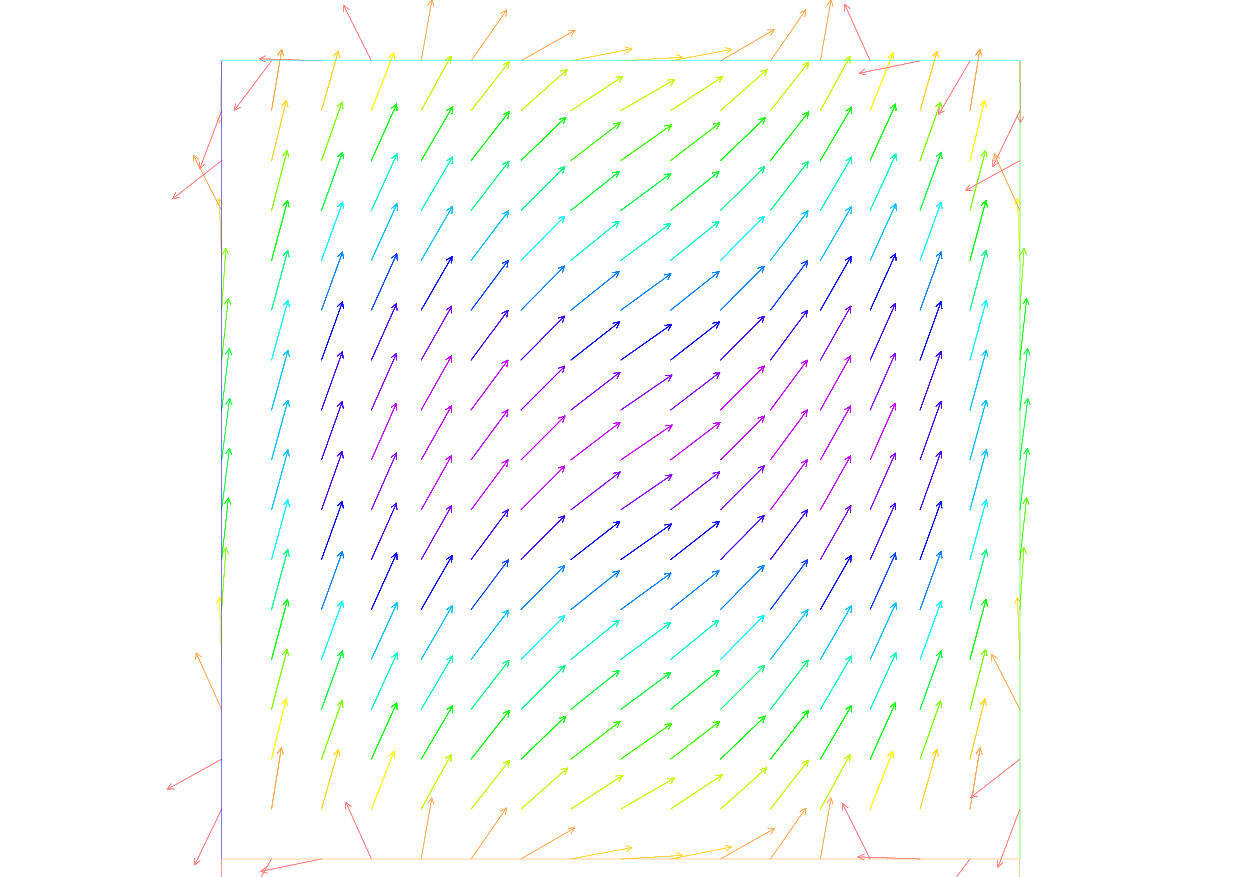}
		\caption{$t=1\times 10^{-4}$}
	\end{subfigure}
	\begin{subfigure}[b]{0.3\textwidth}
		\label{fig:time 2e-04}
		\centering
		\includegraphics[width=\textwidth]{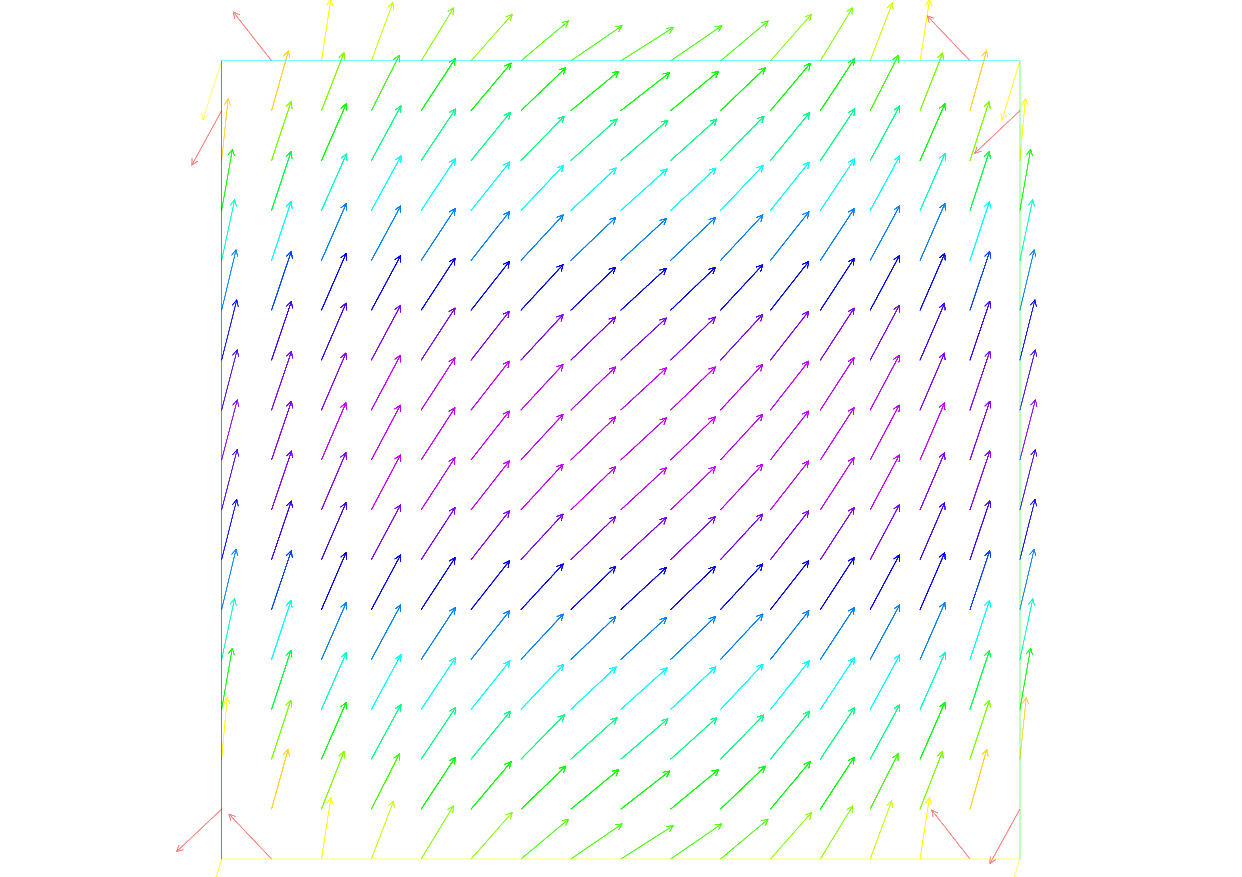}
		\caption{$t=2\times 10^{-4}$}
	\end{subfigure}
	\begin{subfigure}[b]{0.3\textwidth}
		\label{fig:time 3e-04}
		\centering
		\includegraphics[width=\textwidth]{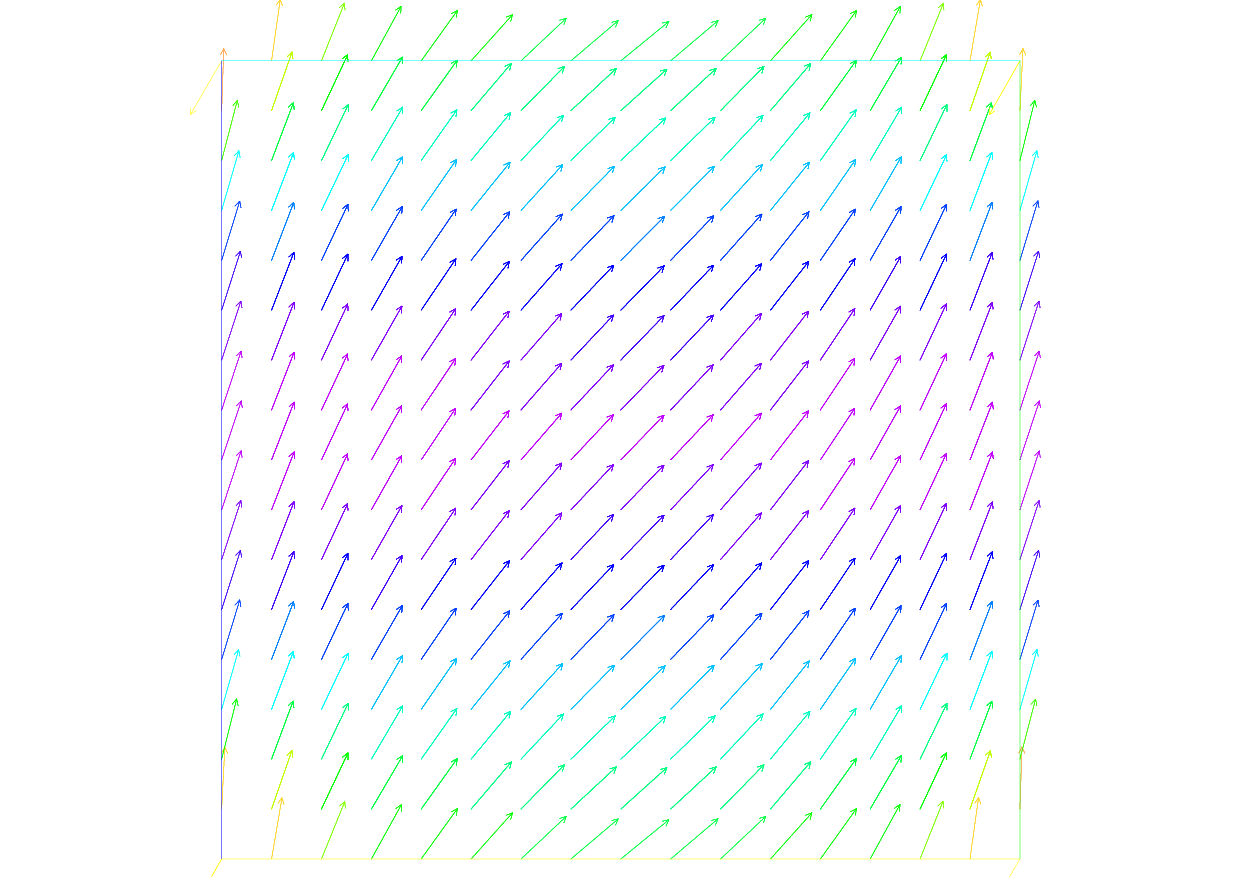}
		\caption{$t=3\times 10^{-4}$}
	\end{subfigure}
	\begin{subfigure}[b]{0.3\textwidth}
		\label{fig:time 5e-04}
		\centering
		\includegraphics[width=\textwidth]{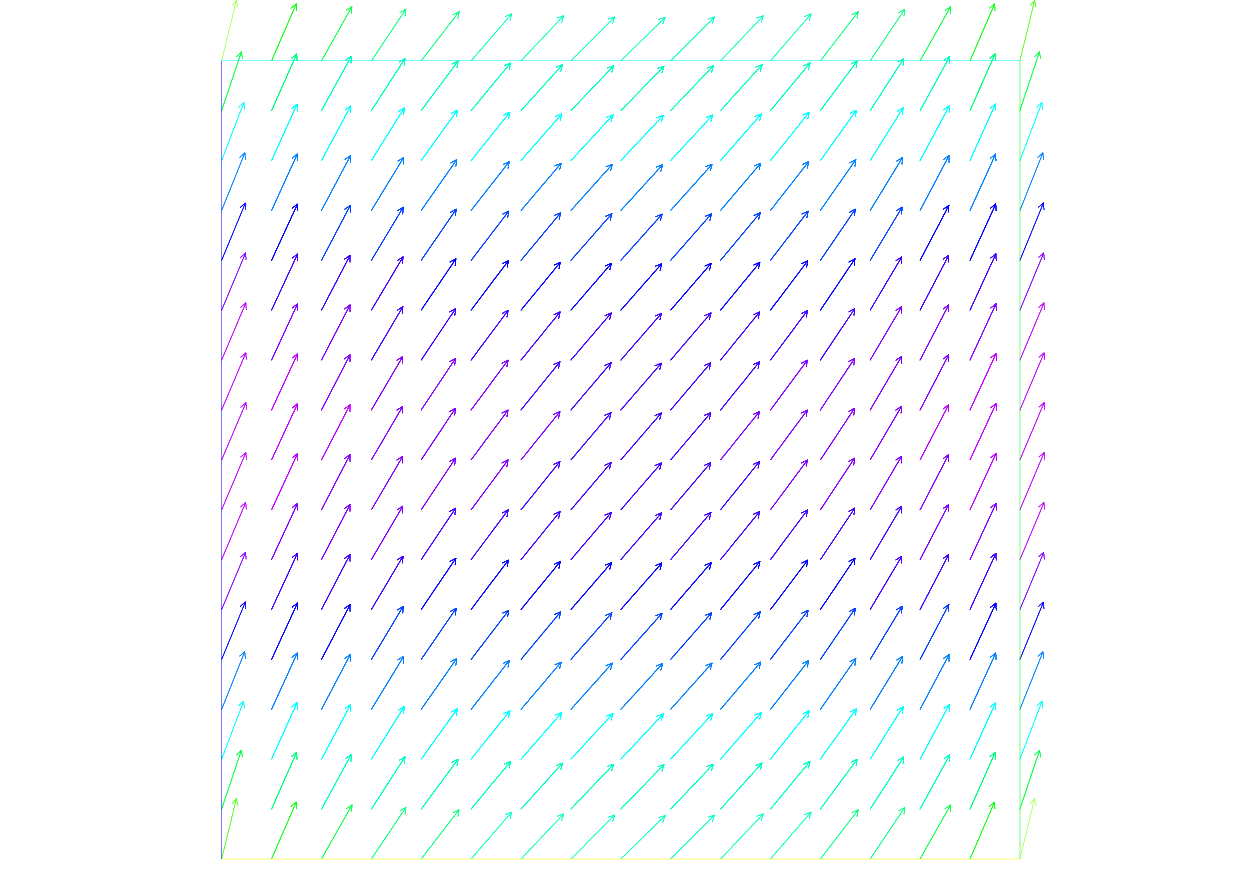}
		\caption{$t=5\times 10^{-4}$}
	\end{subfigure}
	\begin{subfigure}[b]{0.3\textwidth}
		\label{fig:time 1e-03}
		\centering
		\includegraphics[width=\textwidth]{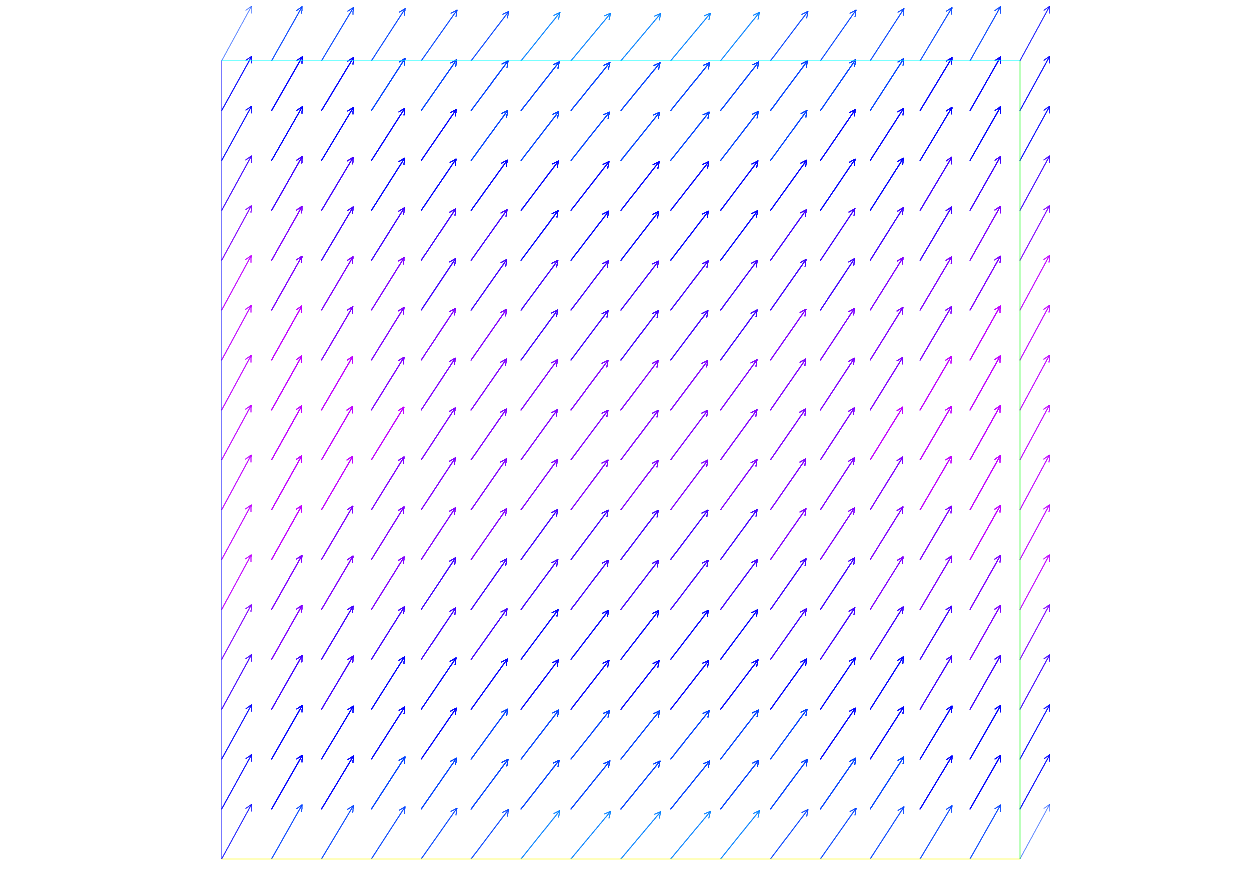}
		\caption{$t=1\times 10^{-3}$}
	\end{subfigure}
	\caption{Snapshots of the magnetic spin field $\bff{u}$ (projected onto $\bb{R}^2$) at given times in experiment 1. The colours indicate relative magnitude of the vectors.}
	\label{fig:snapshots field}
\end{figure}

\subsection{Experiment 2 (Landau--Lifshitz--Baryakhtar)}
Let $\bff{u}:\Omega\to \bb{R}^3$ be vector-valued. The coefficients in \eqref{equ:llbar} are taken to be $\beta_1=0.5, \beta_2=1.0, \beta_3=0.2, \beta_4=2.0$, $\beta_5=0.2$, $\beta_6=0.01$. The initial data $\bff{u}_0$ is given by
\begin{align}\label{equ:initial data 2}
	\bff{u}_0(x,y)
	=
	\big( xy, \, 2xy, \, 40 \sin (2\pi x) \big),
\end{align}
where $(x,y)\in [0,2]\times[0,2]$. The time step size $k= 1\times 10^{-9}$.

\begin{figure}[!hbt]
	\begin{center}
		\begin{tikzpicture}
			\begin{axis}[
				title=Plot of $\|\bff{e}_h\|$ against $1/h$,
				height=0.45\textwidth,
				width=0.45\textwidth,
				xlabel= $1/h$,
				ylabel= $\|\bff{e}_h\|$,
				xmode=log,
				ymode=log,
				legend pos=outer north east,
				legend cell align=left,
				]
				\addplot+[mark=*,red] coordinates {(2,0.296)(4,1.48)(8,0.39)(16,0.088)(32,0.01)(64,0.00074)(128,0.000046)};
				\addplot+[mark=x,blue] coordinates {(2,2.35)(4,10.1)(8,7.1)(16,2.9)(32,0.51)(64,0.11)(128,0.013)};
				\addplot+[mark=square,green] coordinates {(2,24.6)(4,74.7)(8,191.7)(16,133.1)(32,47)(64,14)(128,3.2)};
				\addplot+[dashed,no marks,red,domain=20:128]{3000/x^4};
				\addplot+[dashed,no marks,blue,domain=20:128]{6500/x^3};
				\addplot+[dashed,no marks,green,domain=20:128]{15800/x^2};
				\legend{$L^\infty(\bb{L}^2)$-norm of $\bff{e}_h$,$L^\infty(\bb{H}^1)$-norm of $\bff{e}_h$,$L^\infty(\bb{H}^2)$-norm of $\bff{e}_h$,order 4 reference line,order 3 reference line,order 2 reference line}
			\end{axis}
		\end{tikzpicture}
	\end{center}
	\caption{Order of convergence for experiment 2.}
	\label{fig:order exp2}
\end{figure}

The plot of $\norm{\bff{e}_h}{}$ against $1/h$ is shown in Figure \ref{fig:order exp2}.
Furthermore, snapshots of the magnetic spin field $\bff{u}$ at some indicated times (with initial data \eqref{equ:initial data 2}) are shown in Figure \ref{fig:snapshots field exp2}.

\begin{figure}[!hbt]
	\centering
	\begin{subfigure}[b]{0.3\textwidth}
		\label{fig:time2 0}
		\centering
		\includegraphics[width=\textwidth]{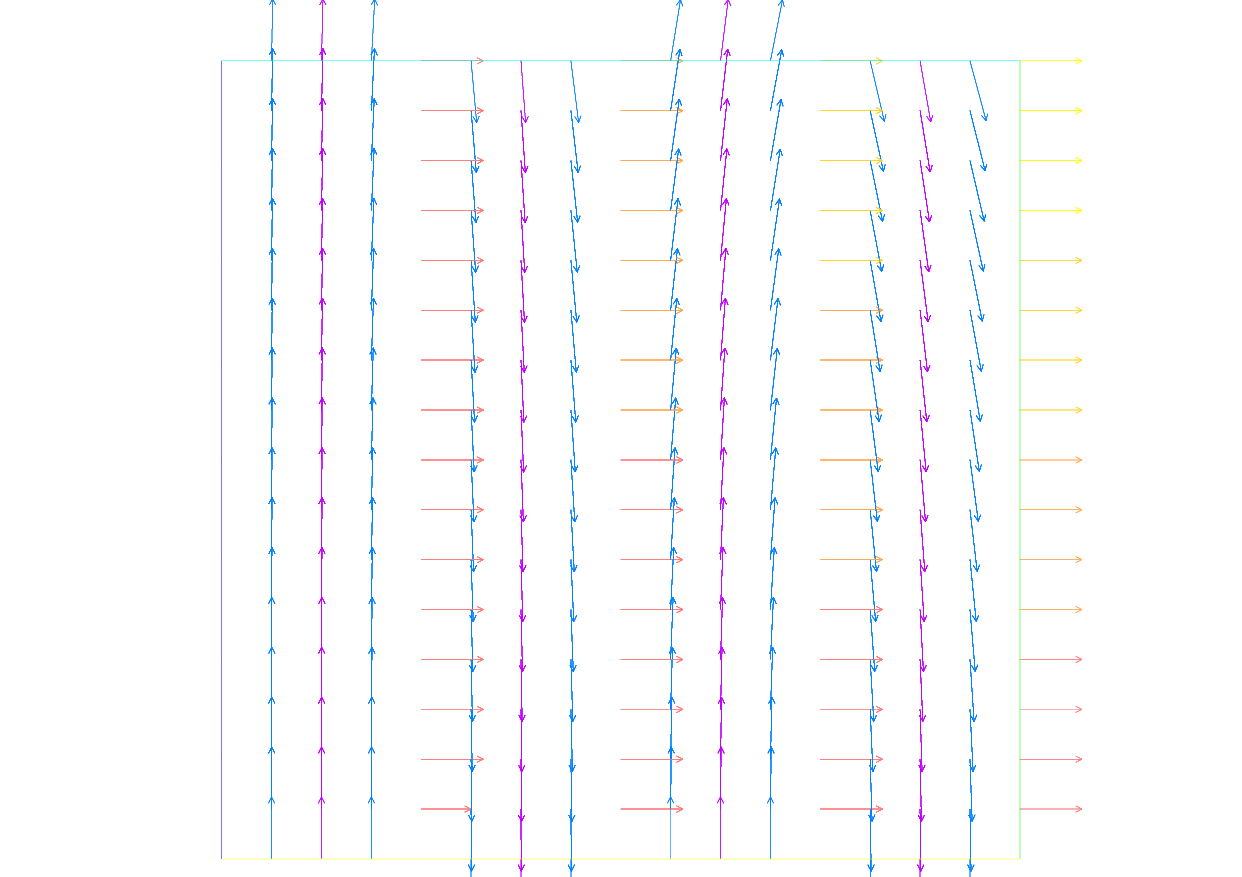}
		\caption{$t=0$}
	\end{subfigure}
	\begin{subfigure}[b]{0.3\textwidth}
		\label{fig:time2 1}
		\centering
		\includegraphics[width=\textwidth]{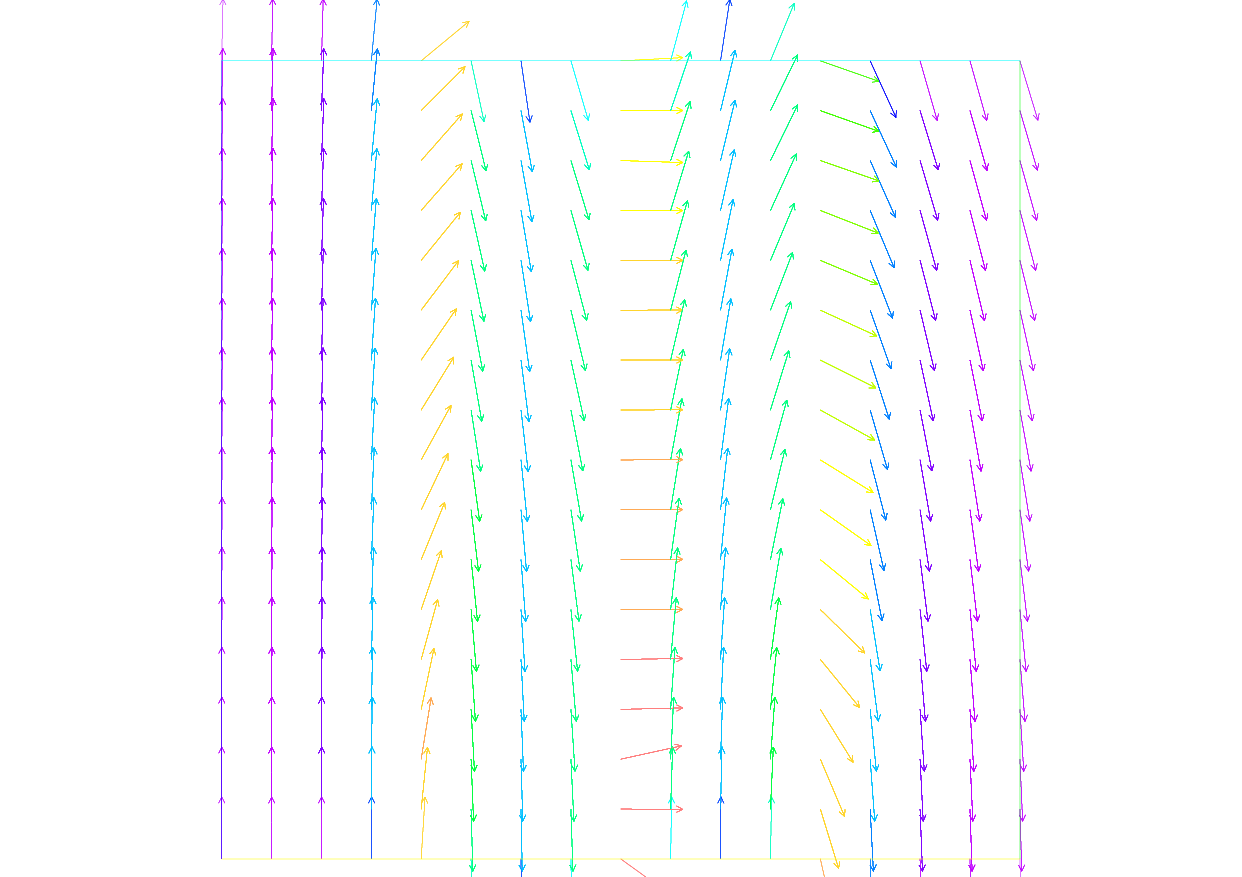}
		\caption{$t=2\times 10^{-4}$}
	\end{subfigure}
	\begin{subfigure}[b]{0.3\textwidth}
		\label{fig:time2 2}
		\centering
		\includegraphics[width=\textwidth]{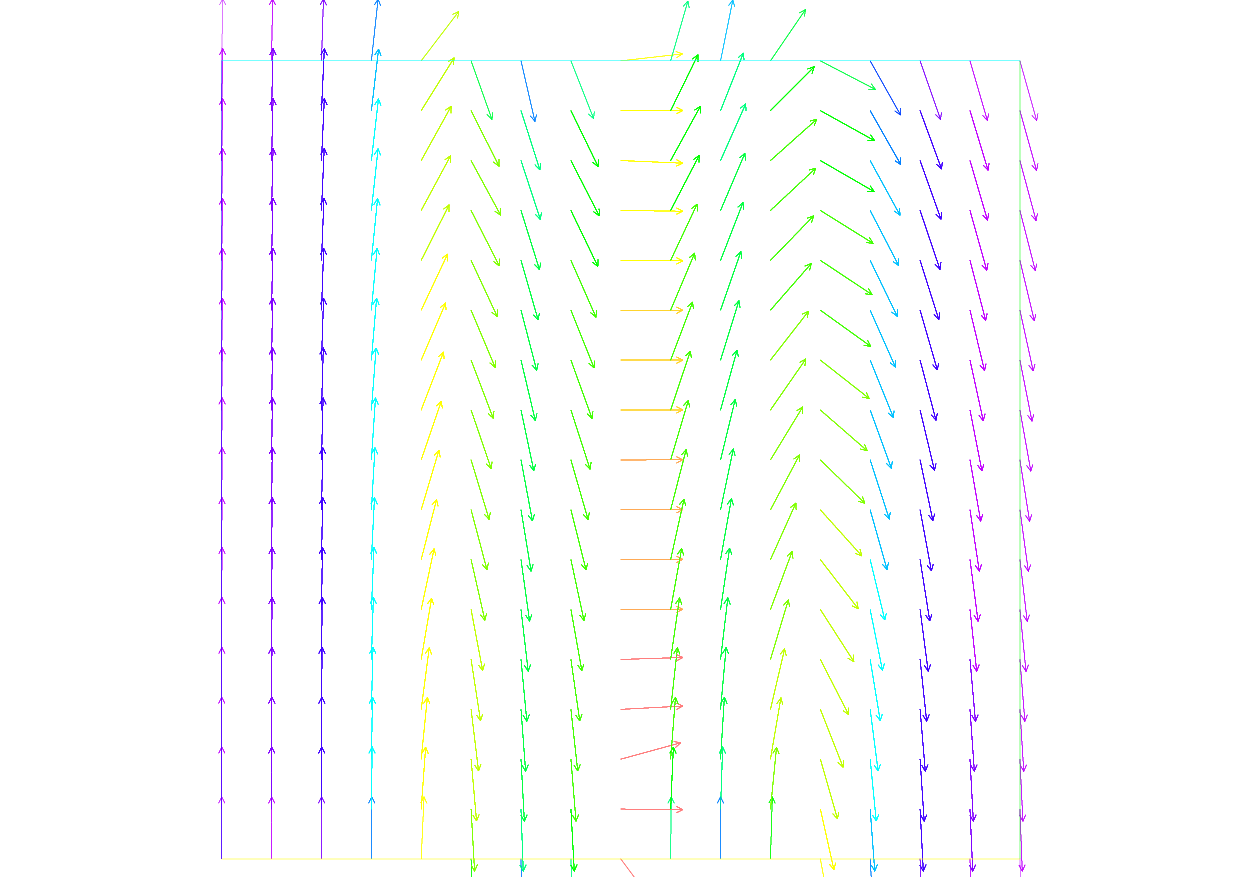}
		\caption{$t=4\times 10^{-4}$}
	\end{subfigure}
	\begin{subfigure}[b]{0.3\textwidth}
		\label{fig:time2 3}
		\centering
		\includegraphics[width=\textwidth]{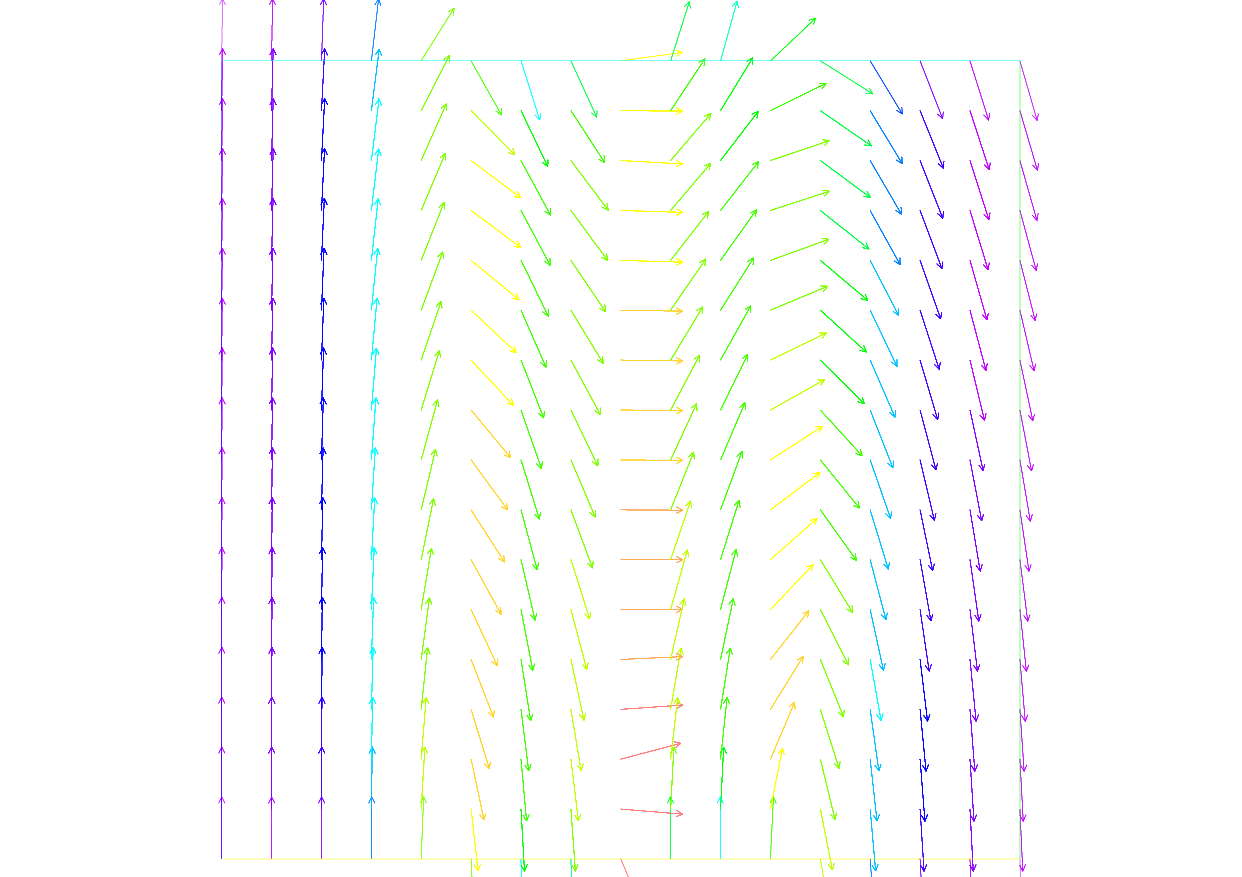}
		\caption{$t=6\times 10^{-4}$}
	\end{subfigure}
	\begin{subfigure}[b]{0.3\textwidth}
		\label{fig:time2 4}
		\centering
		\includegraphics[width=\textwidth]{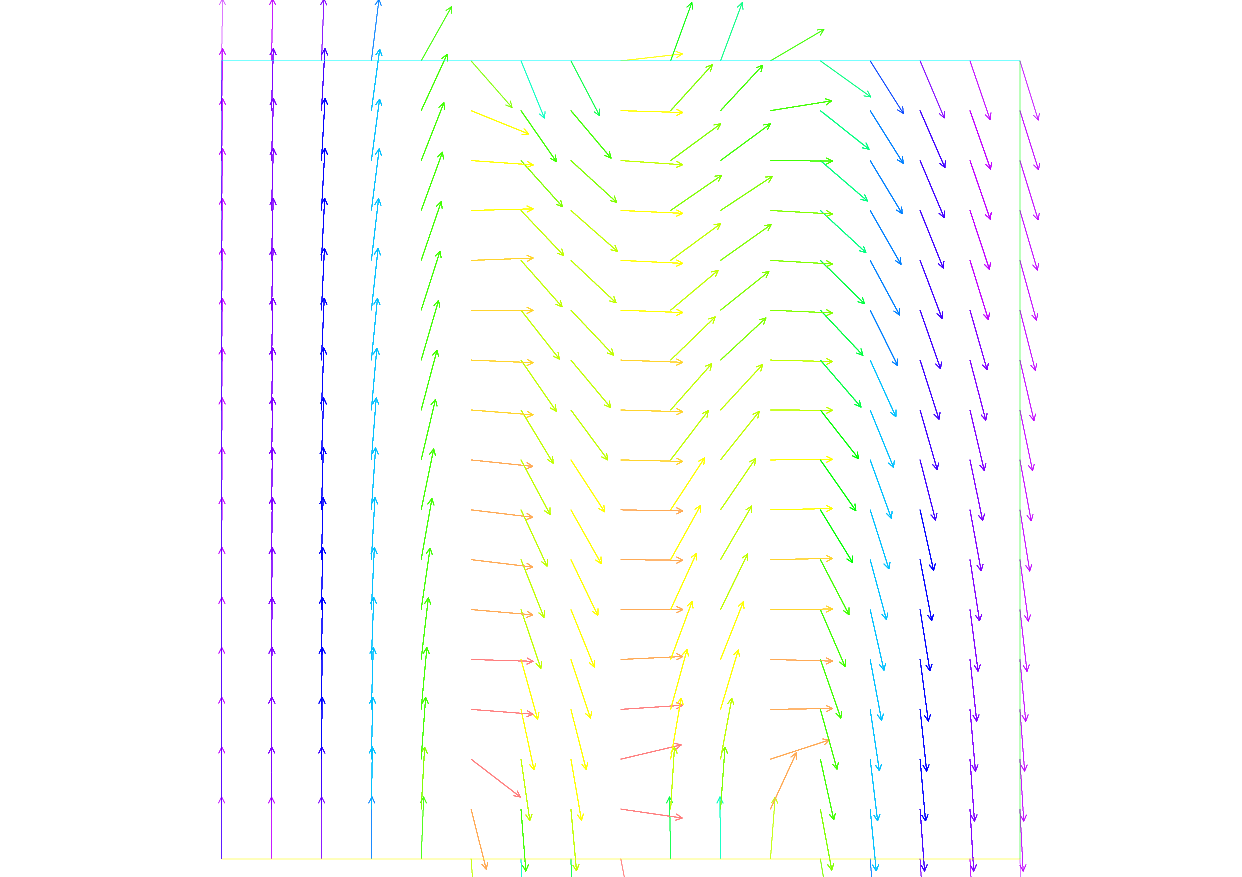}
		\caption{$t=8\times 10^{-4}$}
	\end{subfigure}
	\begin{subfigure}[b]{0.3\textwidth}
		\label{fig:time2 5}
		\centering
		\includegraphics[width=\textwidth]{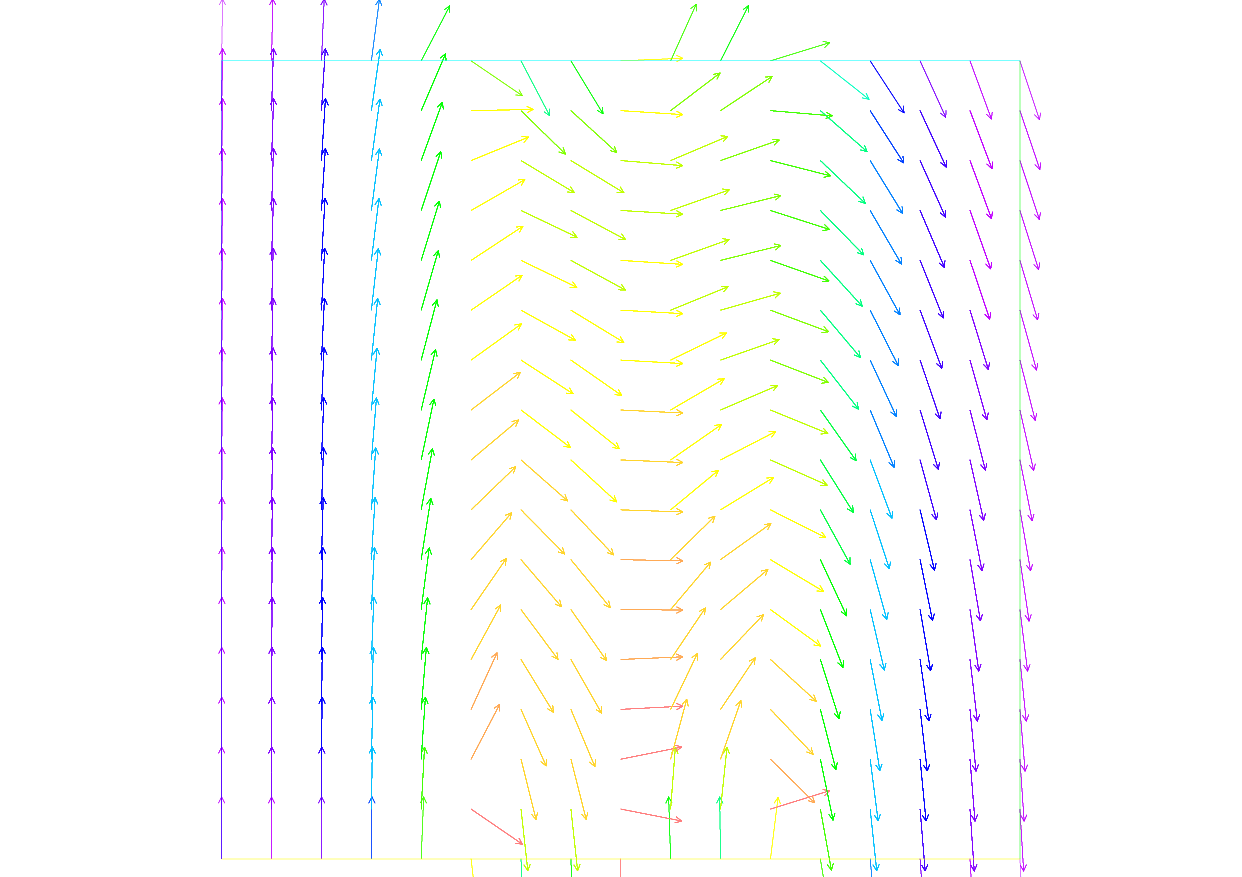}
		\caption{$t=1\times 10^{-3}$}
	\end{subfigure}
	\caption{Snapshots of the magnetic spin field $\bff{u}$ (projected onto $\bb{R}^2$) at given times in experiment 2. The colours indicate relative magnitude of the vectors.}
	\label{fig:snapshots field exp2}
\end{figure}

\subsection{Experiment 3 (Cahn--Hilliard with source term)}
Let $u:\Omega \to \bb{R}$ be scalar-valued. The coefficients in \eqref{equ:llbar} are taken to be $\beta_1=0.1, \beta_2=0.1, \beta_3=4.0, \beta_4=0.0$, $\beta_5=1.0$, and $\beta_6=0.02$. The initial data $u_0$ is given by
\begin{align}\label{equ:initial data 3}
	u_0(x,y)
	=
	2 \sin (2\pi x) \sin(2\pi y),
\end{align}
where $(x,y)\in [0,2]\times[0,2]$. The time step size $k= 1\times 10^{-9}$.

\begin{figure}[!hbt]
	\begin{center}
		\begin{tikzpicture}
			\begin{axis}[
				title=Plot of $\|\bff{e}_h\|$ against $1/h$,
				height=0.45\textwidth,
				width=0.45\textwidth,
				xlabel= $1/h$,
				ylabel= $\|\bff{e}_h\|$,
				xmode=log,
				ymode=log,
				legend pos=outer north east,
				legend cell align=left,
				]
				\addplot+[mark=*,red] coordinates {(4,948)(8,237)(16,25.7)(32,1.04)(64,0.041)(128,0.002)};
				\addplot+[mark=x,blue] coordinates {(4,5070)(8,2585)(16,591)(32,19.8)(64,0.65)(128,0.04)};
				\addplot+[mark=square,green] coordinates {(4,115000)(8,114500)(16,46300)(32,3500)(64,260)(128,30.8)};
				\addplot+[dashed,no marks,red,domain=30:128]{175000/x^4};
				\addplot+[dashed,no marks,blue,domain=30:128]{69000/x^3};
				\addplot+[dashed,no marks,green,domain=30:128]{300000/x^2};
				\legend{$L^\infty(\bb{L}^2)$-norm of $\bff{e}_h$,$L^\infty(\bb{H}^1)$-norm of $\bff{e}_h$,$L^\infty(\bb{H}^2)$-norm of $\bff{e}_h$,order 4 reference line,order 3 reference line,order 2 reference line}
			\end{axis}
		\end{tikzpicture}
	\end{center}
	\caption{Order of convergence for experiment 3.}
	\label{fig:order exp3}
\end{figure}

The plot of $\norm{\bff{e}_h}{}$ against $1/h$ is shown in Figure \ref{fig:order exp3}.
Furthermore, snapshots of the values of $\bff{u}$ at some indicated times (with initial data \eqref{equ:initial data 3}) are shown in Figure \ref{fig:snapshots field exp3}.

\begin{figure}[!hbt]
	\centering
	\begin{subfigure}[b]{0.49\textwidth}
		\label{fig:time3 0}
		\centering
		\includegraphics[width=\textwidth]{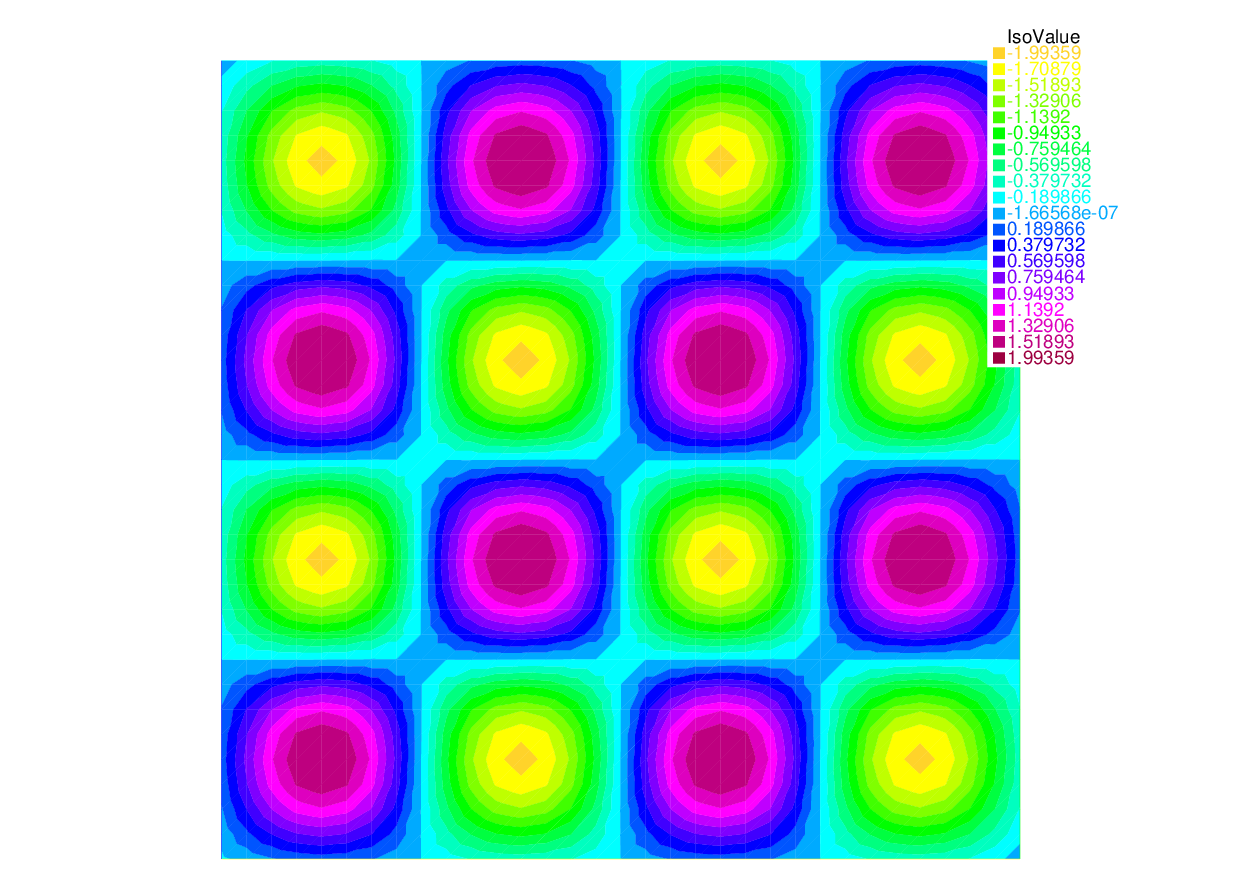}
		\caption{$t=1\times 10^{-3}$}
	\end{subfigure}
	\begin{subfigure}[b]{0.49\textwidth}
		\label{fig:time3 1}
		\centering
		\includegraphics[width=\textwidth]{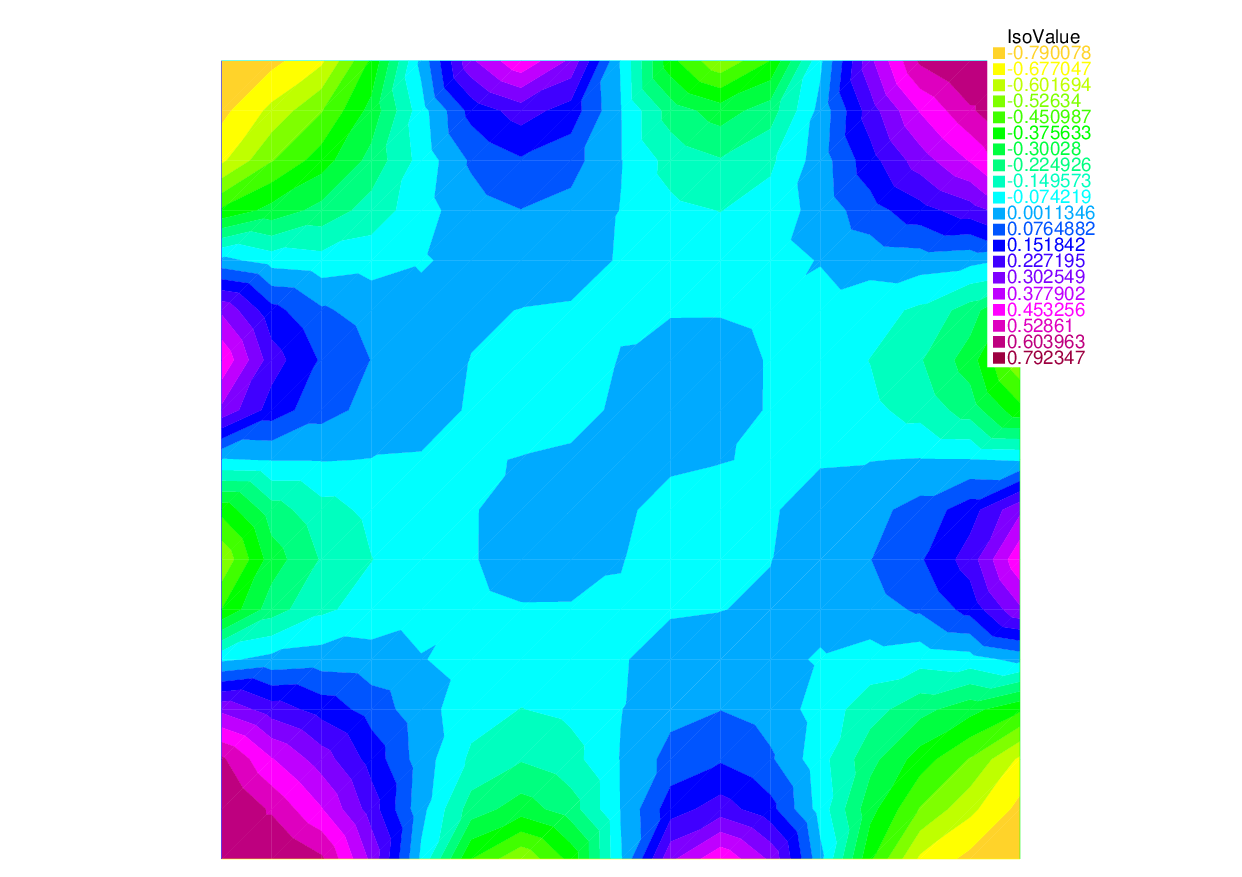}
		\caption{$t=2\times 10^{-2}$}
	\end{subfigure}
	\begin{subfigure}[b]{0.49\textwidth}
		\label{fig:time3 2}
		\centering
		\includegraphics[width=\textwidth]{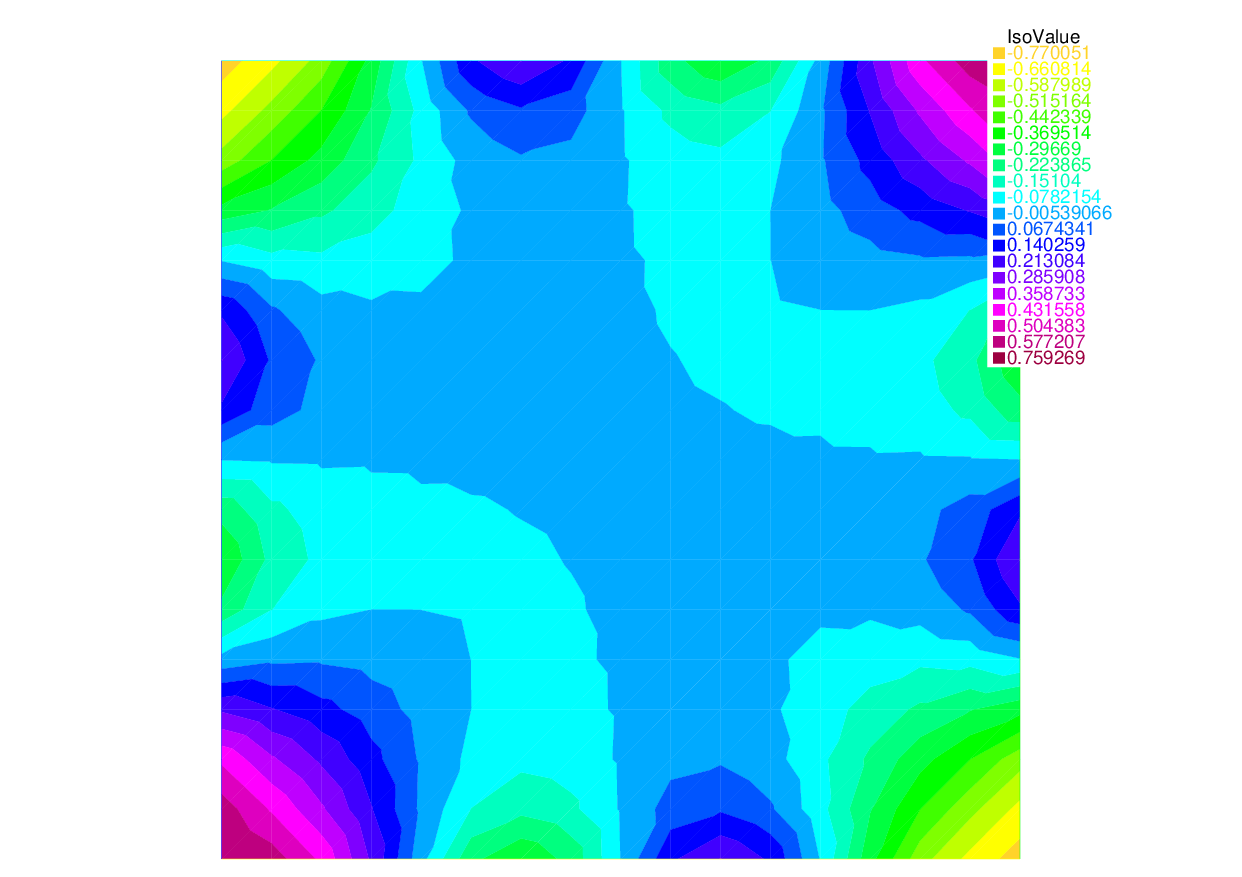}
		\caption{$t=4\times 10^{-2}$}
	\end{subfigure}
	\begin{subfigure}[b]{0.49\textwidth}
		\label{fig:time3 3}
		\centering
		\includegraphics[width=\textwidth]{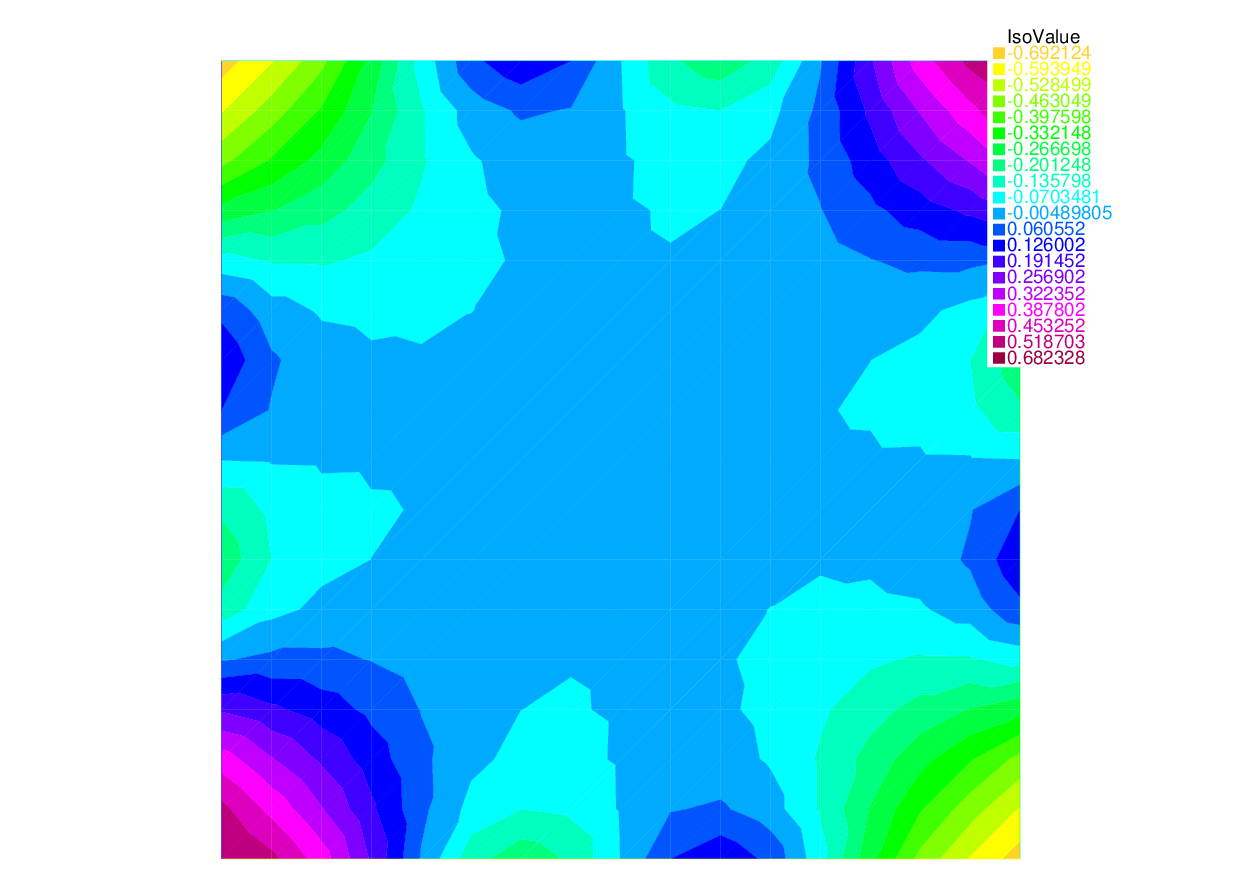}
		\caption{$t=6\times 10^{-2}$}
	\end{subfigure}
	\begin{subfigure}[b]{0.49\textwidth}
		\label{fig:time3 4}
		\centering
		\includegraphics[width=\textwidth]{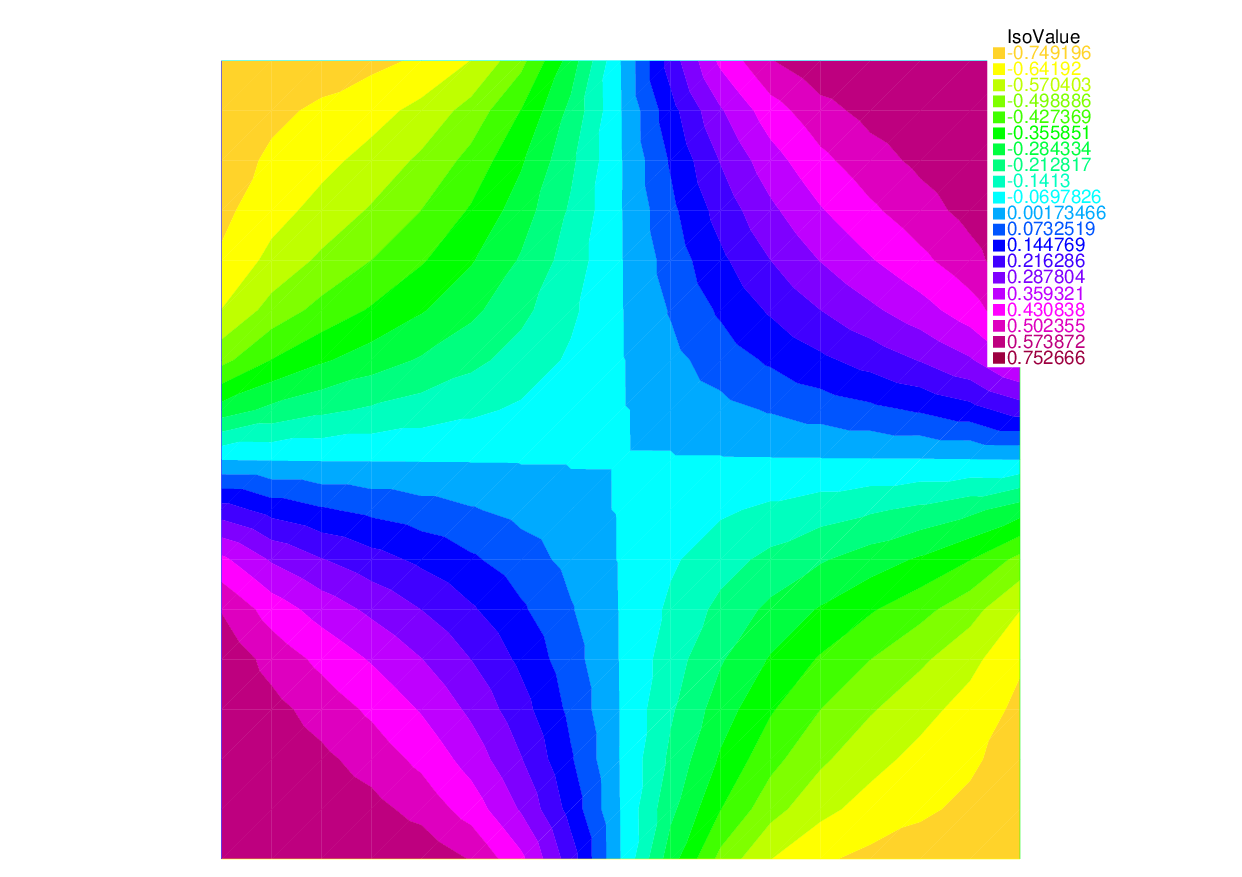}
		\caption{$t=1$}
	\end{subfigure}
	\begin{subfigure}[b]{0.49\textwidth}
		\label{fig:time3 5}
		\centering
		\includegraphics[width=\textwidth]{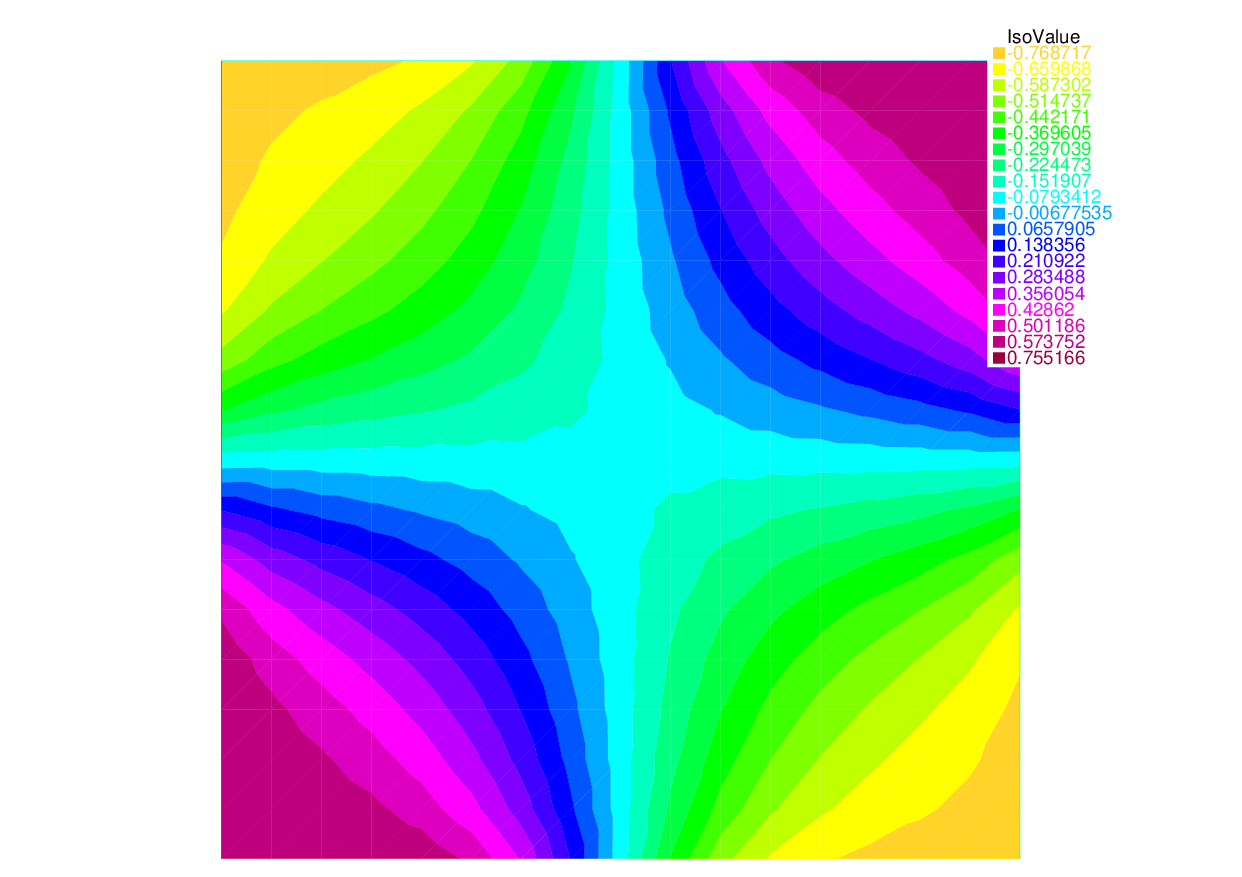}
		\caption{$t=2$}
	\end{subfigure}
	\caption{Snapshots of the function $u$ at given times in experiment 3.}
	\label{fig:snapshots field exp3}
\end{figure}

\section{A Regularity Lemma} \label{sec:lem aux H4}

The following lemma was used in the proof of error estimate for the semi-discrete scheme in Section \ref{sec:semidiscrete}.

\begin{lemma}\label{lem:aux H4}
	Let $\Omega$ be a domain such that the elliptic regularity properties~\eqref{equ:equivnorm-h2}, \eqref{equ:equivnorm-h3}, \eqref{equ:equivnorm-h4} hold, and let $T>0$. Assume that the exact solution $\bff{u}$ of \eqref{equ:llbar}
	belongs to $L^\infty(0,T; \bb{H}_{\bff{n}}^2)$. For any
	$\bff{\varphi} \in \bb{L}^2$ and for each $t\in [0,T]$, there exists
	$\bff{\psi}(t) \in \bb{H}^4 \cap \bb{H}_{\bff{n}}^2$ such that
\begin{align}
	\label{equ:aux H4 part 2}
	\cal{A}(\bff{u}(t); \bff{\xi}, \bff{\psi}(t)) 
	= 
	\inpro{\bff{\varphi}}{\bff{\xi}}_{\bb{L}^2},
	\quad\forall\bff{\xi} \in \bb{H}^2. 
\end{align}
	Moreover,
\begin{align}
	\label{equ:aux H4 est part 2}
	\norm{\bff{\psi}(t)}{\bb{H}^4} 
	&\lesssim 
	\norm{\bff{\varphi}}{\bb{L}^2},
\end{align}
where the constant depends on $T$ and $\norm{\bff{u}}{L^\infty(\bb{H}^2)}$ (but is independent of $t$).
\end{lemma}
\begin{proof}
We use the Faedo--Galerkin method.
Let $\{\bff{e}_i\}_{i=1}^\infty$ denote an orthonormal basis of $\bb{L}^2$
consisting of eigenfunctions for $-\Delta$ such that
\begin{align*}
	-\Delta \bff{e}_i =\lambda_i \bff{e}_i \ \text{ in $\Omega$}
	\quad\text{and}\quad
	\frac{\partial \bff{e}_i}{\partial \bff{n}}= \bff{0} \ \text{ on } \partial \Omega,
\end{align*}
where $\lambda_i>0$ are the eigenvalues of $-\Delta$, associated with
$\bff{e}_i$. By elliptic regularity, for $j=2,3$,
\begin{align*}
	\Delta^{j} \bff{e}_i 
	= \lambda_i^{j} \bff{e}_i \ \text{ in $\Omega$}
	\quad\text{and}\quad
	\frac{\partial(\Delta^{j-1} \bff{e}_i)}{\partial \bff{n}} 
	=\bff{0} \ \text{ on } \partial \Omega.
\end{align*}
Let $\bb{S}_n:= \text{span}\{\bff{e}_1,\ldots,\bff{e}_n\}$ and $\Pi_n: \bb{L}^2\to
\bb{S}_n$ be the orthogonal projection defined by
\begin{align*}
	\inpro{ \Pi_n \bff{v}}{\bff{\phi}}_{\bb{L}^2}
	= \inpro{ \bff{v}}{\bff{\phi}
	}_{\bb{L}^2} \quad \text{for all } 
	\bff{\phi}\in \bb{S}_n \text{ and all }
	\bff{v}\in\bb{L}^2.
\end{align*}
Note that $\Pi_n$ is self-adjoint and satisfies
\begin{equation}\label{equ:Pin}
	\norm{\Pi_n \bff{v}}{\bb{L}^2} \leq \norm{\bff{v}}{\bb{L}^2} \quad \text{for all }
	\bff{v}\in \bb{L}^2.
\end{equation}

Given $\bff{\varphi}\in \bb{L}^2$, for each $t\in[0,T]$,
the Faedo--Galerkin method seeks to approximate the solution to \eqref{equ:aux
H4 part 2} by $\bff{\psi}_n(t) \in \bb{S}_n$ satisfying the equation
\begin{align}\label{equ:faedo aux H4}
	\cal{A}(\bff{u}(t); \bff{\psi}_n(t), \bff{\xi}) = \inpro{\bff{\varphi}}{\bff{\xi}}_{\bb{L}^2}
	\; \text{ for all } 
	\bff{\xi} \in \bb{S}_n.
\end{align}
The existence of such $\bff{\psi}_n(t)$ follows by a standard argument and the
Cauchy--Lipschitz theorem. We will need the following bound on
$\bff{\psi}_n(t)$, the proof of which is deferred till the end of this proof:
\begin{align}\label{equ:del sq psi}
\norm{\bff{\psi}_n(t)}{\bb{H}^4} 
\lesssim 
\norm{\bff{\varphi}}{\bb{L}^2}.
\end{align}
It follows from \eqref{equ:del sq psi} and \eqref{equ:H2 psi}, and the Banach--Alaoglu theorem that for each $t\in [0,T]$, there is a subsequence, which is still denoted by $\{\bff{\psi}_n(t)\}$, such that
\begin{align*}
	\bff{\psi}_n(t) &\rightharpoonup \bff{\psi}(t) \quad \text{weakly in } \bb{H}^4.
\end{align*}
Since $\bb{H}^4$ is compactly embedded into $\bb{H}^3$, a further subsequence then satisfies
\begin{align*}
	\bff{\psi}_n(t) \to \bff{\psi}(t) \quad \text{strongly in } \bb{H}^{3}.
\end{align*}
Note that Sobolev embedding gives $\bb{H}^{3}(\Omega)\subset
\bff{C}^1(\overline{\Omega})$, and thus in particular
$\partial\bff{\psi}/\partial \bff{n}=\bff{0}$ on $\partial \Omega$. Standard
argument then shows $\bff{\psi}(t)$ satisfies the equation \eqref{equ:aux H4
part 2}, while \eqref{equ:aux H4 est part 2} follows from \eqref{equ:del sq
psi}. 

We now prove~\eqref{equ:del sq psi}.
	Firstly, by taking $\bff{\xi}= \bff{\psi}_n$ in \eqref{equ:faedo aux
	H4}, and using \eqref{equ:coercive}, we have
\begin{align*}
	\mu \norm{\bff{\psi}_n}{\bb{H}^2}^2 
	\leq
	\cal{A}(\bff{u}; \bff{\psi}_n, \bff{\psi}_n)
	=
	\inpro{\bff{\varphi}}{\bff{\psi}_n}_{\bb{L}^2}
	\leq
	\norm{\bff{\varphi}}{\bb{L}^2}
	\norm{\bff{\psi}_n}{\bb{L}^2},
\end{align*}
which implies
\begin{equation}\label{equ:H2 psi}
	\norm{\bff{\psi}_n}{\bb{H}^2} 
	\lesssim
	\norm{\bff{\varphi}}{\bb{L}^2}.
\end{equation}
Next, by taking $\bff{\xi}= \Delta^2 \bff{\psi}_n$ in \eqref{equ:faedo aux
H4} and integrating by parts the terms involving $\alpha, \beta_1,\beta_2$, we have (after rearranging the terms)
\begin{align}\label{equ:H4 psi}	
	&\alpha \norm{\Delta \bff{\psi}_n}{\bb{L}^2}^2
	+ \beta_2 \norm{\Delta^2 \bff{\psi}_n}{\bb{L}^2}^2
	\nonumber \\
	&=
	\beta_1 \inpro{\Delta \bff{\psi}_n}{\Delta^2 \bff{\psi}_n}_{\bb{L}^2}
	-
	\cal{B}(\bff{u},\bff{u}; \bff{\psi}_n, \Delta^2 \bff{\psi}_n)
	-
	\cal{C}(\bff{u};\bff{\psi}_n, \Delta^2 \bff{\psi}_n)
	+
	\inpro{\bff{\varphi}}{\Delta^2 \bff{\psi}_n}_{\bb{L}^2}
	\nonumber
	\\
	\nonumber
	&\le
	|\beta_1| 
	\norm{\Delta\bff{\psi}_n}{\bb{L}^2}
	\norm{\Delta^2\bff{\psi}_n}{\bb{L}^2}
	+
	\big| \cal{B}(\bff{u},\bff{u}; \bff{\psi}_n, \Delta^2 \bff{\psi}_n) \big|
	+
	\big| \cal{C}(\bff{u};\bff{\psi}_n, \Delta^2 \bff{\psi}_n) \big|
	+
	\norm{\bff{\varphi}}{\bb{L}^2}
	\norm{\Delta^2\bff{\psi}_n}{\bb{L}^2}
	\\
	&\lesssim
	\norm{\Delta\bff{\psi}_n}{\bb{L}^2}
	\norm{\Delta^2\bff{\psi}_n}{\bb{L}^2}
	+
	\left(1+\norm{\bff{u}}{\bb{H}^2}^2\right)
	\norm{\bff{\psi_n}}{\bb{H}^2}
	\norm{\Delta^2 \bff{\psi}_n}{\bb{L}^2}
	+
	\norm{\bff{\varphi}}{\bb{L}^2}
	\norm{\Delta^2\bff{\psi}_n}{\bb{L}^2}.
\end{align}
We deduce from~\eqref{equ:H4 psi} that
\begin{align*}
	\norm{\Delta^2 \bff{\psi}_n(t)}{\bb{L}^2}^2 
	&\lesssim
	\norm{\bff{\psi}_n(t)}{\bb{H}^2} 
	\norm{\Delta^2 \bff{\psi}_n(t)}{\bb{L}^2}
	+
	\norm{\bff{\varphi}}{\bb{L}^2}
	\norm{\Delta^2 \bff{\psi}_n(t)}{\bb{L}^2},
\end{align*}
which implies
\[
	\norm{\Delta^2 \bff{\psi}_n(t)}{\bb{L}^2}
	\lesssim
	\norm{\bff{\varphi}}{\bb{L}^2}.
\]
With assumption \eqref{equ:equivnorm-h4}, this and \eqref{equ:H2 psi}
imply the required estimate, completing the proof of the lemma.
\end{proof}

\section*{Acknowledgements}
The first author is supported by the Australian Government Research Training
Program (RTP) Scholarship awarded at the University of New South Wales, Sydney. Both
authors are partially supported by the Australian Research Council under grant
number DP200101866 and DP220101811.

\section{Appendix}

We collect here a few results which are used in this paper.

\begin{theorem}[Bihari--Gronwall's inequality \cite{Bih56, BoyFab13}]\label{the:bihari}
Let $f$ be a non-decreasing continuous function which is non-negative on $[0, \infty)$ such that $\int_1^\infty 1/f(x)\,\dx < \infty$. Let $F$ be the anti-derivative of $-1/f$ which vanishes at $\infty$. Let $y$ be a non-negative continuous function and let $g$ be a locally integrable non-negative function $[0,\infty)$. Suppose that there exists a $y_0>0$ such that for all $t \geq 0$,
\begin{align*}
	y(t) \leq y_0 + \int_0^t g(s) \,\ds + \int_0^t f(y(s)) \,\ds.
\end{align*}
Then for any $T < T^*$,
\begin{align*}
	\sup_{t\leq T} y(t) \leq F^{-1} \left(F \left(y_0 + \int_0^T g(s) \,\ds \right) - T \right),
\end{align*}
where $T^*$ is the unique solution of the equation
\begin{align*}
	T^* = F \left(y_0 + \int_0^{T^*} g(s) \,\ds \right).
\end{align*}
\end{theorem}

\begin{theorem}[Discrete form of Bihari--Gronwall's inequality \cite{Qin16}]\label{the:disc bihari}
	Let $\varphi$ be a strictly increasing non-negative function on $\bb{R}_+$ with $\varphi(+\infty)=+\infty$ and $\psi$ be non-decreasing non-negative function on $\bb{R}_+$. Let $c\geq 0$ be a real constant, and let $v_n,f_n$ be non-negative sequences. If the inequality
	\[
	\varphi(v_n)\leq c+ \sum_{s=0}^{n-1} f_s\, \psi(v_s)
	\]
	holds, then for all $n\leq M$,
	\[
	v_n \leq \varphi^{-1} \left[ G^{-1} \left(G(c) + \sum_{s=0}^{n-1} f_s \right) \right],
	\]
	where $G^{-1}$ and $\varphi^{-1}$ are the inverse functions of $G$ and $\varphi$, respectively, and
	\[
	G(z):= \int_{z_0}^z \frac{\ds}{\psi(\varphi^{-1}(s))}, \quad z\geq z_0>0.
	\]
	Here, $M$ is chosen so that if $n\leq M$, then
	\[
	G(c)+ \sum_{s=0}^{n-1} f_s \in \text{dom}(G^{-1}).
	\]
\end{theorem}


\newcommand{\noopsort}[1]{}\def\cprime{$'$}
\def\soft#1{\leavevmode\setbox0=\hbox{h}\dimen7=\ht0\advance \dimen7
	by-1ex\relax\if t#1\relax\rlap{\raise.6\dimen7
		\hbox{\kern.3ex\char'47}}#1\relax\else\if T#1\relax
	\rlap{\raise.5\dimen7\hbox{\kern1.3ex\char'47}}#1\relax \else\if
	d#1\relax\rlap{\raise.5\dimen7\hbox{\kern.9ex \char'47}}#1\relax\else\if
	D#1\relax\rlap{\raise.5\dimen7 \hbox{\kern1.4ex\char'47}}#1\relax\else\if
	l#1\relax \rlap{\raise.5\dimen7\hbox{\kern.4ex\char'47}}#1\relax \else\if
	L#1\relax\rlap{\raise.5\dimen7\hbox{\kern.7ex
			\char'47}}#1\relax\else\message{accent \string\soft \space #1 not
		defined!}#1\relax\fi\fi\fi\fi\fi\fi}

\end{document}